\documentclass[reqno]{amsart}
\usepackage{amssymb, amsmath, amsthm, mathrsfs, ascmac, color}
\usepackage[pdftex]{graphicx} 
\usepackage{comment}
\usepackage[
top=2.5cm, bottom=2.5cm, left=2.5cm, right=2.5cm
]{geometry}
\usepackage{harvard}
\usepackage[foot]{amsaddr}
\newtheoremstyle{mystyle}
  {}
  {}
  {\normalfont}
  { }
  {\bfseries}
  {}
  {10pt}
  { }
\theoremstyle{mystyle}
\newtheorem{thm}{Theorem}
\newtheorem{prop}{Proposition}
\newtheorem{lem}{Lemma}
\newtheorem{cor}{Corollary}
\newtheorem{rmk}{Remark}
\makeatletter

\@addtoreset{equation}{section}
\makeatother

\newcommand{\dd}{\mathrm d}

\newcommand{\EE}{\mathbb E}

\newcommand{\GG}{\mathscr G_{i-1}^n}

\newcommand{\DeX}{\Delta X_i}
\newcommand{\Xt}{X_{t_i}}
\newcommand{\Xs}{X_{t_{i-1}}}
\newcommand{\TT}{\mathsf T}

\newcommand{\Hi}{\mathcal H_{1,n}}
\newcommand{\Hd}{\mathcal H_{2,n}}

\newcommand{\tr}{\mathrm{tr}}
\newcommand{\diag}{\mathrm{diag}}
\newcommand{\Ro}{R_{i-1}(1,\theta)}
\newcommand{\Ri}{R_{i-1}(h,\theta)}
\newcommand{\Rd}{R_{i-1}(h^2,\theta)}

\newcommand{\lto}{\longrightarrow}
\newcommand{\pto}{\stackrel{p}{\longrightarrow}}
\newcommand{\dto}{\stackrel{d}{\longrightarrow}}
\newcommand{\wto}{\stackrel{w}{\longrightarrow}}
\newcommand{\ato}{\stackrel{\mathrm{a.s.}}{\longrightarrow}}

\newcommand{\Dea}{\vartheta_\alpha}
\newcommand{\Deb}{\vartheta_\beta}
\allowdisplaybreaks

\begin{document}
\title[
Estimation for change point of ergodic diffusion processes]
{
Estimation for change point
of discretely observed ergodic diffusion processes 
}
\author[Y. Tonaki]{Yozo Tonaki}
\author[Y. Kaino]{Yusuke Kaino}
\address[Y. Tonaki]
{Graduate School of Engineering Science, Osaka University, 1-3, 
Machikaneyama, Toyonaka, Osaka, 560-8531, Japan}
\address[Y. Kaino]{
Graduate School of Engineering Science, Osaka University, 1-3, 
Machikaneyama, Toyonaka, Osaka, 560-8531, Japan}
\author[M. Uchida]{Masayuki Uchida}
\address[M. Uchida]{
Graduate School of Engineering Science, 
and Center for Mathematical Modeling and Date Science,
Osaka University, 1-3, Machikaneyama, Toyonaka, Osaka, 560-8531, Japan}

\keywords{
adaptive tests, change point estimation, diffusion processes, 
discrete observations, test for parameter change}

\maketitle

\begin{abstract}
We treat the change point problem 
in ergodic diffusion processes from discrete observations.
Tonaki et al. (2020)
proposed adaptive tests 
for detecting changes in the diffusion and drift parameters 
in ergodic diffusion models.
When any changes are detected by this method,
the next question to be considered is where the change point is.
Therefore, we propose the method to estimate the change point of the parameter 
for two cases: 
the case where there is a change in the diffusion parameter, 
and 
the case where there is no change in the diffusion parameter
but a change in the drift parameter.
Furthermore, we present rates of convergence and distributional results 
of the change point estimators. 
Some examples and simulation results are also given.
\end{abstract}

\section{Introduction}\label{sec1}
We consider a $d$-dimensional diffusion process $\{X_t\}_{t\ge0}$ 
satisfying the stochastic differential equation:
\begin{align}\label{sde}
\begin{cases}
\dd X_t=b(X_t,\beta)\dd t+a(X_t,\alpha)\dd W_t,\\
X_0=x_0,
\end{cases}
\end{align}
where parameter space
$\Theta=\Theta_A\times\Theta_B$, which is a  
compact convex subset of 
$\mathbb R^p\times\mathbb R^q$, 
$\theta=(\alpha,\beta)\in\Theta$ is an unknown parameter and
$\{W_t\}_{t\ge0}$ 
is a $d$-dimensional standard Wiener process.
The diffusion coefficient 
$a:\mathbb R^d\times\Theta_A\lto\mathbb R^d\otimes\mathbb R^d$ and
the drift coefficient $b:\mathbb R^d\times\Theta_B\lto\mathbb R^d$
are known except for the parameter $\theta$, 
and the true parameter $\theta^*=(\alpha^*,\beta^*)$ belongs to $\mathrm{Int\,}\Theta$.
We assume that the solution of \eqref{sde} exists, and 
$P_\theta$ and $\EE_\theta$ denote the law of the 
solution and the expectation with respect to $P_\theta$, respectively. 
Let $\{\Xt\}_{i=0}^n$ be discrete observations, 
where $t_i=t_i^n=ih_n$, 
and $\{h_n\}$ is a positive sequence with $h=h_n\lto0$, $T=nh\lto\infty$ and 
$nh^2\lto0$ as $n\to\infty$.


We deal with the change point problem of parameters in diffusion process models.
The change point problem consists of two parts:
detection of the parameter change and estimation of the change point.
For example, suppose that we obtain the paths shown in Figure \ref{fig_intro1}.
Given such data, 
the first thing we would consider 
is whether any changes occur in that path.
The paths in Figure \ref{fig_intro1} clearly shows that some change occurs, 
but the paths in Figure \ref{fig_intro2} may be difficult to find a change.
The first step of the change point problem 
is to investigate whether there is a change in the data or not, 
which is ``detection of the parameter change".
Suppose that we can detect a change in the paths 
in Figure \ref{fig_intro1} or \ref{fig_intro2}.
In Figure \ref{fig_intro1}, 
we can visually see that the change occurs at $t=60$ 
in the left figure and at $t=40$ in the right figure, 
but the paths in Figure \ref{fig_intro2} are not even visually recognizable, 
so we have no idea where the change occurs.
The goal of the change point problem
is to investigate the change point when we see that there is a change in the data, 
which is``estimation of the change point". 
\begin{figure}
\centering
\includegraphics[keepaspectratio,width=140mm]{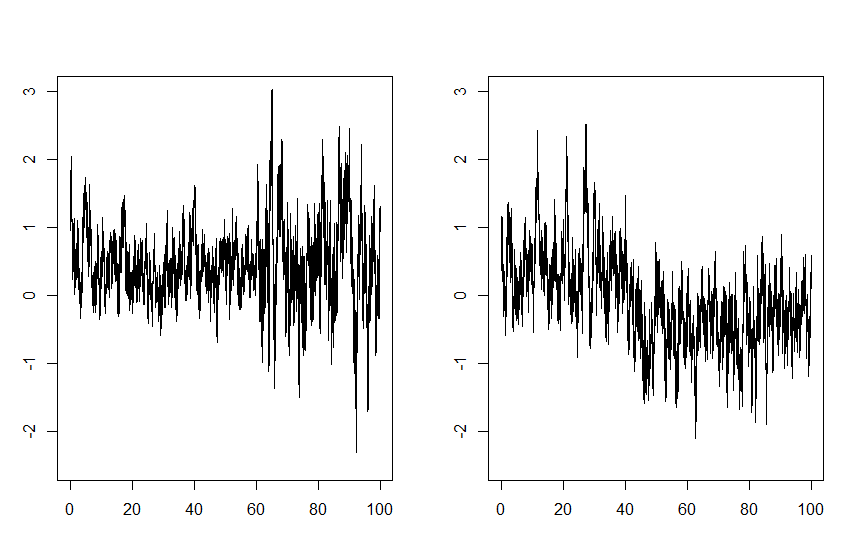}
\caption{
Sample paths of the 
hyperbolic diffusion model 
$\dd X_t=\Bigl(\beta+\frac{\gamma X_t}{\sqrt{1+X_t^2}}\Bigr)\dd t+\alpha\dd W_t$
whose parameter changes form 
$(\alpha,\beta,\gamma)=(0.75,1,3)$ to $(1.5,1,3)$ at $t=60$ (left),
and changes from
$(\alpha,\beta,\gamma)=(1,1,3)$ to $(1,-1,3)$ at $t=40$ (right).
}
\label{fig_intro1}
\includegraphics[keepaspectratio,width=140mm]{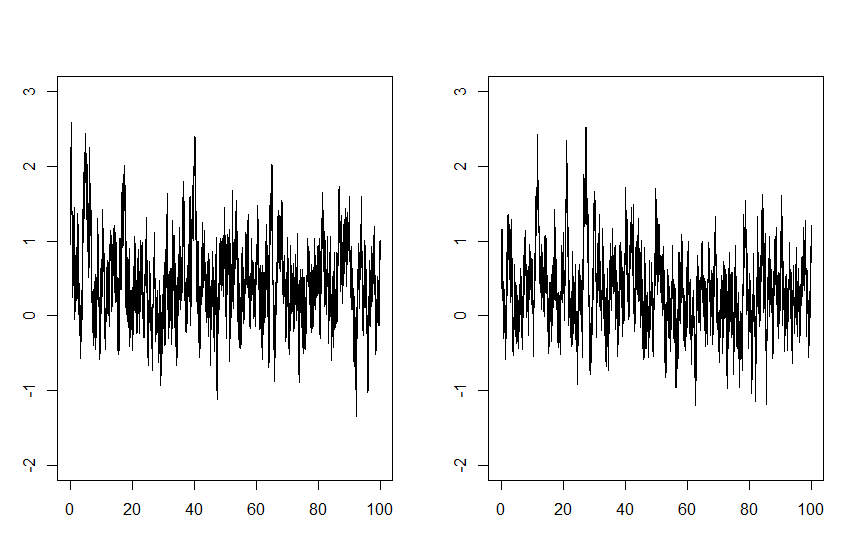}
\caption{
Sample paths of the 
hyperbolic diffusion model 
$\dd X_t=\Bigl(\beta+\frac{\gamma X_t}{\sqrt{1+X_t^2}}\Bigr)\dd t+\alpha\dd W_t$
whose parameter changes form 
$(\alpha,\beta,\gamma)=(1,1,3)$ to $(1.05,1,3)$ at $t=30$ (left),
and changes from
$(\alpha,\beta,\gamma)=(1,1,3)$ to $(1,0.5,3)$ at $t=50$ (right).
}
\label{fig_intro2}
\end{figure}

The change point problem for diffusion processes from discrete observations
has been studied by many researchers. 
For non-ergodic diffusion processes, see 
De Gregorio and Iacus (2008)
and
Iacus and Yoshida (2012).
For ergodic diffusion processes, we can refer to 
Song and Lee (2009), 
Lee (2011),
Negri and Nishiyama (2017),
Song (2020),
and
Tonaki et al. (2020).
De Gregorio and Iacus (2008)
considered the detection of the parameter change in diffusion parameter 
and the estimation of the change point, 
and Iacus and Yoshida (2012) 
studied the estimating problem of the change point 
in diffusion parameter.
In contrast, 
Song and Lee (2009), 
Lee (2011)
and Song (2020)
considered the detection of the parameter change in diffusion parameter, 
and 
Negri and Nishiyama (2017)
and 
Tonaki et al. (2020)
considered the detection of the parameter change in diffusion and drift parameters.
Specifically, 
Negri and Nishiyama (2017)
proposed simultaneous test for changes in diffusion and drift parameters, 
and 
Tonaki et al. (2020)
proposed adaptive tests for those changes.

When we consider the change point problem of the diffusion processes, 
we first need to investigate whether there are any change points in the given data.
Negri and Nishiyama (2017) considered the following hypothesis testing problem:
\begin{center}
$H_0^\theta:\theta=(\alpha, \beta)$ does not change over $0\le t\le T$\quad
v.s.\quad
$H_1^\theta:$ not $H_0^\theta$. 
\end{center}
They considered the simultaneous test for changes in diffusion and drift parameters, 
but even if any changes are detected, 
we can not determine which parameter changed.
In order to solve this problem, 
Tonaki et al. (2020)
considered the adaptive tests for changes in diffusion and drift parameters.
Specifically, consider the following steps.
First, consider the following hypothesis testing problem.
\begin{center}
$H_0^\alpha:\alpha$ does not change over $0\le t\le T$\quad
v.s.\quad
$H_1^\alpha:$ not $H_0^\alpha$. 
\end{center}
If $H_0^\alpha$ is not rejected, 
then consider the following hypothesis testing problem.
\begin{center}
$H_0^\beta:\beta$ does not change over $0\le t\le T$\quad
v.s.\quad
$H_1^\beta:$ not $H_0^\beta$. 
\end{center}
When we find that there is a change,
the next task is to estimate the change point.
Since the study of the estimation of the change point 
for ergodic diffusion processes is still in its infancy,
we consider the estimation of the change point for ergodic diffusion processes.
With the adaptive detection method,
when $H_0^\alpha$ is rejected, 
we can consider the estimation of the change point of the diffusion parameter $\alpha$, 
and when $H_0^\alpha$ is not rejected
and $H_0^\beta$ is rejected, 
we can consider the estimation of the change point of the drift parameter $\beta$, 
which brings us one step closer to our goal.


This paper is organized as follows.
In Section \ref{sec2}, we state the main results.
We propose change point estimators for diffusion and drift parameters
and present rates of convergence and distributional results 
of change point estimators.
Section \ref{sec3} 
discusses the nuisance parameters of the proposed estimators.
Section \ref{sec4}
presents the powers of tests proposed by 
Tonaki et al. (2020).
In Sections \ref{sec5} and \ref{sec6},
we give some examples and simulation studies.
Finally, we provide the proofs in Section \ref{sec7}.

\section{Main Results}\label{sec2}
We consider the following two situations:

\noindent\textbf{Situation I. }
$\alpha^*$ changes in $0\le t\le T$, that is,
there exists $\tau_*^\alpha\in(0,1)$ such that
\begin{align*}
\alpha^*
=
\begin{cases}
\alpha_1^*,\quad t\in[0,\tau_*^\alpha T),\\
\alpha_2^*,\quad t\in[\tau_*^\alpha T, T],
\end{cases}
\end{align*}
where $\alpha_1^*,\alpha_2^*\in\mathrm{Int} \Theta_A$, $\alpha_1^*\neq \alpha_2^*$. 
Now \eqref{sde} can be expressed as follows. 
\begin{align*}
X_t=
\begin{cases}
\displaystyle
X_0+\int_0^t b(X_s,\beta)\dd s+\int_0^t a(X_s,\alpha_1^*)\dd W_s, 
\quad t\in[0,\tau_*^\alpha T)\\
\displaystyle
X_{\tau_*^\alpha T}+\int_{\tau_*^\alpha T}^t b(X_s,\beta)\dd s
+\int_{\tau_*^\alpha T}^t a(X_s,\alpha_2^*)\dd W_s.
\quad t\in[\tau_*^\alpha T,T]
\end{cases}
\end{align*}

\noindent\textbf{Situation II. }
$\alpha^*$ does not change  
and 
$\beta^*$ changes 
in $0\le t\le T$, that is,
there exists $\tau_*^\beta\in(0,1)$ such that
\begin{align*}
\beta^*
=
\begin{cases}
\beta_1^*,\quad t\in[0,\tau_*^\beta T),\\
\beta_2^*,\quad t\in[\tau_*^\beta T, T],
\end{cases}
\end{align*}
where $\beta_1^*,\beta_2^*\in\mathrm{Int}\Theta_B$, $\beta_1^*\neq \beta_2^*$. 
Now \eqref{sde} can be expressed as follows. 
\begin{align*}
X_t=
\begin{cases}
\displaystyle
X_0+\int_0^t b(X_s,\beta_1^*)\dd s+\int_0^t a(X_s,\alpha^*)\dd W_s, 
\quad t\in[0,\tau_*^\beta T)\\
\displaystyle
X_{\tau_*^\beta T}+\int_{\tau_*^\beta T}^t b(X_s,\beta_2^*)\dd s
+\int_{\tau_*^\beta T}^t a(X_s,\alpha^*)\dd W_s.
\quad t\in[\tau_*^\beta T,T]
\end{cases}
\end{align*}

We consider the following two cases:

\noindent\textbf{Case A.}
The parameters $\alpha_1^*$ and $\alpha_2^*$ 
(resp. $\beta_1^*$ and $\beta_2^*$)
depend on $n$ in Situation I (resp. II),

\noindent\textbf{Case B.}
The parameters $\alpha_1^*$ and $\alpha_2^*$ 
(resp. $\beta_1^*$ and $\beta_2^*$)
are fixed and not depend on $n$ in Situation I (resp. II).

For matrices $c\in\mathbb R^{d_1}\otimes\mathbb R^{d_2}$,
we write $c^{\otimes2}=cc^\TT$, 
where $c^\TT$ is the transpose of $c$.
Let $A(x,\alpha)=a(x,\alpha)^{\otimes2}$ and $\DeX=\Xt-\Xs$.
Let $C^{k,\ell}_{\uparrow}(\mathbb R^d\times\Theta)$ be 
the space of all functions $f$ satisfying the following conditions:
\begin{enumerate}
\item[(i)] 
$f$ is continuously differentiable with respect to $x\in\mathbb R^d$ up to
order $k$ for all $\theta\in\Theta$;
\item[(ii)]  
$f$ and all its $x$-derivatives up to order $k$ are 
$\ell$ times continuously differentiable with respect to $\theta\in\Theta$;
\item[(iii)]  
$f$ and all derivatives are of polynomial growth in $x\in\mathbb R^d$ 
uniformly in $\theta\in\Theta$, i.e.,  
$g$ is of polynomial growth in $x\in\mathbb R^d$ uniformly 
in $\theta\in\Theta$ if, for some $C>0$, we have
\begin{align*}
\sup_{\theta\in\Theta}\|g(x,\theta)\|\le C(1+\|x\|)^C.
\end{align*}
\end{enumerate}

We assume the following conditions:
\begin{enumerate}
\renewcommand{\labelenumi}{{\textbf{[C\arabic{enumi}]}}}
\item 
There exists a constant $C$ such that for any $x,y\in\mathbb R^d$, 
\begin{align*}
\sup_{\alpha\in\Theta_A}\|a(x,\alpha)-a(y,\alpha)\|
+\sup_{\beta\in\Theta_B}\|b(x,\beta)-b(y,\beta)\|\le C\|x-y\|.
\end{align*}
\item   
$\displaystyle\sup_t\EE_{\theta}\left[\|X_t\|^k\right]<\infty$
for all $k\ge0$ and $\theta\in\Theta$.
\item $\displaystyle\inf_{x,\alpha}\det \left(A(x,\alpha)\right)>0$. 
\item 
$a\in C^{4,4}_{\uparrow}(\mathbb R^d\times\Theta_A)$
and
$b\in C^{4,4}_{\uparrow}(\mathbb R^d\times\Theta_B)$.
\item 
The solution of \eqref{sde} is ergodic with its invariant measure $\mu_\theta$ 
such that 
\begin{align*}
\int_{\mathbb R^d}\|x\|^k\dd\mu_\theta(x)<\infty\quad\text{for all }\ k\ge0 
\text{ and } \theta\in\Theta,
\end{align*}
and
\begin{align*}
\int_{\mathbb R^d}f(x)\dd \mu_{\theta_n}(x)
\lto
\int_{\mathbb R^d}f(x)\dd \mu_{\theta_0}(x)
\quad\text{for all measurable function } f  
\end{align*}
as $\theta_n\lto\theta_0$.
\end{enumerate}

\subsection{Estimation for the diffusion parameter}\label{sec2.1}
First, we consider Situation I, that is, the estimation for the diffusion parameter.
Write $\Dea=|\alpha_1^*-\alpha_2^*|$ and let
\begin{align*}
\Xi^\alpha(x,\alpha)
&=
\Bigl[
\tr\bigl(
A^{-1}\partial_{\alpha^{\ell_1}}A
A^{-1}\partial_{\alpha^{\ell_2}}A(x,\alpha)
\bigr)
\Bigr]_{\ell_1,\ell_2=1}^p,
\\
\Gamma^\alpha(x,\alpha_1,\alpha_2) 
&=
\tr\bigl(
A^{-1}(x,\alpha_1)A(x,\alpha_2)-I_d
\bigr)
-\log\det A^{-1}(x,\alpha_1)A(x,\alpha_2).
\end{align*}

Now we additionally assume the following conditions:
\begin{enumerate}
\renewcommand{\labelenumi}{{\textbf{[C\arabic{enumi}-I]}}}
\setcounter{enumi}{5}
\item 
There exist estimators $\hat\alpha_k=\hat\alpha_{k,n}\ (k=1,2)$ such that
\begin{align*}
\sqrt{n}(\hat\alpha_k-\alpha_k^*)=O_p(1).
\end{align*}
\renewcommand{\labelenumi}{{\textbf{[A\arabic{enumi}-II]}}}
\end{enumerate}

\begin{enumerate}
\renewcommand{\labelenumi}{{\textbf{[A\arabic{enumi}-I]}}}
\item
$\alpha_1^*$ and $\alpha_2^*$ depend on $n$, 
and
$\Dea=\vartheta_{\alpha,n}$ satisfies the following.
\begin{align*}
\Dea\lto0,\quad
n\Dea^2\lto\infty,\quad
\frac{h}{\Dea^2}\lto\infty,\quad
T\Dea\lto0
\end{align*}
as $n\to\infty$, and
$\Dea^{-1}(\alpha_k^*-\alpha_0)=O(1)$, where $\alpha_0\in\mathrm{Int}\,\Theta_A$,
\item
For the following three functions 
and for any $r\in(1,2)$ such that $nh^r\lto\infty$,
\begin{align*}
\max_{[n^{1/r}]\le k\le n}
\left|
\frac1k\sum_{i=[n\tau_*^\alpha]+1}^{[n\tau_*^\alpha]+k}f(\Xs)
-\int_{\mathbb R^d}f(x)\dd\mu_{\alpha_0}(x)
\right|\pto0.
\end{align*}
\begin{enumerate}
\item
$\Xi^\alpha(x,\alpha_0)$,
\item
$\partial_\alpha\Xi^\alpha(x,\alpha_0)$,
\item 
$\partial_\alpha^3
\Bigl(
\tr(A^{-1}(x,\alpha)A(x,\alpha_0))
-\log\det A^{-1}(x,\alpha)A(x,\alpha_0)
\Bigr)\Bigr|_{\alpha=\alpha_0}$.
\end{enumerate} 
\end{enumerate}

\begin{enumerate}
\renewcommand{\labelenumi}{{\textbf{[B\arabic{enumi}-I]}}}
\item
$\displaystyle
\inf_x \Gamma^\alpha(x,\alpha_1^*,\alpha_2^*)>0$.
\item
There exists a constant $C>0$ such that
\begin{enumerate}
\item
$\displaystyle\sup_{x,\alpha_k}
\bigl|
\partial_{(\alpha_1,\alpha_2)}\Gamma^\alpha(x,\alpha_1,\alpha_2)
\bigr|<C$,
\item 
$\displaystyle\sup_{x,\alpha_k}
\left|
\left[
\tr\Bigl(
\bigl(A^{-1}(x,\alpha_1)-A^{-1}(x,\alpha_2)\bigr)
\partial_{\alpha^{\ell}} A(x,\alpha_3)
\Bigr)
\right]_{\ell=1}^p
\right|<C$,
\item 
$\displaystyle\sup_{x,\theta}|Q(x,\theta)|<C$, 
\end{enumerate}
where 
$Q(x,\theta)$ is the coefficient of $h^2$ of $\EE_\theta[(\Xt-\Xs)^{\otimes2}|\GG]$, 
that is,
$\EE_\theta[(\Xt-\Xs)^{\otimes2}|\GG]=hA(\Xs,\alpha)+h^2Q(\Xs,\theta)+\cdots$.
Here, $\GG=\sigma\left[\{W_s\}_{s\le t_i^n}\right]$.
\end{enumerate}

\begin{rmk}\label{rmk1}
See Section \ref{sec3} for how to construct the estimators $\hat\alpha_k$
that satisfy \textbf{[C6-I]}. 
If $a(x,\alpha)=\sigma(x)c(\alpha)$ 
for 
$\sigma:\mathbb R^d\lto\mathbb R^d\otimes\mathbb R^d$, 
$c:\mathbb R^p\lto\mathbb R^d\otimes\mathbb R^d$, 
then \textbf{[A2-I]}, \textbf{[B1-I]} and \textbf{[B2-I]}(a),(b)
hold because the functions of (a)-(c) of \textbf{[A2-I]}, 
$\Gamma^\alpha(x,\alpha_1,\alpha_2)$, 
$\partial_{(\alpha_1,\alpha_2)}\Gamma^\alpha(x,\alpha_1,\alpha_2)$
and
\begin{align*}
\left[
\tr\Bigl(
\bigl(A^{-1}(x,\alpha_1)-A^{-1}(x,\alpha_2)\bigr)
\partial_{\alpha^{\ell}} A(x,\alpha_3)
\Bigr)
\right]_{\ell=1}^p
\end{align*}
do not depend on $x$.
Therefore, the Ornstein-Uhlenbeck process and hyperbolic diffusion model, 
which will be described later in Section \ref{sec5}, 
are models that satisfy \textbf{[A2-I]}, 
and these models are examples in Case A.
Hyperbolic diffusion model also satisfies \textbf{[B2-I]}(c),  
and this model is an example in Case B, see Section \ref{sec5.7}.
If the diffusion coefficient is $a(x,\alpha)=\alpha$, then
\textbf{[A1-I]} and \textbf{[A2-I]} can be reduced to the following condition:
\begin{enumerate}
\renewcommand{\labelenumi}{{\textbf{[A\arabic{enumi}'-I]}}}
\item
$\alpha_1^*$ and $\alpha_2^*$ depend on $n$, 
and
$\Dea=\vartheta_{\alpha,n}$ satisfies the following.
\begin{align*}
\Dea\lto0,\quad
n\Dea^2\lto\infty,\quad
T\Dea\lto0
\end{align*}
as $n\to\infty$, and
$\Dea^{-1}(\alpha_k^*-\alpha_0)=O(1)$, where $\alpha_0\in\mathrm{Int}\,\Theta_A$.
\end{enumerate}
Therefore, if $\{X_t\}_{t\ge0}$ after the parameter change
is stationary, then
\textbf{[A1-I]} and \textbf{[A2-I]} can be reduced to the following condition:
\begin{enumerate}
\renewcommand{\labelenumi}{{\textbf{[A\arabic{enumi}''-I]}}}
\item
$\alpha_1^*$ and $\alpha_2^*$ depend on $n$, 
and
$\Dea=\vartheta_{\alpha,n}$ satisfies the following.
\begin{align*}
\Dea\lto0,\quad
n\Dea^2\lto\infty,\quad
\frac{h}{\Dea^2}\lto\infty
\end{align*}
as $n\to\infty$, and
$\Dea^{-1}(\alpha_k^*-\alpha_0)=O(1)$, where $\alpha_0\in\mathrm{Int}\,\Theta_A$,
\end{enumerate}
That is, if $\{X_t\}_{t\ge0}$ after the parameter change 
is stationary and we assume \textbf{[A1'-I]}, 
then \textbf{[A1-I]} and \textbf{[A2-I]} hold, see Section \ref{sec5.2}.
\end{rmk}

In Situation I, we set
\begin{align*}
&F_i(\alpha)
=
\tr\left(A^{-1}(\Xs,\alpha)\frac{(\DeX)^{\otimes 2}}{h}\right)
+\log\det A(\Xs,\alpha),\\
&\Phi_n(\tau:\alpha_1,\alpha_2)
=
\sum_{i=1}^{[n\tau]}F_i(\alpha_1)+\sum_{i=[n\tau]+1}^nF_i(\alpha_2)
\end{align*}
and 
propose
\begin{align*}
\hat\tau_n^\alpha
=\underset{\tau\in[0,1]}{\mathrm{argmin}}\,\Phi_n(\tau:\hat\alpha_1,\hat\alpha_2)
\end{align*}
as an estimator of $\tau_*^\alpha$.

In Case A, we set for $v\in\mathbb R$,
\begin{align*}
e_\alpha
&=\lim_{n\to\infty}\Dea^{-1}(\alpha_1^*-\alpha_2^*),\\
\mathcal J_\alpha
&=\frac12e_\alpha^\TT 
\int_{\mathbb R^d}
\Xi^\alpha(x,\alpha_0)
\dd\mu_{\alpha_0}(x)
e_\alpha,\\
\mathbb F(v)
&=-2\mathcal J_\alpha^{1/2}\mathcal W(v)+\mathcal J_\alpha|v|,
\end{align*}
where $\mathcal W$ is a two sided standard Wiener process.
\begin{thm}\label{th1}
Suppose that \textbf{[C1]}-\textbf{[C5]} and \textbf{[C6-I]} hold in Situation I.
Then, under \textbf{[A1-I]} and \textbf{[A2-I]} in Case A, 
\begin{align*}
n\Dea^2(\hat\tau_n^\alpha-\tau_*^\alpha)
\dto
\underset{v\in\mathbb R}{\mathrm{argmin}}\,\mathbb F(v).
\end{align*}
\end{thm}

\begin{thm}\label{th2}
Suppose that \textbf{[C1]}-\textbf{[C5]} and \textbf{[C6-I]} 
hold in Situation I.
Then,  
under \textbf{[B1-I]} and  \textbf{[B2-I]} in Case B,
\begin{align*}
n(\hat\tau_n^\alpha-\tau_*^\alpha)=O_p(1).
\end{align*}
\end{thm}

\begin{cor}\label{cor1}
Suppose that \textbf{[C1]}-\textbf{[C5]} and \textbf{[C6-I]} 
hold in Situation I.
If for $\epsilon_1\in[0,1)$ there exists $\delta_1\in(0,\frac{1-\epsilon_1}{2})$ 
such that $nh^{1/(\epsilon_1+\delta_1)}\lto0$, then,  
under \textbf{[B1-I]}, \textbf{[B2-I]}(a) and (b) in Case B,
\begin{align*}
n^{\epsilon_1}(\hat\tau_n^\alpha-\tau_*^\alpha)=o_p(1).
\end{align*}
Particularly, for $\epsilon_1\in[0,\frac12)$ there exists always $\delta_1$ 
regardless of $h$.
\end{cor}

\begin{rmk}
Although we could not specify the distribution of $n(\hat\tau_n^\alpha-\tau_*^\alpha)$ 
in Case B, we have strong results in this case.
Actually, in Case B,  
we see $\hat\tau_n^\alpha-\tau_*^\alpha=O_p(n^{-1})$. 
However, if we choose $\Dea=n^{-\nu}$, $0<\nu<\frac12$ 
that satisfies \textbf{[A1-I]}, 
we have 
$n^{1-2\nu}(\hat\tau_n^\alpha-\tau_*^\alpha)=O_p(1)$
in Case A, but 
$n^{1-2\nu}(\hat\tau_n^\alpha-\tau_*^\alpha)=o_p(1)$
in Case B.
\end{rmk}

\subsection{Estimation for the drift parameter}\label{sec2.2}
Next, we consider Situation II, that is, the estimation for the drift parameter.
Write $\Deb=|\beta_1^*-\beta_2^*|$ and let
\begin{align*}
\Xi^\beta(x,\alpha,\beta)
&=
\Bigl[
\partial_{\beta^{\ell_1}}b(x,\beta)^\TT
A^{-1}(x,\alpha)\partial_{\beta^{\ell_2}}b(x,\beta)
\Bigr]_{\ell_1,\ell_2=1}^q,
\\
\Gamma^\beta(x,\alpha,\beta_1,\beta_2)
&= 
\tr
\bigl[
A^{-1}(x,\alpha)(b(x,\beta_1)-b(x,\beta_2))^{\otimes2}
\bigr].
\end{align*}
 
Now we additionally assume the following conditions:
\begin{enumerate}
\renewcommand{\labelenumi}{{\textbf{[C\arabic{enumi}-II]}}}
\setcounter{enumi}{5}
\item 
There exist estimators 
$\hat\alpha=\hat\alpha_n$, $\hat\beta_k=\hat\beta_{k,n}\ (k=1,2)$ 
such that
\begin{align*}
\sqrt{n}(\hat\alpha-\alpha^*)=O_p(1),\quad
\sqrt{T}(\hat\beta_k-\beta_k^*)=O_p(1).
\end{align*}
\end{enumerate}

\begin{enumerate}
\renewcommand{\labelenumi}{{\textbf{[A\arabic{enumi}-II]}}}
\item
$\beta_1^*$ and $\beta_2^*$ depend on $n$ and 
$\Deb=\vartheta_{\beta,n}$ satisfies the following.
\begin{align*}
\Deb\lto0,\quad
T\Deb^2\lto\infty, \quad
T\Deb^4\lto0
\end{align*}
as $n\to\infty$, and
$\Deb^{-1}(\beta_k^*-\beta_0)=O(1)$, where $\beta_0\in\mathrm{Int}\, \Theta_B$.
\item
For the following three functions 
and for any $r\in(1,2)$ such that $nh^r\lto\infty$,
\begin{align*}
\max_{[n^{1/r}]\le k\le n}
\left|
\frac1k\sum_{i=[n\tau_*^\beta]+1}^{[n\tau_*^\beta]+k}f(\Xs)
-\int_{\mathbb R^d}f(x)\dd\mu_{(\alpha^*,\beta_0)}(x)
\right|\pto0.
\end{align*}
\begin{enumerate}
\item
$\Xi^\beta(x,\alpha^*,\beta_0)$,
\item
$\partial_\beta\Xi^\beta(x,\alpha^*,\beta_0)$,
\item 
$\partial_\beta^3
\Bigl(
\tr(A^{-1}(x,\alpha^*)
(b(x,\beta)-b(x,\beta_0))^{\otimes2}
\Bigr)\Bigr|_{\beta=\beta_0}$.
\end{enumerate}

\item There exists an integer $m\ge 4$ such that 
$n^{m/2-1}h^{(m-1)/2}\lto\infty$, $h^{-1/2}\Deb^{m-2}\lto0$ 
and
$b\in C_{\uparrow}^{4,m}(\mathbb R^d\times\Theta_B)$.
\end{enumerate}

\begin{enumerate}
\renewcommand{\labelenumi}{{\textbf{[B\arabic{enumi}-II]}}}
\item
$\displaystyle
\inf_x\Gamma^\beta(x,\alpha^*,\beta_1^*,\beta_2^*)>0$.
\item
There exists a constant $C>0$ such that
\begin{enumerate}
\item

$\displaystyle\sup_{x,\alpha,\beta_k}
\left|
\partial_{(\alpha,\beta_1,\beta_2)}
\Gamma^\beta(x,\alpha,\beta_1,\beta_2)
\right|
<C$, 
\item 
$\displaystyle
\sup_{x,\alpha,\beta_k}
\left|
\Bigl[
\tr
\Bigl(
A^{-1}(x,\alpha)
\partial_{\beta^{\ell}}b(x,\beta_1)
\bigl(
b(x,\beta_2)-b(x,\beta_3)
\bigr)^\TT
\Bigr)
\Bigr]_{\ell=1}^q
\right|<C$. 
\end{enumerate}
\end{enumerate}


\begin{rmk}
See Section \ref{sec3} for how to construct the estimators $\hat\beta_k$
that satisfy \textbf{[C6-II]}. 
If $\{X_t\}_{t\ge0}$ after the parameter change is stationary, then 
\textbf{[A1-II]} and \textbf{[A2-II]} can be reduced to the following condition:
\begin{enumerate}
\renewcommand{\labelenumi}{{\textbf{[A\arabic{enumi}'-II]}}}
\item
$\beta_1^*$ and $\beta_2^*$ depend on $n$ and 
$\Deb=\vartheta_{\beta,n}$ satisfies the following.
\begin{align*}
\Deb\lto0,\quad
T\Deb^2\lto\infty
\end{align*}
as $n\to\infty$, and
$\Deb^{-1}(\beta_k^*-\beta_0)=O(1)$, where $\beta_0\in\mathrm{Int}\, \Theta_B$.
\end{enumerate}
That is, if $\{X_t\}_{t\ge0}$ after the parameter change 
is stationary and we assume \textbf{[A1'-II]}, 
then \textbf{[A1-II]} and \textbf{[A2-II]} hold, see Section \ref{sec5.2}.
The hyperbolic diffusion model is one of the models that satisfy 
\textbf{[B1-II]} and \textbf{[B2-II]}, see Section \ref{sec5.7}.
\end{rmk}

In Situation II, we set 
\begin{align*}
&G_i(\beta|\alpha)
=
\tr\left(A^{-1}(\Xs,\alpha)
\frac{(\DeX-hb(\Xs,\beta))^{\otimes 2}}{h}\right),\\
&\Psi_n(\tau:\beta_1,\beta_2|\alpha)
=
\sum_{i=1}^{[n\tau]}G_i(\beta_1|\alpha)+\sum_{i=[n\tau]+1}^nG_i(\beta_2|\alpha)
\end{align*}
and propose
\begin{align*}
\hat\tau_n^\beta
=\underset{\tau\in[0,1]}{\mathrm{argmin}}\,
\Psi_n(\tau:\hat\beta_1,\hat\beta_2|\hat\alpha)
\end{align*}
as an estimator of $\tau_*^\beta$.

In Case A, we set for $v\in\mathbb R$,
\begin{align*}
e_\beta
&=\lim_{n\to\infty}\Deb^{-1}(\beta_1^*-\beta_2^*),\\
\mathcal J_\beta
&=e_\beta^\TT 
\int_{\mathbb R^d}
\Xi^\beta(x,\alpha^*,\beta_0)
\dd\mu_{(\alpha^*,\beta_0)}(x)
e_\beta,\\
\mathbb G(v)
&=-2\mathcal J_\beta^{1/2}\mathcal W(v)+\mathcal J_\beta|v|.
\end{align*}
\begin{thm}\label{th3}
Suppose that \textbf{[C1]}-\textbf{[C5]} and \textbf{[C6-II]} hold in Situation II.
Then, under \textbf{[A1-II]}-\textbf{[A3-II]} in Case A, 
\begin{align*}
T\Deb^2(\hat\tau_n^\beta-\tau_*^\beta)
\dto
\underset{v\in\mathbb R}{\mathrm{argmin}}\,
\mathbb G(v).
\end{align*}
\end{thm}

\begin{thm}\label{th4}
Suppose that \textbf{[C1]}-\textbf{[C5]} and \textbf{[C6-II]} hold in Situation II.
Then,  
under \textbf{[B1-II]}, \textbf{[B2-II]} in Case B, 
\begin{align*}
T(\hat\tau_n^\beta-\tau_*^\beta)=O_p(1).
\end{align*}
\end{thm}

\begin{rmk}
Although we could not specify the distribution of $T(\hat\tau_n^\beta-\tau_*^\beta)$ 
in Case B, we have strong results in this case.
Actually, in Case B, 
we see $\hat\tau_n^\beta-\tau_*^\beta=O_p(T^{-1})$. 
However, if we choose $\Deb=T^{-\nu}$, $0<\nu<\frac12$ 
that satisfies \textbf{[A1-II]}, 
we have 
$T^{1-2\nu}(\hat\tau_n^\beta-\tau_*^\beta)=O_p(1)$
in Case A, but 
$T^{1-2\nu}(\hat\tau_n^\beta-\tau_*^\beta)=o_p(1)$
in Case B.
\end{rmk}

\section{Estimation of the nuisance parameters $\alpha_k^*$, $\beta_k^*$}\label{sec3}
When the values of the parameters $\alpha_k^*$, $\beta_k^*$ are unknown, 
it is necessary to estimate the parameters $\alpha_k^*$, $\beta_k^*$
for the change point estimation.
In this section, we will discuss estimation of nuisance parameters 
$\alpha_k^*$, $\beta_k^*$.

First, we need the following information to estimate $\alpha_k^*$ or $\beta_k^*$:
\begin{enumerate}
\renewcommand{\labelenumi}{{\textbf{[D\arabic{enumi}]}}}
\item
There exist $\underline\tau^\alpha, \overline\tau^\alpha\in(0,1)$ such that 
$\tau_*^\alpha\in[\underline\tau^\alpha,\overline\tau^\alpha]$.
\item
There exist $\underline\tau^\beta, \overline\tau^\beta\in(0,1)$ such that 
$\tau_*^\beta\in[\underline\tau^\beta,\overline\tau^\beta]$.
\end{enumerate}

If this information is obtained, 
since there are no change points in the intervals 
$[0,\underline\tau^\alpha T]$ and $[\overline\tau^\alpha T,T]$, 
we can estimate $\alpha_1^*$ from the data of $[0,\underline\tau^\alpha T]$
and $\alpha_2^*$ from the data of $[\overline\tau^\alpha T,T]$.
For example,  
we can construct estimators that satisfy \textbf{[C6-I]} or \textbf{[C6-II]}
by the following procedures. 
In Situation I, set
\begin{align*}
\underline{U}_n^{(1)}(\alpha)
&=\sum_{i=1}^{[n\underline\tau^\alpha]}
\Biggl(
\tr\left(
A^{-1}(\Xs,\alpha)\frac{(\DeX)^{\otimes2}}{h}
\right)
+\log\det A(\Xs,\alpha)
\Biggr),\\
\overline{U}_n^{(1)}(\alpha)
&=\sum_{i=[n\overline\tau^\alpha]+1}^{n}
\Biggl(
\tr\left(
A^{-1}(\Xs,\alpha)\frac{(\DeX)^{\otimes2}}{h}
\right)
+\log\det A(\Xs,\alpha)
\Biggr),
\end{align*}
and let
$\hat\alpha_1=\mathrm{arginf}_{\alpha}\,\underline{U}_n^{(1)}(\alpha)$,
and
$\hat\alpha_2=\mathrm{arginf}_{\alpha}\,\overline{U}_n^{(1)}(\alpha)$
as estimators of $\alpha_1^*$ and $\alpha_2^*$, respectively.
Then, we have
\begin{align*}
\sqrt{n}(\hat\alpha_1-\alpha_1^*)=O_p(1),\quad
\sqrt{n}(\hat\alpha_2-\alpha_2^*)
=O_p(1).
\end{align*}
Similarly in Situation II,
set
\begin{align*}
U_n^{(1)}(\alpha)
&=\sum_{i=1}^{n}
\Biggl(
\tr\left(
A_{i-1}^{-1}(\alpha)\frac{(\DeX)^{\otimes2}}{h}
\right)
+\log\det A_{i-1}(\alpha)
\Biggr),\\
\underline{U}_n^{(2)}(\beta|\alpha)
&=\sum_{i=1}^{[n\underline\tau^\beta]}
\tr\left(
A_{i-1}^{-1}(\alpha)\frac{(\DeX-h b_{i-1}(\beta))^{\otimes2}}{h}
\right),\\
\overline{U}_n^{(2)}(\beta|\alpha)
&=\sum_{i=[n\overline\tau^\beta]+1}^{n}
\tr\left(
A_{i-1}^{-1}(\alpha)\frac{(\DeX-h b_{i-1}(\beta))^{\otimes2}}{h}
\right),
\end{align*}
and let
$\hat\alpha=\mathrm{arginf}_{\alpha}\,U_n^{(1)}(\alpha)$,
$\hat\beta_1=\mathrm{arginf}_{\beta}\,\underline{U}_n^{(2)}(\beta|\hat\alpha)$, 
and
$\hat\beta_2=\mathrm{arginf}_{\beta}\,\overline{U}_n^{(2)}(\beta|\hat\alpha)$
as estimators of $\alpha^*$, $\beta_1^*$ and $\beta_2^*$, respectively.
Then, we have
\begin{align*}
\sqrt{T}(\hat\beta_1-\beta_1^*)=O_p(1),\quad
\sqrt{T}(\hat\beta_2-\beta_2^*)=O_p(1).
\end{align*}


Thus, if \textbf{[D1]} or \textbf{[D2]} is satisfied, 
then we can construct estimators that satisfy \textbf{[C6-I]} or \textbf{[C6-II]}.
Next, let us discuss how to find $\underline\tau$ and $\overline\tau$ 
that satisfy \textbf{[D1]} or \textbf{[D2]}.
To find these, 
we need some method to detect changes in the diffusion or drift parameters.
However, in Situations I and II,
the adaptive tests for changes in diffusion and drift parameters proposed 
by 
Tonaki et al. (2020)
makes it possible. 
Namely, we can detect a change in the diffusion or drift parameters 
in the interval $[\tau_1 T,\tau_2 T]$ 
by the following test statistics $\mathcal T_{n}^\alpha(\tau_1,\tau_2)$,
$\mathcal T_{1,n}^\beta(\tau_1,\tau_2)$ and $\mathcal T_{2,n}^\beta(\tau_1,\tau_2)$.
\begin{align*}
\mathcal T_{n}^\alpha(\tau_1,\tau_2)
&=\frac{1}{\sqrt{2d([n\tau_2]-[n\tau_1])}}\max_{1\le k\le [n\tau_2]-[n\tau_1]}
\left|\sum_{i=[n\tau_1]+1}^{[n\tau_1]+k}\hat\eta_i
-\frac{k}{[n\tau_2]-[n\tau_1]}\sum_{i=[n\tau_1]+1}^{[n\tau_2]}\hat\eta_i\right|,
\\
\hat\eta_i
&=\tr\left(A^{-1}(\Xs,\hat\alpha)
\frac{(\Delta X_i)^{\otimes2}}{h}\right),\\
\mathcal T_{1,n}^\beta(\tau_1,\tau_2)
&=\frac{1}{\sqrt{dT(\tau_2-\tau_1)}}\max_{1\le k\le [n\tau_2]-[n\tau_1]}
\left|\sum_{i=[n\tau_1]+1}^{[n\tau_1]+k}\hat\xi_i
-\frac{k}{[n\tau_2]-[n\tau_1]}\sum_{i=[n\tau_1]+1}^{[n\tau_2]}\hat\xi_i\right|,
\\
\hat\xi_i
&=
1_d^\TT a^{-1}(\Xs,\hat\alpha)
(\Xt-\Xs-hb(\Xs,\hat\beta)),\\
\mathcal T_{2,n}^\beta(\tau_1,\tau_2)
&=\frac{1}{\sqrt{T(\tau_2-\tau_1)}}\max_{1\le k\le [n\tau_2]-[n\tau_1]}
\left\|\mathcal I_n^{-1/2}(\tau_1,\tau_2)
\left(\sum_{i=[n\tau_1]+1}^{[n\tau_1]+k}\hat\zeta_i
-\frac{k}{[n\tau_2]-[n\tau_1]}
\sum_{i=[n\tau_1]+1}^{[n\tau_2]}\hat\zeta_i\right)\right\|,
\\
\hat\zeta_i&=\partial_\beta b(\Xs,\hat\beta)^\TT A^{-1}(\Xs,\hat\alpha)
(\Xt-\Xs-hb(\Xs,\hat\beta)),\\
\mathcal I_n(\tau_1,\tau_2)
&=\frac{1}{[n\tau_2]-[n\tau_1]}\sum_{i=[n\tau_1]+1}^{[n\tau_2]}
\partial_\beta b(\Xs,\hat\beta)^\TT A^{-1}(\Xs,\hat\alpha)
\partial_\beta b(\Xs,\hat\beta).
\end{align*}
Below we describe how to find $\underline\tau$ and $\overline\tau$.

First, the assumption is that a change is detected in the interval $[0,T]$.
\begin{enumerate}
\renewcommand{\labelenumi}{{\textbf{U-\arabic{enumi})}}}
\item
Choose $\tau_1^U\in(0,1)$ (e.g., $\tau_1^U=3/4$), and 
investigate the change point in the interval $[0,\tau_1^U T]$.
\begin{enumerate}
\item
If a change is detected, 
set $\overline\tau=\tau_1^U$ and go to L-1). 
\item
If not detected, go to U-2).
\end{enumerate}
\renewcommand{\labelenumi}{{{\textbf{U-}\it\textbf{\alph{enumi}}}\textbf{)}}}
\setcounter{enumi}{10}
\item
Choose $\tau_k^U\in(\tau_{k-1}^U,1)$ (e.g., $\tau_k^U=1-2^{-(k+1)}$), and 
investigate the change point in the interval $[0,\tau_k^U T]$.
\begin{enumerate}
\item
If a change is detected, 
set $\overline\tau=\tau_k^U$ and go to L-1). 
\item
If not detected, go to U-$(k+1)$).
\end{enumerate}
\end{enumerate}

Assume that we chose $\tau_k^U$ as $\overline\tau$ in U-k).
\begin{enumerate}
\renewcommand{\labelenumi}{{\textbf{L-\arabic{enumi})}}}
\item
Choose $\tau_1^L\in(0,\tau_k^U)$ (e.g., $\tau_1^L=1/4$ or $\tau_{k-1}^U$), and 
investigate the change point in the interval $[\tau_1^L T,T]$.
\begin{enumerate}
\item
If a change is detected, 
set $\underline\tau=\tau_1^L$. 
\item
If not detected, go to L-2).
\end{enumerate}
\renewcommand{\labelenumi}{{{\textbf{L-}\it\textbf{\alph{enumi}}}\textbf{)}}}
\setcounter{enumi}{12}
\item
Choose $\tau_m^L\in(0,\tau_{m-1}^L)$ 
(e.g., $\tau_m^L=2^{-(m+1)}$ or $\tau_{k-m}^L$ (if $k>m$)), and 
investigate the change point in the interval $[\tau_m^L T,T]$.
\begin{enumerate}
\item
If a change is detected, 
set $\underline\tau=\tau_m^L$. 
\item
If not detected, go to L-$(m+1)$).
\end{enumerate}
\end{enumerate}
We may also choose $\underline\tau$ and $\overline\tau$ at the same time,
that is,
\begin{enumerate}
\renewcommand{\labelenumi}{{\textbf{\arabic{enumi})}}}
\item
Choose $\tau_1\in(0,1/2)$ (e.g., $\tau_1=1/4$), and 
investigate the change point in the interval $[\tau_1 T,(1-\tau_1)T]$.
\begin{enumerate}
\item
If a change is detected, 
set $\underline\tau=\tau_1$ and $\overline\tau=1-\tau_1$. 
\item
If not detected, go to 2).
\end{enumerate}
\renewcommand{\labelenumi}{{{\it\textbf{\alph{enumi}}}\textbf{)}}}
\setcounter{enumi}{10}
\item
Choose $\tau_k\in(0,\tau_{k-1})$ 
(e.g., $\tau_k=2^{-(k+1)}$), and 
investigate the change point in the interval $[\tau_k T,(1-\tau_k)T]$.
\begin{enumerate}
\item
If a change is detected, 
set $\underline\tau=\tau_k$ and $\overline\tau=1-\tau_k$.
\item
If not detected, go to $k+1$).
\end{enumerate}
\end{enumerate}
We can choose $\underline\tau$ and $\overline\tau$ in the above manner.

\section{The powers of tests in Case A}\label{sec4}
As we saw in Section \ref{sec3}, 
we need estimators of the nuisance parameters $\alpha_k^*$ and $\beta_k^*$
when estimating the change point, 
and have to find $\underline\tau$ and $\overline\tau$ 
that satisfy \textbf{[D1]} or \textbf{[D2]} in order to construct the estimators.
In Section \ref{sec3}, 
we tried to find them using the adaptive tests proposed by 
Tonaki et al. (2020).
When using those tests, 
it is necessary to check whether the consistency holds. 
However, the consistency of the tests in Case B was mentioned, 
but not in Case A.
Thus, in this section, 
we discuss the consistency of the following tests 
$\mathcal T_n^\alpha$, $\mathcal T_{1,n}^\beta$ and $\mathcal T_{2,n}^\beta$
in Case A.
\begin{align*}
\mathcal T_{n}^\alpha
&=\frac{1}{\sqrt{2dn}}\max_{1\le k\le n}
\left|\sum_{i=1}^k\hat\eta_i-\frac{k}{n}\sum_{i=1}^n\hat\eta_i\right|,
\quad
\hat\eta_i
=\tr\left(A^{-1}(\Xs,\hat\alpha)
\frac{(\Delta X_i)^{\otimes2}}{h}\right),\\
\mathcal T_{1,n}^\beta
&=\frac{1}{\sqrt{dT}}\max_{1\le k\le n}
\left|\sum_{i=1}^k\hat\xi_i-\frac{k}{n}\sum_{i=1}^n\hat\xi_i\right|,
\quad
\hat\xi_i=
1_d^\TT a^{-1}(\Xs,\hat\alpha)
(\Xt-\Xs-hb(\Xs,\hat\beta)),\\
\mathcal T_{2,n}^\beta
&=\frac{1}{\sqrt{T}}\max_{1\le k\le n}
\left\|\mathcal I_n^{-1/2}\left(\sum_{i=1}^k\hat\zeta_i-\frac kn
\sum_{i=1}^n\hat\zeta_i\right)\right\|,
\end{align*}
\begin{align*}
\hat\zeta_i&=\partial_\beta b(\Xs,\hat\beta)^\TT A^{-1}(\Xs,\hat\alpha)
(\Xt-\Xs-hb(\Xs,\hat\beta)),\\
\mathcal I_n
&=\frac1n\sum_{i=1}^n
\partial_\beta b(\Xs,\hat\beta)^\TT A^{-1}(\Xs,\hat\alpha)
\partial_\beta b(\Xs,\hat\beta).
\end{align*}

First, we consider the power of the test for the diffusion parameter $\alpha$,
that is, the following hypothesis testing problem:
$H_0^\alpha$ : $\alpha^*$ does not change over $0\le t\le T$,
$H_1^\alpha$ : There exists $\tau_*^\alpha\in(0,1)$ such that
\begin{align*}
\alpha^*
=
\begin{cases}
\alpha_1^*, & t\in[0,\tau_*^\alpha T), \\
\alpha_2^*, & t\in[\tau_*^\alpha T, T],
\end{cases}
\end{align*}
where 
$\alpha_1^*$ and $\alpha_2^*$ depend on $n$, 
and 
$\alpha_1^*\neq\alpha_2^*$. Let $\Dea=|\alpha_1^*-\alpha_2^*|$.
Now, we assume the following conditions:
\begin{enumerate}
\renewcommand{\labelenumi}{{\textbf{[E\arabic{enumi}]}}}
\item 
Under $H_1^\alpha$,
$\Dea\lto0$ and $n\Dea^2\lto\infty$ as $n\to\infty$, 
and
there exist $\alpha_0\in\mathrm{Int}\,\Theta_A$ and $c_k\in\mathbb R^p$ 
such that  
$\vartheta_\alpha^{-1}(\alpha_k^*-\alpha_0)\lto c_k$ 
for $k=1,2$.
\item
Under $H_1^\alpha$,
there exists an estimator $\hat\alpha$ of $\alpha_0$ such that 
$\vartheta_\alpha^{-1}(\hat\alpha-\alpha_0)=O_p(1)$.
\item 
$\displaystyle
\int_{\mathbb R^d}
\left[\tr\bigl(
A^{-1}\partial_{\alpha^\ell} A(x,\alpha_0)
\bigr)
\right]_\ell
\dd\mu_{\alpha_0}(x)(c_1-c_2)\neq0$
under $H_1^\alpha$.
\end{enumerate}

\begin{rmk}
1-dimensional Ornstein-Ulenbeck process and hyperbolic diffusion model are examples of models
which satisfy \textbf{[E2]} and \textbf{[E3]}.
See Section \ref{sec5.3} for some discussion of estimator $\hat\alpha$ 
that satisfies \textbf{[E2]}.
See Section \ref{sec5.4} for models that satisfy \textbf{[E3]}.
\end{rmk}

Let $\epsilon\in(0,1)$ and $w_k(\epsilon)$ denote the upper-$\epsilon$ point of
$\displaystyle\sup_{0\le s\le 1}\|\boldsymbol B_k^0(s)\|$,
that is,
$\displaystyle P(\sup_{0\le s\le 1}\|\boldsymbol B_k^0(s)\|>w_k(\epsilon))=\epsilon$.

\begin{prop}\label{prop3}
Assume \textbf{[C1]}-\textbf{[C5]},  
\textbf{[E1]}-\textbf{[E3]}.
Then, under $H_1^\alpha$, 
\begin{align*}
P\bigl(\mathcal T_n^\alpha>w_1(\epsilon)\bigr)\lto1,
\end{align*}
that is, the test $\mathcal T_n^\alpha$ is consistent.
\end{prop}

In the following, we assume that $\alpha^*$ does not change over $0\le t\le T$.
We consider the power of the test for the drift parameter $\beta$,
that is, 
the following hypothesis testing problem:
$H_0^\beta$ : $\beta^*$ does not change over $0\le t\le T$,
$H_1^\beta$ : There exists $\tau_*^\beta\in(0,1)$ such that
\begin{align*}
\beta^*
=
\begin{cases}
\beta_1^*, &t\in[0,\tau_*^\beta T), \\
\beta_2^*, &t\in[\tau_*^\beta T,T],
\end{cases}
\end{align*}
where 
$\beta_1^*$ and $\beta_2^*$ depend on $n$,
and $\beta_1^*\neq\beta_2^*$. Let $\Deb=|\beta_1^*-\beta_2^*|$.
Now, we assume the following conditions:
 
\begin{enumerate}
\renewcommand{\labelenumi}{{\textbf{[F\arabic{enumi}]}}}
\item 
Under $H_1^\beta$,
$\Deb\lto0$ and $T\Deb^2\lto\infty$ as $n\to\infty$, 
and
there exist $\beta_0\in\mathrm{Int}\,\Theta_B$ and $d_k\in\mathbb R^q$ 
such that  
$\Deb^{-1}(\beta_k^*-\beta_0)\lto d_k$ 
for $k=1,2$.
\item
Under $H_1^\beta$,
there exist 
$\bar\beta^*$ with $\bar\beta^*-\beta_0=O(\Deb)$
and
estimators $\hat\alpha$, $\hat\beta$ such that
\begin{align*}
\sqrt{n}(\hat\alpha-\alpha^*)=O_p(1), \quad
\sqrt{T}(\hat\beta-\bar\beta^*)=O_p(1).
\end{align*}
\item 
$\displaystyle
\int_{\mathbb R^d}
1_d^\TT a^{-1}(x,\alpha^*)\partial_\beta b(x,\beta_0)
\dd\mu_{(\alpha^*,\beta_0)}(x)(d_1-d_2)
\neq0$
under $H_1^\beta$.
\end{enumerate}
\begin{rmk}
See 
Kessler (1995,1997), 
Uchida and Yoshida (2011, 2012), 
Tonaki et al. (2020)
for how to construct the estimators $\hat\alpha$ and $\hat\beta$ that satisfy 
\textbf{[F2]}.
See Section \ref{sec5.6} and \ref{sec5.7} for models that satisfy \textbf{[F3]}.
\end{rmk}

\begin{prop}\label{prop4}
Assume \textbf{[C1]}-\textbf{[C5]}, 
\textbf{[F1]}-\textbf{[F3]}.
Then, under $H_1^\beta$, 
\begin{align*}
P\bigl(\mathcal T_{1,n}^\beta>w_1(\epsilon)\bigr)\lto1.
\end{align*}
\end{prop}

We additionally assume the following condition:
\begin{enumerate}
\renewcommand{\labelenumi}{{\textbf{[F\arabic{enumi}]}}}
\setcounter{enumi}{3}
\item
There exists an integer $M\ge3$ such that $h^{-1/2}\Deb^{M-1}\lto0$ 
and $b\in C_{\uparrow}^{4,M}(\mathbb R^d\times\Theta_B)$.
\end{enumerate}
\begin{prop}\label{prop5}
Assume \textbf{[C1]}-\textbf{[C5]}, 
\textbf{[F1]}, \textbf{[F2]} and \textbf{[F4]}.
Then, under $H_1^\beta$, 
\begin{align*}
P\bigl(\mathcal T_{2,n}^\beta>w_q(\epsilon)\bigr)\lto1.
\end{align*}
\end{prop}

\begin{rmk}
For the consistency of the test $\mathcal T_{2,n}^\beta$, 
only \textbf{[F1]}, \textbf{[F2]} and \textbf{[F4]} are sufficient.
This is because 
\begin{enumerate}
\renewcommand{\labelenumi}{{\textbf{[F\arabic{enumi}']}}}
\setcounter{enumi}{2}
\item
$\displaystyle
\int_{\mathbb R^d}
\partial_\beta b(x,\beta_0)^\TT A^{-1}(x,\alpha^*)\partial_\beta b(x,\beta_0)
\dd\mu_{\beta_0}(x)(d_1-d_2)
\neq0$,
\end{enumerate} 
which corresponds to \textbf{[F3]}, 
is always valid since
$\displaystyle
\int_{\mathbb R^d}
\partial_\beta b(x,\beta_0)^\TT A^{-1}(x,\alpha^*)\partial_\beta b(x,\beta_0)
\dd\mu_{\beta_0}(x)$
is regular and 
\begin{align*}
|d_1-d_2|=\lim_{n\to\infty}
|\vartheta_\beta^{-1}(\beta_1^*-\beta_0)-\vartheta_\beta^{-1}(\beta_2^*-\beta_0)|
=\lim_{n\to\infty}|\vartheta_\beta^{-1}(\beta_1^*-\beta_2^*)|=1\neq0.
\end{align*}
\end{rmk}

\section{Examples}\label{sec5}
\subsection{Sufficient condition of assumptions [A2-I] and [A2-II]}\label{sec5.2}
A process $\{X_t\}_{t\ge0}$ with a single change point can be expressed as follows.
There exists a process $\{\tilde X_t\}_{t\ge0}$ such that
$X_t^1=\tilde X_t(\theta_1^*)$, 
$X_0^1=x_0^1$,
$X_t^2=\tilde X_t(\theta_2^*)$, 
$X_0^2=x_0^2$,
$X_{\tau T}^1=X_{\tau T}^2$
and 
\begin{align*}
X_t
=
\begin{cases}
X_t^1,\quad 0\le t< \tau T,\\
X_t^2,\quad \tau T\le t\le T.
\end{cases}
\end{align*}
If $\{X_t^2\}_{t\ge0}$ is stationary and $\theta_2^*\lto\theta_0$, then
we have, for $f\in C_\uparrow^1(\mathbb R^d)$,
\begin{align*}
&\max_{[n^{1/r}]\le k\le n}
\left|
\frac{1}{k}\sum_{i=[n\tau]+1}^{[n\tau]+k}f(\Xs)
-\int_{\mathbb R^d}f(x)\dd\mu_{\theta_0}(x)
\right|\\
&=
\max_{[n^{1/r}]\le k\le n}
\left|
\frac{1}{k}\sum_{i=[n\tau]+1}^{[n\tau]+k}f(\Xs^2)
-\int_{\mathbb R^d}f(x)\dd\mu_{\theta_0}(x)
\right|\\
&{\stackrel{d}{=}}
\max_{[n^{1/r}]\le k\le n}
\left|
\frac{1}{k}\sum_{i=1}^{k}f(\Xs^2)
-\int_{\mathbb R^d}f(x)\dd\mu_{\theta_0}(x)
\right|
\pto0
\end{align*}
and \textbf{[A2-I]} or \textbf{[A2-II]} holds.

\subsection{Model that satisfies the assumption [E2]}\label{sec5.3}
As an example of the model that satisfies \textbf{[E2]}, 
we consider the $d$-dimensional diffusion process 
\begin{align*}
X_t=
\begin{cases}
\displaystyle
X_0+\int_0^tb(X_s,\beta)\dd s
+\int_0^t\sigma(X_s)\delta(\boldsymbol\alpha_1^*)\dd W_s, 
\quad t\in[0, \tau_*^\alpha T)\\
\displaystyle
X_{\tau_*^\alpha T}+\int_{\tau_*^\alpha T}^tb(X_s,\beta)\dd s
+\int_{\tau_*^\alpha T}^t\sigma(X_s)\delta(\boldsymbol\alpha_2^*)\dd W_s, 
\quad t\in[\tau_*^\alpha T, T]
\end{cases}
\end{align*}
where
$\sigma:\mathbb R^d\lto\mathbb R^d\otimes\mathbb R^d$, 
$\delta(\boldsymbol\alpha)=\diag(\alpha_1,\ldots,\alpha_d)$,
$\boldsymbol\alpha=(\alpha_1,\ldots,\alpha_d)^\TT$, 
$\alpha_1,\ldots,\alpha_d>0$. 
The true value of the parameters are
\begin{align*}
\boldsymbol\alpha_1^*
=\left(
\begin{array}{c}
\alpha_{1,1}^*\\
\vdots\\
\alpha_{1,d}^*
\end{array}
\right),\ 
\boldsymbol\alpha_2^*
=\left(
\begin{array}{c}
\alpha_{2,1}^*\\
\vdots\\
\alpha_{2,d}^*
\end{array}
\right),
\end{align*}
which convergence to $\boldsymbol\alpha_0=(\alpha_{0,1},\ldots,\alpha_{0,d})^\TT$.
We define
\begin{align*}
U^{(1)}_n(\boldsymbol\alpha)
=
\sum_{i=1}^n
\Biggl(
\tr\left(
A_{i-1}^{-1}(\boldsymbol\alpha)\frac{(\DeX)^{\otimes2}}{h}
\right)
+\log\det A_{i-1}(\boldsymbol\alpha)
\Biggr)
\end{align*}
and set the estimator
$\hat{\boldsymbol\alpha}
=\mathrm{arginf}_{\boldsymbol\alpha}\,U^{(1)}_n(\boldsymbol\alpha)$.
Then, we have 
\begin{align}\label{c2}
\Dea^{-1}(\hat{\boldsymbol\alpha}-\boldsymbol\alpha_0)=O_p(1).
\end{align}
\begin{proof}[Proof of \eqref{c2}]
We see
\begin{align*}
\partial_{\alpha_j} U_n^{(1)}(\boldsymbol\alpha)
=\frac{2}{\alpha_j}
\left(
\frac{-1}{\alpha_j^2}\sum_{i=1}^n
\tr\left(
[\sigma(\Xs)^\TT]^{-1}
\delta(\mathrm e_j)
\sigma(\Xs)^{-1}
\frac{(\DeX)^{\otimes2}}{h}
\right)
+n
\right),
\end{align*}
where $\mathrm e_j=(0,\ldots,1,\ldots,0)^\TT$. 
Then, we have 
\begin{align*}
\hat\alpha_j
=
\sqrt{
\frac{1}{n}\sum_{i=1}^n
\tr\left(
[\sigma(\Xs)^\TT]^{-1}
\delta(\mathrm e_j)
\sigma(\Xs)^{-1}
\frac{(\DeX)^{\otimes2}}{h}
\right)
}
\end{align*}
On the other hand, we define
\begin{align*}
\underline{U}_n^{(1)}(\boldsymbol\alpha)
&=
\sum_{i=1}^{[n\tau_*^\alpha]}
\Biggl(
\tr\left(
A_{i-1}^{-1}(\boldsymbol\alpha)\frac{(\DeX)^{\otimes2}}{h}
\right)
+\log\det A_{i-1}(\boldsymbol\alpha)
\Biggr),\\
\overline{U}_n^{(1)}(\boldsymbol\alpha)
&=
\sum_{i=[n\tau_*^\alpha]+1}^n
\Biggl(
\tr\left(
A_{i-1}^{-1}(\boldsymbol\alpha)\frac{(\DeX)^{\otimes2}}{h}
\right)
+\log\det A_{i-1}(\boldsymbol\alpha)
\Biggr)
\end{align*}
and set
$\hat{\boldsymbol\alpha}_1
=\mathrm{arginf}_{\boldsymbol\alpha}\,\underline{U}_n^{(1)}(\boldsymbol\alpha)$, 
$\hat{\boldsymbol\alpha}_2
=\mathrm{arginf}_{\boldsymbol\alpha}\,\overline{U}_n^{(1)}(\boldsymbol\alpha)$.
In the same way, we have
\begin{align*}
\hat\alpha_{1,j}
&=\sqrt{\frac1{[n\tau_*^\alpha]} \sum_{i=1}^{[n\tau_*^\alpha]} 
\tr\left(
[\sigma(\Xs)^\TT]^{-1}
\delta(\mathrm e_j)
\sigma(\Xs)^{-1}
\frac{(\DeX)^{\otimes2}}{h}
\right)},\\
\hat\alpha_{2,j}
&=\sqrt{\frac1{n-[n\tau_*^\alpha]} \sum_{i=[n\tau_*^\alpha]+1}^n 
\tr\left(
[\sigma(\Xs)^\TT]^{-1}
\delta(\mathrm e_j)
\sigma(\Xs)^{-1}
\frac{(\DeX)^{\otimes2}}{h}
\right)}.
\end{align*}
Noting that
$\vartheta_\alpha^{-1}(\hat{\boldsymbol\alpha}_k-\boldsymbol\alpha_0)=O_p(1)$
and
\begin{align*}
\hat\alpha_j
=\sqrt{\frac{[n\tau_*^\alpha]}{n}\hat\alpha_{1,j}^2
+\frac{n-[n\tau_*^\alpha]}{n}\hat\alpha_{2,j}^2},
\end{align*}
we obtain
\begin{align*}
\|\hat{\boldsymbol\alpha}-\boldsymbol\alpha_0\|
\le \sum_{j=1}^d
|\hat\alpha_j-\alpha_{0,j}|
&=\sum_{j=1}^d
\frac{|\hat\alpha_j^2-\alpha_{0,j}^2|}{|\hat\alpha_j+\alpha_{0,j}|}\\
&=\sum_{j=1}^d
\frac{1}{|\hat\alpha_j+\alpha_{0,j}|}
\left|
\frac{[n\tau_*^\alpha]}{n}(\hat\alpha_{1,j}^2-\alpha_{0,j}^2)
+\frac{n-[n\tau_*^\alpha]}{n}(\hat\alpha_{2,j}^2-\alpha_{0,j}^2)
\right|\\
&\le
\sum_{j=1}^d
\frac{1}{|\hat\alpha_j+\alpha_{0,j}|}
\left(
|\hat\alpha_{1,j}-\alpha_{0,j}||\hat\alpha_{1,j}+\alpha_{0,j}|
+|\hat\alpha_{2,j}-\alpha_{0,j}||\hat\alpha_{2,j}+\alpha_{0,j}|
\right)\\
&\le 
\sum_{j=1}^d
\vartheta_\alpha
\left(\frac{|\hat\alpha_{1,j}+\alpha_{0,j}|}{|\hat\alpha_j+\alpha_{0,j}|}
\vartheta_\alpha^{-1}|\hat\alpha_{1,j}-\alpha_{0,j}|
+\frac{|\hat\alpha_{2,j}+\alpha_{0,j}|}{|\hat\alpha_j+\alpha_{0,j}|}
\vartheta_\alpha^{-1}|\hat\alpha_{2,j}-\alpha_{0,j}|
\right).
\end{align*}
From $\hat{\boldsymbol\alpha}_k\pto\boldsymbol\alpha_0$ and
$\hat{\boldsymbol\alpha}\pto\boldsymbol\alpha_0$,
we have
$\dfrac{|\hat\alpha_{k,j}+\alpha_{0,j}|}
{|\hat\alpha_j+\alpha_{0,j}|}=O_p(1)$ 
and
$\vartheta_\alpha^{-1}(\hat{\boldsymbol\alpha}-\boldsymbol\alpha_0)=O_p(1)$.
\end{proof}
From the above, 
1-dimensional Ornstein-Uhlenbeck process and hyperbolic diffusion model 
are models which satisfy \textbf{[E2]} 
because the diffusion coefficient is $a(x,\alpha)=\alpha$.


\subsection{Model that satisfies the assumption [E3] and model that does not}\label{sec5.4}
First, as an example of a model that satisfies \textbf{[E3]},
we consider the $d$-dimensional diffusion process with the diffusion coefficient
\begin{align*}
a(x,\boldsymbol\alpha)=
\sigma(x)
\diag(\alpha_1,\ldots,\alpha_d),
\end{align*}
where 
$\sigma:\mathbb R^d\lto\mathbb R^d\otimes\mathbb R^d$, 
$\boldsymbol\alpha=(\alpha_1,\ldots,\alpha_d)^\TT$, $\alpha_1,\ldots, \alpha_d>0$.
The true value of the parameters are
$\boldsymbol\alpha_1^*=\boldsymbol\alpha_0+\Dea\boldsymbol c_1$
and
$\boldsymbol\alpha_2^*=\boldsymbol\alpha_0+\Dea\boldsymbol c_2$,
where
$\boldsymbol\alpha_0=(\alpha_{0,1},\ldots,\alpha_{0,3})^\TT$, 
$\boldsymbol c_1=(c_{1,1},\ldots, c_{1,d})^\TT$, 
$\boldsymbol c_2=(c_{2,1},\ldots, c_{2,d})^\TT$.
Now we have
\begin{align*}
A(x,\boldsymbol\alpha)
&=\sigma(x)\diag(\alpha_1^2,\ldots,\alpha_d^2)\sigma(x)^\TT,\\
A^{-1}(x,\boldsymbol\alpha)
&=[\sigma(x)^\TT]^{-1}
{\mathrm{diag}}(1/\alpha_1^2,\ldots,1/\alpha_d^2)\sigma(x)^{-1},\\
\partial_{\alpha_j}A(x,\boldsymbol\alpha)
&=\sigma(x)\diag(0,\ldots,2\alpha_j,\ldots,0)\sigma(x)^\TT,\\
\tr\left(
A^{-1}\partial_{\alpha_j}A(x,\boldsymbol\alpha)
\right)
&=
\tr\Bigl(
[\sigma(x)^\TT]^{-1}\diag(0,\ldots,2/\alpha_j,\ldots,0)\sigma(x)^\TT
\Bigr)
=\frac{2}{\alpha_j},
\end{align*}
\begin{align*}
\int_{\mathbb R^d}
\Bigl[
\tr\bigl(
A^{-1}\partial_{\alpha^\ell} A(x,\boldsymbol\alpha_0)
\bigr)
\Bigr]_\ell
\dd\mu_{\alpha_0}(x)(\boldsymbol c_1-\boldsymbol c_2)
=\sum_{j=1}^d\frac{2(c_{1,j}-c_{2,j})}{\alpha_{0,j}}.
\end{align*}
Therefore if
\begin{align}\label{eq5.4-1}
\sum_{j=1}^d\frac{c_{1,j}-c_{2,j}}{\alpha_j}\neq0,
\end{align}
then \textbf{[E3]} holds. 
Especially, we have \eqref{eq5.4-1} if any of the following cases:
\begin{enumerate}
\item
$c_{1,j}-c_{2,j}\ge0$ for all $1\le j\le d$,
and 
$c_{1,j}-c_{2,j}>0$ for some $1\le j\le d$.
\item
$c_{1,j}-c_{2,j}\le 0$ for all $1\le j\le d$,
and 
$c_{1,j}-c_{2,j}<0$ for some $1\le j\le d$.
\end{enumerate}
That is, \textbf{[E3]} holds when only $\alpha_j$ $(1\le j\le d)$ changes.
From the above, 
1-dimensional Ornstein-Uhlenbeck process and hyperbolic diffusion model 
are models which satisfy \textbf{[E3]} 
because the diffusion coefficient is $a(x,\alpha)=\alpha$.


Next, as an example of a model that does not satisfy \textbf{[E3]}, 
we consider the 2-dimensional diffusion process with diffusion coefficient
\begin{align*}
a(x,\boldsymbol\alpha)=
\left(
\begin{array}{cc}
\alpha_1 & \alpha_2\\
0 & \alpha_3 
\end{array}
\right),
\quad(\alpha_1,\alpha_3>0).
\end{align*}
The true value of the parameters are
\begin{align*}
\boldsymbol\alpha_0
=\left(
\begin{array}{c}
\alpha_{0,1} \\
\alpha_{0,2} \\
\alpha_{0,3}
\end{array}
\right),\ 
\boldsymbol\alpha_1^*
=\boldsymbol\alpha_0
+\vartheta_\alpha
\left(
\begin{array}{c}
c_{1,1} \\
c_{1,2}\\
c_{1,3}
\end{array}
\right),\ 
\boldsymbol\alpha_2^*
=\boldsymbol\alpha_0
+\vartheta_\alpha
\left(
\begin{array}{c}
c_{2,1} \\
c_{2,2}\\
c_{2,3} 
\end{array}
\right).
\end{align*}
Now we have
\begin{align*}
A(x,\boldsymbol\alpha)
&=
\left(
\begin{array}{cc}
\alpha_1 & \alpha_2\\
0 & \alpha_3 
\end{array}
\right)
\left(
\begin{array}{cc}
\alpha_1 & 0\\
\alpha_2 & \alpha_3 
\end{array}
\right)
=
\left(
\begin{array}{cc}
\alpha_1^2+\alpha_2^2 & \alpha_2\alpha_3\\
\alpha_2\alpha_3 & \alpha_3^2 
\end{array}
\right),\\ 
A^{-1}(x,\boldsymbol\alpha)
&=
\frac{1}{\alpha_1^2\alpha_3^2}
\left(
\begin{array}{cc}
\alpha_3^2 & -\alpha_2\alpha_3\\
-\alpha_2\alpha_3 & \alpha_1^2+\alpha_2^2 
\end{array}
\right),
\end{align*}
\begin{align*}
\partial_{\alpha_1}A(x,\boldsymbol\alpha)
=
\left(
\begin{array}{cc}
2\alpha_1 & 0\\
0 & 0 
\end{array}
\right),\ 
\partial_{\alpha_2}A(x,\boldsymbol\alpha)
=
\left(
\begin{array}{cc}
2\alpha_2 & \alpha_3\\
\alpha_3 & 0 
\end{array}
\right),\ 
\partial_{\alpha_3}A(x,\boldsymbol\alpha)
=
\left(
\begin{array}{cc}
0 & \alpha_2\\
\alpha_2 & 2\alpha_3 
\end{array}
\right),
\end{align*}
\begin{align*}
\tr\left(
A^{-1}\partial_{\alpha_1}A(x,\boldsymbol\alpha)
\right)
&=\frac{2\alpha_1\alpha_3^2}{\alpha_1^2\alpha_3^2}
=\frac{2}{\alpha_1},\\
\tr\left(
A^{-1}\partial_{\alpha_2}A(x,\boldsymbol\alpha)
\right)
&=\frac{\alpha_2\alpha_3^2-\alpha_2\alpha_3^2}{\alpha_1^2\alpha_3^2}
=0,\\
\tr\left(
A^{-1}\partial_{\alpha_3}A(x,\boldsymbol\alpha)
\right)
&=\frac{-\alpha_2^2\alpha_3+2\alpha_1^2\alpha_3+\alpha_2^2\alpha_3}{\alpha_1^2\alpha_3^2}
=\frac{2}{\alpha_3},
\end{align*}

\begin{align}\label{eq5.4-2}
\int_{\mathbb R^2}
\left(
\frac{2}{\alpha_{0,1}},
0,
\frac{2}{\alpha_{0,3}}
\right)
\dd\mu_{\boldsymbol \alpha_0}(x)
\left(
\begin{array}{c}
c_{1,1}-c_{2,1}\\
c_{1,2}-c_{2,2}\\ 
c_{1,3}-c_{2,3}
\end{array}
\right)
=
\frac{2(c_{1,1}-c_{2,1})}{\alpha_{0,1}}
+\frac{2(c_{1,3}-c_{2,3})}{\alpha_{0,3}}.
\end{align}
Therefore, \eqref{eq5.4-2} does not equal $0$ in the following cases:
\begin{enumerate}
\item
$c_{1,j}-c_{2,j}\ge0$ for all $j=1, 3$,
and 
$c_{1,j}-c_{2,j}>0$ for some $j=1, 3$.
\item
$c_{1,j}-c_{2,j}\le 0$ for all $j=1, 3$,
and 
$c_{1,j}-c_{2,j}<0$ for some $j=1, 3$.
\end{enumerate}
However \eqref{eq5.4-2} equals $0$ when 
$c_{1,1}-c_{2,1}=c_{1,3}-c_{2,3}=0$, 
hence when only $\alpha_2$ changes, \textbf{[E3]} does not hold.

\subsection{Ornstein-Uhlenbeck process}\label{sec5.6}
We consider the 1-dimensional Ornstein-Uhlenbeck process
\begin{align*}
\dd X_t=-\beta(X_t-\gamma)\dd t+\alpha\dd W_t,\quad X_0=x_0.\quad
(\alpha,\beta>0, \gamma\in\mathbb R).
\end{align*}
In this subsection, 
we refer to the consistency of tests $\mathcal T_n^\alpha$ and $\mathcal T_{1,n}^\beta$
 in Case A. 

First, let us consider the consistency of the test $\mathcal T_n^\alpha$
in Case A. 
Thus, we consider the following stochastic differential equation
\begin{align*}
X_t=
\begin{cases}
\displaystyle 
X_0-\int_0^t \beta(X_s-\gamma)\dd s+\alpha_1^* W_t,
\quad t\in[0,\tau_*^\alpha T),\\
\displaystyle 
X_{\tau_*^\alpha T}-\int_{\tau_*^\alpha T}^t \beta(X_s-\gamma)\dd s
+\alpha_2^* (W_t-W_{\tau_*^\alpha T}),
\quad t\in[\tau_*^\alpha T,T],\\
\end{cases}
\end{align*}
where
$\alpha_1^*=\alpha_0+\Dea c_1$
and 
$\alpha_2^*=\alpha_0+\Dea c_2$,
which hold \textbf{[E1]}.
Further, from Section \ref{sec5.3} and \ref{sec5.4}, 
\textbf{[E2]} and \textbf{[E3]} hold. 
Therefore the test $\mathcal T_n^\alpha$ is consistent by Proposition \ref{prop3}. 

Next, we investigate the consistency of the test $\mathcal T_{1,n}^\beta$ 
in Case A. 
Thus, we consider the following stochastic differential equation
\begin{align*}
X_t=
\begin{cases}
\displaystyle 
X_0-\int_0^t \beta_1^*(X_s-\gamma_1^*)\dd s+\alpha^* W_t,
\quad t\in[0,\tau_*^\beta T), \\
\displaystyle 
X_{\tau_*^\beta T}-\int_{\tau_*^\beta T}^t \beta_2^*(X_s-\gamma_2^*)\dd s
+\alpha^* (W_t-W_{\tau_*^\beta T}),
\quad t\in[\tau_*^\beta T,T], \\
\end{cases}
\end{align*}
where
\begin{align*}
{\boldsymbol\beta}_0
=\left(
\begin{array}{c}
\beta_0\\
\gamma_0
\end{array}
\right),\ 
{\boldsymbol d}_k
=\left(
\begin{array}{c}
d_{k,1}\\
d_{k,2}
\end{array}
\right),\ 
{\boldsymbol\beta}_k^*
=\left(
\begin{array}{c}
\beta_k^*\\
\gamma_k^*
\end{array}
\right)
={\boldsymbol\beta}_0
+
\Deb{\boldsymbol d}_k,
\end{align*}
which hold \textbf{[F1]}. Further,
\begin{align*}
\int_{\mathbb R}
\frac{1}{\alpha^*}
(-x+\gamma_0,\beta_0)\dd\mu_{(\beta_0,\gamma_0)}(x)
({\boldsymbol d}_1-{\boldsymbol d}_2)
=\left(0,\frac{\beta_0}{\alpha^*}\right)
\left(
\begin{array}{c}
d_{1,1}-d_{2,1}\\
d_{1,2}-d_{2,2}
\end{array}
\right)
=\frac{\beta_0}{\alpha^*}(d_{1,2}-d_{2,2}).
\end{align*}
Therefore, 
if $d_{1,2}-d_{2,2}\neq0$, that is, 
if $\gamma$ changes and  $\beta$ does not change, then
\textbf{[F3]} holds, and the test $\mathcal T_{1,n}^\beta$ is consistent
by Proposition \ref{prop4}.
However, 
when $\beta$ changes and  $\gamma$ does not change,
\textbf{[F3]} does not hold.

\subsection{Hyperbolic diffusion model}\label{sec5.7}
We consider the hyperbolic diffusion model
\begin{align}
\dd X_t=
\left(
\beta-\frac{\gamma X_t}{\sqrt{1+X_t^2}}
\right)\dd t
+\alpha\dd W_t,\quad
X_0=x_0.
\quad(\alpha>0,\beta\in\mathbb R, \gamma>|\beta|).
\label{eq5-5.999}
\end{align}
In this subsection, 
we mention 
the consistency of the tests 
$\mathcal T_n^\alpha$ and $\mathcal T_{1,n}^\beta$ in Case A
and
the fact that this model is an example in Case B.

Let  
$b(x,\boldsymbol\beta)=\beta-\dfrac{\gamma x}{\sqrt{1+x^2}}$
and 
$a(x,\alpha)=\alpha$.
By the same discussion as Ornstein-Uhlenbeck process (Section \ref{sec5.6}), 
the test $\mathcal T_n^\alpha$ is consistent in Case A.
Next, we investigate the consistency of the test $\mathcal T_{1,n}^\beta$ in Case A.
Thus, we consider the following stochastic differential equation
\begin{align*}
X_t=
\begin{cases}
\displaystyle 
X_0+\int_0^t \left(\beta_1^*-\frac{\gamma_1^* X_s}{\sqrt{1+X_s^2}}\right)\dd s
+\alpha^* W_t,
\quad t\in[0,\tau_*^\beta T), \\
\displaystyle 
X_{\tau_*^\beta T}
+\int_{\tau_*^\beta T}^t \left(\beta_2^*-\frac{\gamma_2^* X_s}{\sqrt{1+X_s^2}}\right)\dd s
+\alpha^* (W_t-W_{\tau_*^\beta T}),
\quad t\in[\tau_*^\beta T,T], \\
\end{cases}
\end{align*}
where
\begin{align*}
{\boldsymbol\beta}_0
=\left(
\begin{array}{c}
\beta_0\\
\gamma_0
\end{array}
\right),\ 
{\boldsymbol d}_k
=\left(
\begin{array}{c}
d_{k,1}\\
d_{k,2}
\end{array}
\right),\ 
{\boldsymbol\beta}_k^*
=\left(
\begin{array}{c}
\beta_k^*\\
\gamma_k^*
\end{array}
\right)
={\boldsymbol\beta}_0
+
\Deb{\boldsymbol d}_k,
\end{align*}
which hold \textbf{[F1]}.
The invariant density of the solution in \eqref{eq5-5.999} is  
\begin{align*}
\pi(x)=\frac{m(x)}{M},
\end{align*}
where
\begin{align*}
m(x)=\exp
\left(
\frac{2}{\alpha^2}\left(\beta x-\gamma\sqrt{1+x^2}\right)
\right),\ 
M=\int_{\mathbb R}m(x)\dd x.
\end{align*}
Now, we have
\begin{align*}
\int_{\mathbb R}\partial_{\beta}b(x,\boldsymbol\beta)\pi(x)\dd x
=\int_{\mathbb R}\pi(x)\dd x=1,
\end{align*}
\begin{align*}
\int_{\mathbb R}\partial_{\gamma}b(x,\boldsymbol\beta)\pi(x)\dd x
&=-\frac1M\int_{\mathbb R}
\frac{x}{\sqrt{1+x^2}}
\exp
\left(
\frac{2}{\alpha^2}\left(\beta x-\gamma\sqrt{1+x^2}\right)
\right)
\dd x\\
&=\frac{\alpha^2}{2\gamma}\left(\frac1M
\int_{\mathbb R}
\frac{2}{\alpha^2}\left(\beta-\frac{\gamma x}{\sqrt{1+x^2}}\right)
\exp
\left(
\frac{2}{\alpha^2}\left(\beta x-\gamma\sqrt{1+x^2}\right)
\right)
\dd x
-\frac{2\beta}{\alpha^2}\right)
\end{align*}
and
\begin{align*}
&\int_{\mathbb R}
\frac{2}{\alpha^2}\left(\beta-\frac{\gamma x}{\sqrt{1+x^2}}\right)
\exp
\left(
\frac{2}{\alpha^2}\left(\beta x-\gamma\sqrt{1+x^2}\right)
\right)
\dd x
\\
&=
\left[
\exp
\left(
\frac{2}{\alpha^2}\left(\beta x-\gamma\sqrt{1+x^2}\right)
\right)
\right]_{x=-\infty}^{x=\infty}\\
&=
\left[
\exp
\left(
-\frac{2\gamma}{\alpha^2}\sqrt{1+x^2}
\left(1-\frac{\beta}{\gamma}\frac{x}{\sqrt{1+x^2}}\right)
\right)
\right]_{x=-\infty}^{x=\infty}\\
&=0.
\end{align*}
Hence, we have 
\begin{align*}
\int_{\mathbb R}\partial_{\gamma}b(x,\boldsymbol\beta)\pi(x)\dd x
=-\frac{\beta}{\gamma}.
\end{align*}
From the above, we obtain
\begin{align}
\int_{\mathbb R}
\frac{1}{\alpha^*}
\left(1,-\frac{\beta_0}{\gamma_0}\right)\dd\mu_{(\beta_0,\gamma_0)}(x)
({\boldsymbol d}_1-{\boldsymbol d}_2)
&=\frac{1}{\alpha^*}\left(1,-\frac{\beta_0}{\gamma_0}\right)
\left(
\begin{array}{c}
d_{1,1}-d_{2,1}\\
d_{1,2}-d_{2,2}
\end{array}
\right)\nonumber\\
&=
\frac{1}{\alpha^*}
\left(
(d_{1,1}-d_{2,1})
-\frac{\beta_0}{\gamma_0}(d_{1,2}-d_{2,2})
\right).\label{eq5-999}
\end{align}
In the following cases, 
\textbf{[F3]} holds because equation \eqref{eq5-999} does not equal $0$:
\begin{enumerate}
\renewcommand{\labelenumi}{(\arabic{enumi})}
\item
$\beta$ changes and $\gamma$ does not change, 
\item 
$\beta_0\neq0$, 
$\gamma$ changes and $\beta$ does not change,
\item 
$\beta_0>0$,  
$d_{1,1}-d_{2,1}<0$ (resp. $>0$) and $d_{1,2}-d_{2,2}>0$ (resp. $<0$),
\item 
$\beta_0<0$, 
$d_{1,1}-d_{2,1}<0$ (resp. $>0$) and $d_{1,2}-d_{2,2}<0$ (resp. $>0$).
\end{enumerate}

Finally, in Case B, 
we confirm that the hyperbolic diffusion model satisfies the 
assumptions 
\textbf{[B1-I]}, \textbf{[B2-I]}, \textbf{[B1-II]} and \textbf{[B2-II]}.
Note that it was mentioned in Remark \ref{rmk1} 
that \textbf{[B1-I]} and \textbf{[B2-I]}(a), (b) hold.
Thus, in the following, we verify that 
\textbf{[B2-I]}(c), \textbf{[B1-II]} and \textbf{[B2-II]} hold. 
From the proof of Lemma 1 of Kessler (1997), 
we can express $Q(x,\theta)=L_\theta^2f(y|x)|_{y=x}$,
where $f(y|x)=(y-x)^2$,  
\begin{align*}
L_\theta f(y|x)
&=b(y,\boldsymbol\beta)\partial_y f(y|x)+\frac{A(y,\alpha)}{2}\partial_y^2 f(y|x),
\\
L_\theta^2 f(y|x)
&=L_\theta[L_\theta f](y|x).
\end{align*}
The specific calculation is as follows.
\begin{align*}
L_\theta f(y|x)
&=
2b(y,\boldsymbol\beta)(y-x)+\alpha^2,
\\
\partial_y [L_\theta f](y|x)
&=
2\bigl(
\partial_y b(y,\boldsymbol\beta)(y-x)+b(y,\boldsymbol\beta)
\bigr),
\\
\partial_y^2 [L_\theta f](y|x)
&=
2\bigl(
\partial_y^2 b(y,\boldsymbol\beta)(y-x)+2\partial_y b(y,\boldsymbol\beta)
\bigr),
\\
L_\theta^2 f(y|x)
&=
2b(y,\beta)
\bigl(
\partial_y b(y,\boldsymbol\beta)(y-x)+b(y,\boldsymbol\beta)
\bigr)
+\alpha^2
\bigl(
\partial_y^2 b(y,\boldsymbol\beta)(y-x)+2\partial_y b(y,\boldsymbol\beta)
\bigr),
\end{align*}
and 
$Q(x,\theta)=2\bigl(
b(x,\boldsymbol\beta)^2+\alpha^2\partial_x b(x,\boldsymbol\beta)
\bigr)$.
$b(x,\boldsymbol\beta)$ and $\partial_x b(x,\boldsymbol\beta)$ are bounded, 
so we have $\sup_{x,\theta}|Q(x,\theta)|<C$.
Thus, \textbf{[B2-I]}(c) holds.
Moreover, we have
\begin{align*}
\Gamma^\beta(x,\alpha,\boldsymbol\beta_1,\boldsymbol\beta_2)
&=
\frac{1}{\alpha^2}
\left[
\left(\beta_1-\frac{\gamma_1x}{\sqrt{1+x^2}}\right)
-
\left(\beta_2-\frac{\gamma_2x}{\sqrt{1+x^2}}\right)
\right]^2
\\
&=
\frac{1}{\alpha^2}
\left(
(\beta_1-\beta_2)
-
(\gamma_1-\gamma_2)
\frac{x}{\sqrt{1+x^2}}
\right)^2,
\end{align*}
where $\boldsymbol\beta_k=(\beta_k,\gamma_k)^\TT$.
Therefore, since $-1<\dfrac{x}{\sqrt{1+x^2}}<1$ for $x\in\mathbb R$,
we have 
$\sup_x|\Gamma^\beta(x,\alpha^*,\boldsymbol\beta_1^*,\boldsymbol\beta_2^*)|>0$ 
in the following cases, 
and \textbf{[B1-II]} holds. 
\begin{enumerate}
\renewcommand{\labelenumi}{(\arabic{enumi})}
\item
$\gamma_1^*=\gamma_2^*$,
\item 
$\gamma_1^*\neq\gamma_2^*$ and 
$\beta_1^*-\beta_2^*<-(\gamma_1^*-\gamma_2^*)$,
\item 
$\gamma_1^*\neq\gamma_2^*$ and 
$\beta_1^*-\beta_2^*>\gamma_1^*-\gamma_2^*$.
\end{enumerate}
Furthermore, we see, from 
\begin{align*}
\partial_\alpha \Gamma^\beta(x,\alpha,\boldsymbol\beta_1,\boldsymbol\beta_2)
&=
\frac{-2}{\alpha^3}
\left(
(\beta_1-\beta_2)
-
(\gamma_1-\gamma_2)
\frac{x}{\sqrt{1+x^2}}
\right)^2,
\\
\partial_{\beta_1} \Gamma^\beta(x,\alpha,\boldsymbol\beta_1,\boldsymbol\beta_2)
&=
\frac{2}{\alpha^2}
\left(
(\beta_1-\beta_2)
-
(\gamma_1-\gamma_2)
\frac{x}{\sqrt{1+x^2}}
\right),
\\
\partial_{\gamma_1} \Gamma^\beta(x,\alpha,\boldsymbol\beta_1,\boldsymbol\beta_2)
&=
\frac{-2x}{\alpha^2\sqrt{1+x^2}}
\left(
(\beta_1-\beta_2)
-
(\gamma_1-\gamma_2)
\frac{x}{\sqrt{1+x^2}}
\right),
\end{align*}
$\partial_{\beta_2} \Gamma^\beta(x,\alpha,\boldsymbol\beta_1,\boldsymbol\beta_2)
=-\partial_{\beta_1} \Gamma^\beta(x,\alpha,\boldsymbol\beta_1,\boldsymbol\beta_2)$,
$\partial_{\gamma_2} \Gamma^\beta(x,\alpha,\boldsymbol\beta_1,\boldsymbol\beta_2)
=-\partial_{\gamma_1} \Gamma^\beta(x,\alpha,\boldsymbol\beta_1,\boldsymbol\beta_2)$,
\begin{align*}
\frac{1}{\alpha^2}
\partial_{\beta}b(x,\boldsymbol\beta)
\Bigl(
b(x,\boldsymbol\beta_1)-b(x,\boldsymbol\beta_2)
\Bigr)^2
&=
\frac{1}{\alpha^2}
\left(
(\beta_1-\beta_2)
-
(\gamma_1-\gamma_2)
\frac{x}{\sqrt{1+x^2}}
\right),
\\
\frac{1}{\alpha^2}
\partial_{\gamma}b(x,\boldsymbol\beta)
\Bigl(
b(x,\boldsymbol\beta_1)-b(x,\boldsymbol\beta_2)
\Bigr)^2
&=
\frac{-x}{\alpha^2\sqrt{1+x^2}}
\left(
(\beta_1-\beta_2)
-
(\gamma_1-\gamma_2)
\frac{x}{\sqrt{1+x^2}}
\right)
\end{align*}
and
the boundedness of $\dfrac{x}{\sqrt{1+x^2}}$,
that 
\begin{align*}
&\sup_{x,\alpha,\boldsymbol\beta_k}
\left|
\partial_{(\alpha,\boldsymbol\beta_1,\boldsymbol\beta_2)} 
\Gamma^\beta(x,\alpha,\boldsymbol\beta_1,\boldsymbol\beta_2)
\right|
<C, 
\\
&\sup_{x,\alpha,\boldsymbol\beta_k}
\left|
\frac{1}{\alpha^2}
\partial_{\boldsymbol\beta}b(x,\boldsymbol\beta)
\Bigl(
b(x,\boldsymbol\beta_1)-b(x,\boldsymbol\beta_2)
\Bigr)^2
\right|<C
\end{align*}
and
\textbf{[B2-II]} holds.



\section{Simulations}\label{sec6}


\subsection{Case A}

In this subsection, we consider the $1$-dimensional Ornstein-Uhlenbeck process:
\begin{align*}
\dd X_t=-\beta(X_t-\gamma)\dd t+\alpha\dd W_t,\quad X_0=x_0,
\end{align*}
where $\alpha, \beta>0, \gamma\in\mathbb R$.

In order to check the asymptotic behavior of the estimator $\hat{\tau}_n^\alpha$
in Case A,
we consider the following stochastic differential equation
\begin{align*}
X_t=
\begin{cases}
\displaystyle 
X_0-\int_0^t \beta^* (X_s-\gamma^* )\dd s+\alpha_1^* W_t,
\quad t\in[0,\tau_*^\alpha T), \\
\displaystyle 
X_{\tau_*^\alpha T}-\int_{\tau_*^\alpha T}^t \beta^* (X_s-\gamma^*)\dd s
+\alpha_2^* (W_t-W_{\tau_*^\alpha T}),
\quad t\in[\tau_*^\alpha T,T], \\
\end{cases}
\end{align*}
where
$x_0=2$, 
$\beta^* =1$, $\gamma^* =2$
and
$\tau_{*}^{\alpha} = 0.5$.
The number of iteration is 1000.
We set that the sample size of the data $\{\Xt\}_{i=0}^n$ is $n=10^6$ and
$h_n=n^{-2/3}=10^{-4}$.
Note that 
$T=n h_n=n^{1/3} =10^2$, 
$n h_n^2=n^{-1/3}=10^{-2}$, 
$\Dea=n^{-0.35}\approx 0.0079$, 
$n\Dea^2=n^{0.3}\approx 63.1$, 
$T\Dea=n^{-1/60}\approx 0.79$, 
$\alpha_0=0.1$, 
$\alpha_1^*=\alpha_0+ n^{-0.35}$
and
$\alpha_2^*=\alpha_0=0.1$. 
The existence of a change point in the intervals  
$[1/4T,T]$ and $[0, 3/4T]$ was investigated using the method of 
Tonaki et al. (2020).
In all 1000 iterations, the change point was detected. 
Therefore, we estimated $\hat{\alpha}_1$ and $\hat{\alpha}_2$ with 
$\underline\tau^\alpha=1/4$ and $\overline\tau^\alpha=3/4$, respectively.
The estimates of $\alpha_1^*$, $\alpha_2^*$ and $\tau_*^\alpha$ 
are reported in Table \ref{tab1}.
Moreover, one has that
for $v\in\mathbb R$,
\begin{align*}
e_\alpha
&=\lim_{n\to\infty}\Dea^{-1}(\alpha_1^*-\alpha_2^*)=1,\\
\mathcal J_\alpha
&=\frac12e_\alpha^\TT 
\int_{\mathbb R^d}
\Xi^\alpha(x,\alpha_0)
\dd\mu_{\alpha_0}(x)
e_\alpha
=\frac{2}{\alpha_0^2},
\\
\mathbb F(v)
&=-2\mathcal J_\alpha^{1/2}\mathcal W(v)+\mathcal J_\alpha|v| 
=-2\left(\mathcal J_\alpha^{1/2}\mathcal W(v)-\frac{1}{2}|\mathcal J_\alpha v|\right)
\overset{d}{=}
-2\left(\mathcal W(\mathcal J_\alpha v)-\frac{1}{2}|\mathcal J_\alpha v|\right).
\end{align*}
It follows from Theorem \ref{th1} that
\begin{align*}
n\Dea^2(\hat\tau_n^\alpha-\tau_*^\alpha)
\dto
\underset{v\in\mathbb R}{\mathrm{argmin}}\,\mathbb F(v).
\end{align*}

\noindent
For $v\in \mathbb{R}$,  let $G(v) =  \mathcal W(v) - \frac{1}{2} |v|$
and $\hat{\eta}  = \inf \left\{ \eta \in \mathbb{R}  \ |  \ G(\eta) = \sup_{v\in \mathbb{R}} G(v) \right\}$.
For the probability density function of the distribution of $\hat{\eta}$,
see Lemma 1.6.3 of Cs{\"o}rg{\"o} and Horv{\'a}th (1997).   
From Figure \ref{caseA-alpha}, 
we can see that the distribution of the estimator almost 
corresponds with the asymptotic distribution in Theorem \ref{th1}
and 
the estimators
have good performance.

\begin{table}[h]
\caption{
Mean and standard deviation of the estimators under 
$n=10^6$, $T=100$, $h=10^{-4}$, $\tau_*^\alpha=0.5$,
$\alpha_1^*\approx 0.1079$, $\alpha_2^*=0.1$ 
in Case A.
}

\

\begin{center}
\begin{tabular*}{.4\textwidth}{@{\extracolsep{\fill}}ccc}\hline
$\hat\alpha_1$ & $\hat\alpha_2$ & $\hat\tau_n^\alpha$ 
\rule[0mm]{0cm}{4mm}\\\hline 
0.10795 & 0.10001 & 0.49999  \rule[0mm]{0cm}{4mm}\\ 
(0.00016) & (0.00014) & (0.00041) \\\hline
\end{tabular*}
\end{center}
\label{tab1}
\end{table}

\begin{figure}[h] 
\begin{center}
\includegraphics[width=6cm,pagebox=cropbox,clip]{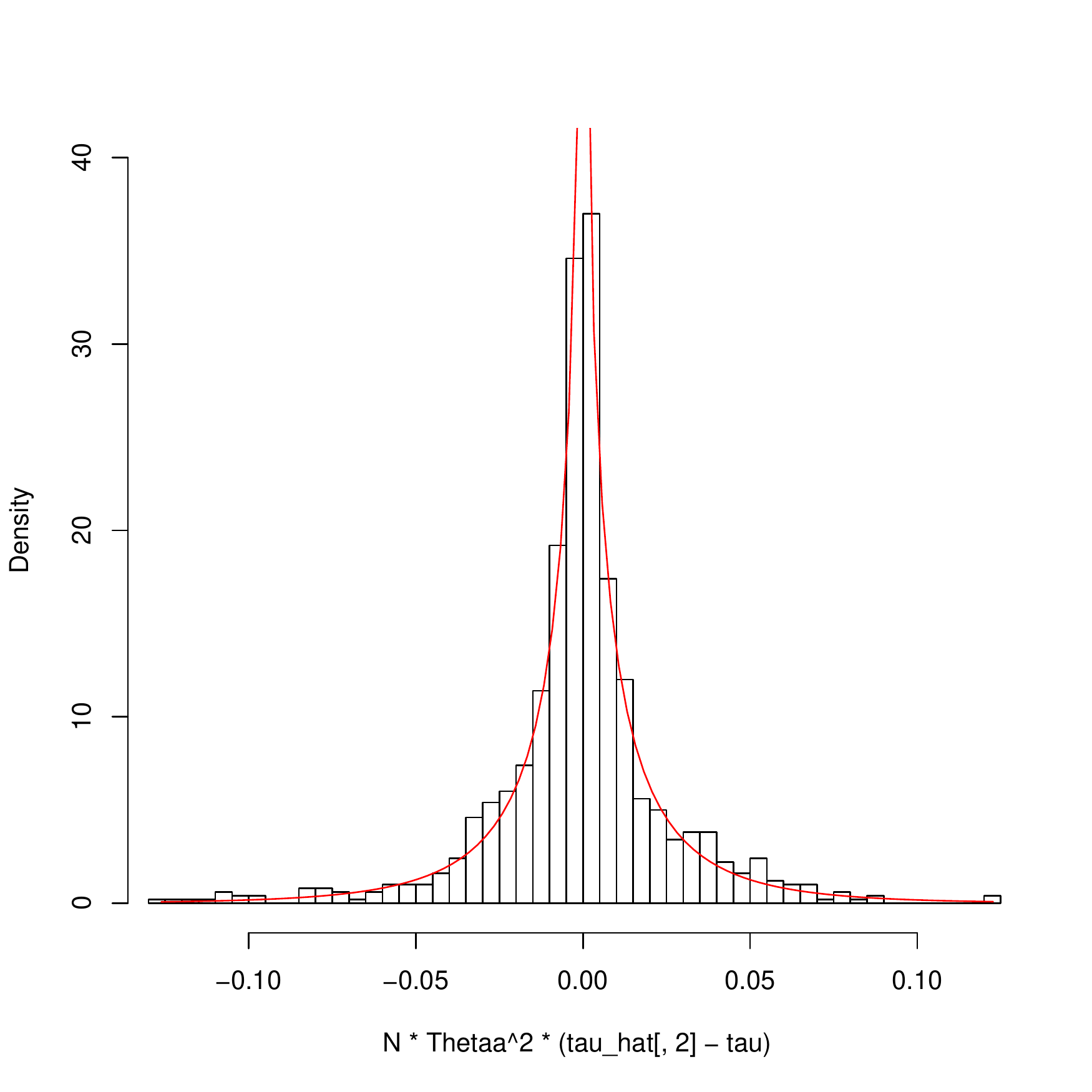}
\includegraphics[width=6cm,pagebox=cropbox,clip]{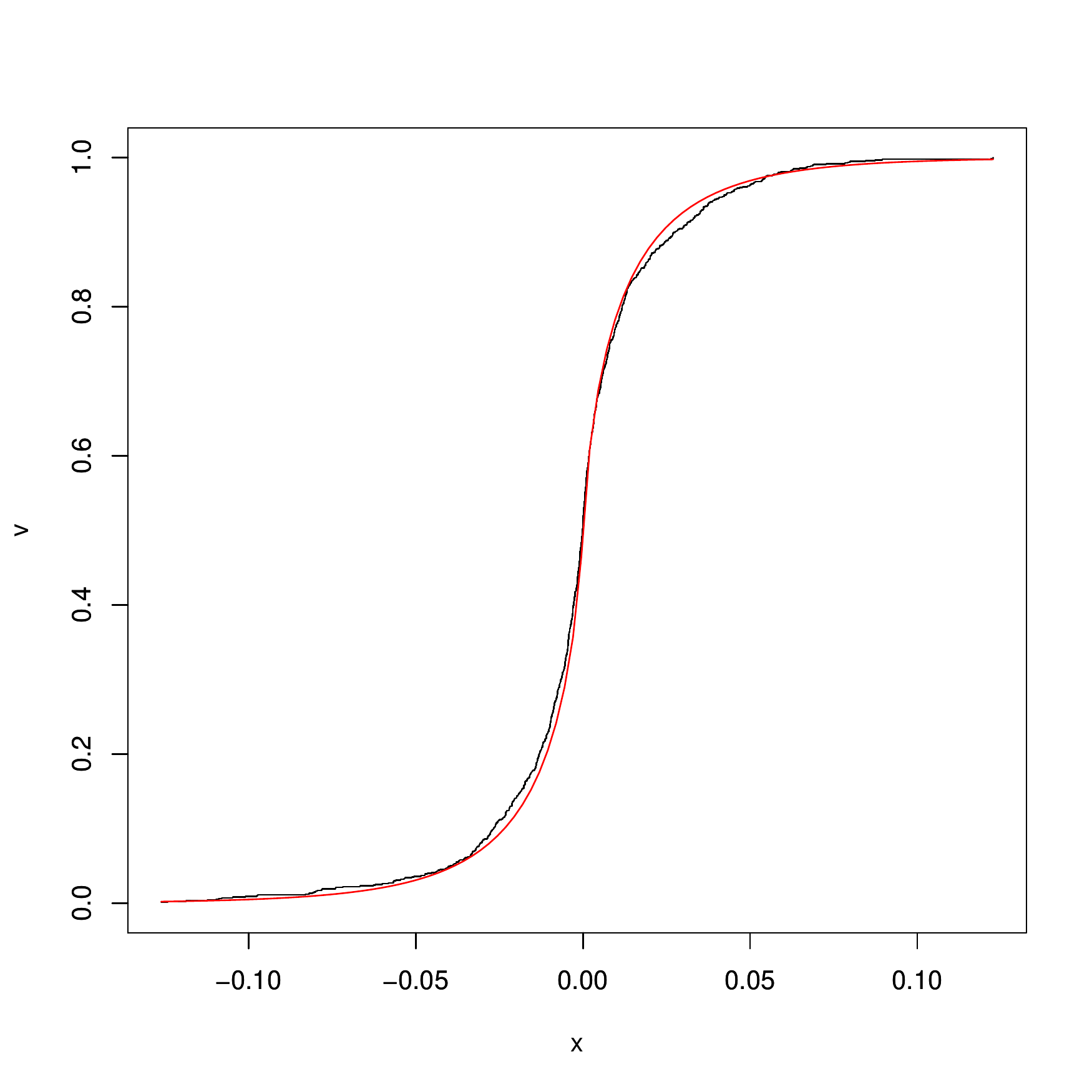}
\caption{
The figure on the left is the histogram of 
$n\Dea^2(\hat\tau_n^\alpha-\tau_*^\alpha)$ (black line) and the theoretical density function (red line).
The figure on the right is the empirical distribution function of 
$n\Dea^2(\hat\tau_n^\alpha-\tau_*^\alpha)$ (black line) and the theoretical distribution function (red line).
}\label{caseA-alpha}
\end{center}
\end{figure}

For the simulations of the estimator $\hat{\tau}_n^{\beta}$ 
in Case A,
we treat the stochastic differential equation as follows.
\begin{align*}
X_t=
\begin{cases}
\displaystyle 
X_0-\int_0^t \beta^*(X_s-\gamma_1^*)\dd s+\alpha^* W_t,
\quad t\in[0,\tau_*^\beta T), \\
\displaystyle 
X_{\tau_*^\beta T}-\int_{\tau_*^\beta T}^t \beta^*(X_s-\gamma_2^*)\dd s
+\alpha^* (W_t-W_{\tau_*^\beta T}),
\quad t\in[\tau_*^\beta T,T], \\
\end{cases}
\end{align*}
where
$x_0=5$, 
$\alpha^*=0.5$,  
$\beta^* =2.5$ 
and
$\tau_{*}^{\beta} = 0.5$.
The number of iteration is 1000.
We set that the sample size of the data $\{\Xt\}_{i=0}^n$ is $n=10^6$ and
$h_n=n^{-4/7} \approx 3.73 \times 10^{-4}$.
Note that 
$T=n h_n =n^{3/7} \approx 3.73 \times 10^{2}$,
$n h_n^2 = n^{-1/7} \approx 0.14$,
$\Deb=n^{-1/8} \approx 0.18$, 
$T\Deb^2 =n^{5/28} \approx 11.8$, 
$T\Deb^4 =n^{-1/14} \approx 0.37$
$\gamma_1^*=5+\Deb$
and 
$\gamma_2^*=5$.
In all iterations, the change point was detected in the intervals $[1/4T,T]$ 
and $[0, 3/4T]$.
Therefore, we estimated $\hat{\boldsymbol\beta}_1$ and $\hat{\boldsymbol\beta}_2$ 
with $\underline\tau^\beta=1/4$ and $\overline\tau^\beta=3/4$, respectively.
The estimates of $\boldsymbol\beta_1^*$, $\boldsymbol\beta_2^*$ and $\tau_*^\beta$ 
are reported in Table \ref{tab2}.
Noting that 
$\boldsymbol \beta_1^*
=(\beta^*, \gamma_1^*)^\TT$, 
$\boldsymbol \beta_2^*
=(\beta^*, \gamma_2^*)^\TT$, 
$\boldsymbol \beta_0
=(\beta^*, \gamma_2^*)^\TT$, 
$\mu_{(\alpha^*,\boldsymbol\beta_0)}\sim
N\bigl(\gamma_2^*,\frac{(\alpha^*)^2}{2\beta^*}\bigr)$
and
\begin{align*}
\Xi^\beta(x,\alpha,\boldsymbol\beta)
=
\frac{1}{\alpha^2}
\left(
\begin{array}{cc}
(x-\gamma)^2 & -\beta(x-\gamma) \\
-\beta(x-\gamma) & \beta^2
\end{array}
\right),
\end{align*}
we obtain that for $v\in\mathbb R$,
\begin{align*}
e_\beta
&=\lim_{n\to\infty}\Deb^{-1}(\boldsymbol\beta_1^*-\boldsymbol\beta_2^*) 
=
\left(
\begin{array}{c}
0\\
1
\end{array}
\right)
,\\
\mathcal J_\beta
&=e_\beta^\TT 
\int_{\mathbb R^d}
\Xi^\beta(x,\alpha^*,\boldsymbol\beta_0)
\dd\mu_{(\alpha^*,\boldsymbol\beta_0)}(x)e_\beta
=
\frac{1}{(\alpha^*)^2}(0,1)
\left(
\begin{array}{cc}
\frac{(\alpha^*)^2}{2\beta^*} & 0 \\
0 & (\beta^*)^2
\end{array}
\right)
\left(
\begin{array}{c}
0\\
1
\end{array}
\right)
=\left(\frac{\beta^*}{\alpha^*}\right)^2
,\\
\mathbb G(v)
&=-2\mathcal J_\beta^{1/2}\mathcal W(v)+\mathcal J_\beta|v| 
\overset{d}{=}
-2\left(\mathcal W(\mathcal J_\beta v)-\frac{1}{2}|\mathcal J_\beta v|\right).
\end{align*}
By Theorem \ref{th3}, we have that
\begin{align*}
T\Deb^2(\hat\tau_n^\beta-\tau_*^\beta)
\dto
\underset{v\in\mathbb R}{\mathrm{argmin}}\,
\mathbb G(v).
\end{align*}
Since Figure \ref{caseA-beta} shows 
that the distribution of the estimator is similar to the asymptotic distribution in Theorem \ref{th3},
the estimator has good behavior.

\begin{table}[h]
\caption{
Mean and standard deviation of the estimators under 
$n=10^6$, $T\approx 3.73\times10^2$, $h\approx 3.73\times 10^{-4}$, 
$\tau_*^\beta=0.5$, $\beta^*=2.5$, 
$\gamma_1^*\approx 5.1778$, $\gamma_2^*=5$ 
in Case A.
}
\begin{center}
\begin{tabular*}{.7\textwidth}{@{\extracolsep{\fill}}ccccc}\hline
 $\hat\beta_1$ & $\hat\gamma_1$ & $\hat\beta_2$ & $\hat\gamma_2$ & $\hat\tau_n^\beta$ 
\rule[0mm]{0cm}{4mm}\\\hline 
2.54981 & 5.17730 & 2.54314 & 4.99981 & 0.49797 \rule[0mm]{0cm}{4mm}\\ 
(0.22565) & (0.02047) & (0.24675) & (0.02030) & (0.01341) \\\hline
\end{tabular*}
\end{center}
\label{tab2}
\end{table}

\begin{figure}[h] 
\begin{center}
\includegraphics[width=6cm,pagebox=cropbox,clip]{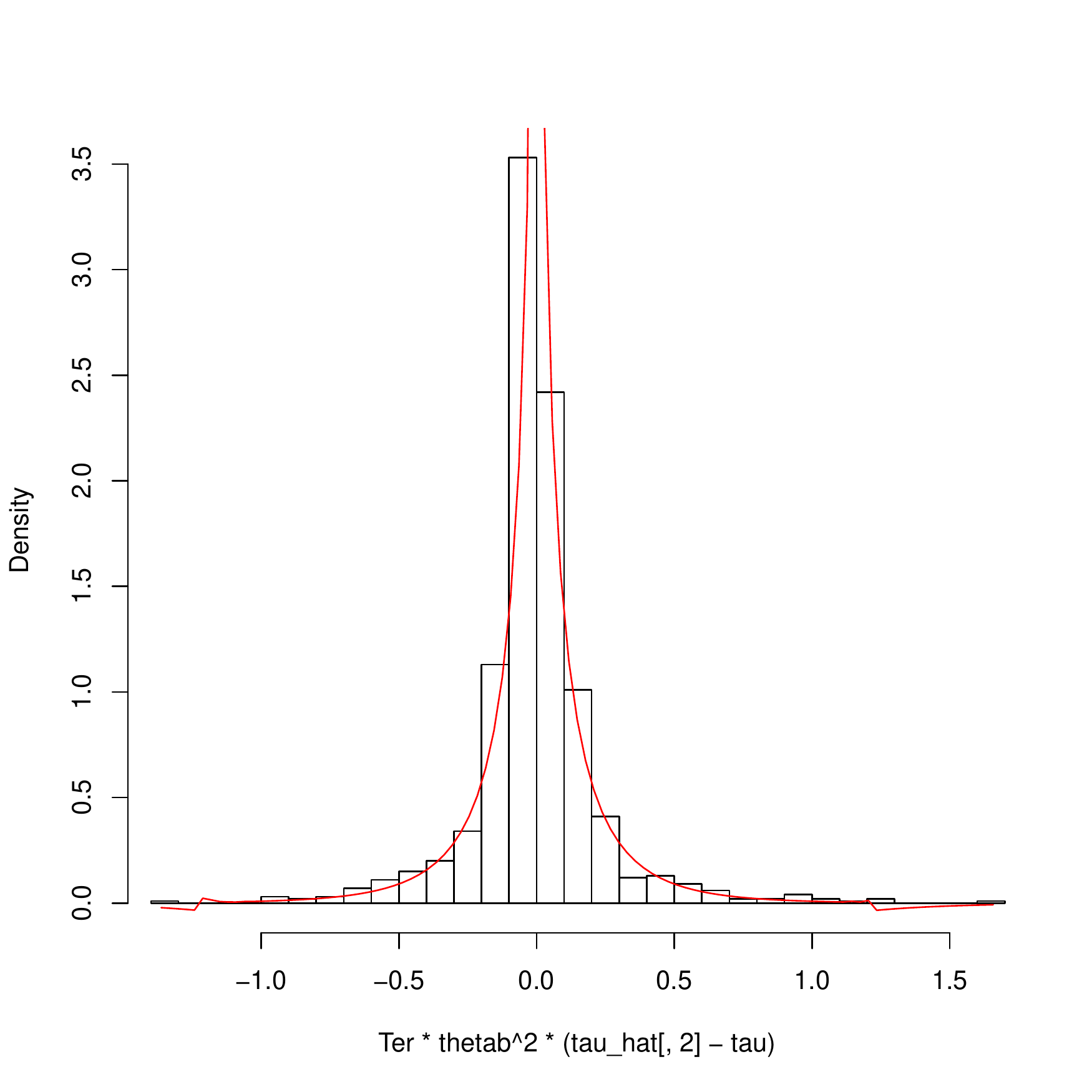}
\includegraphics[width=6cm,pagebox=cropbox,clip]{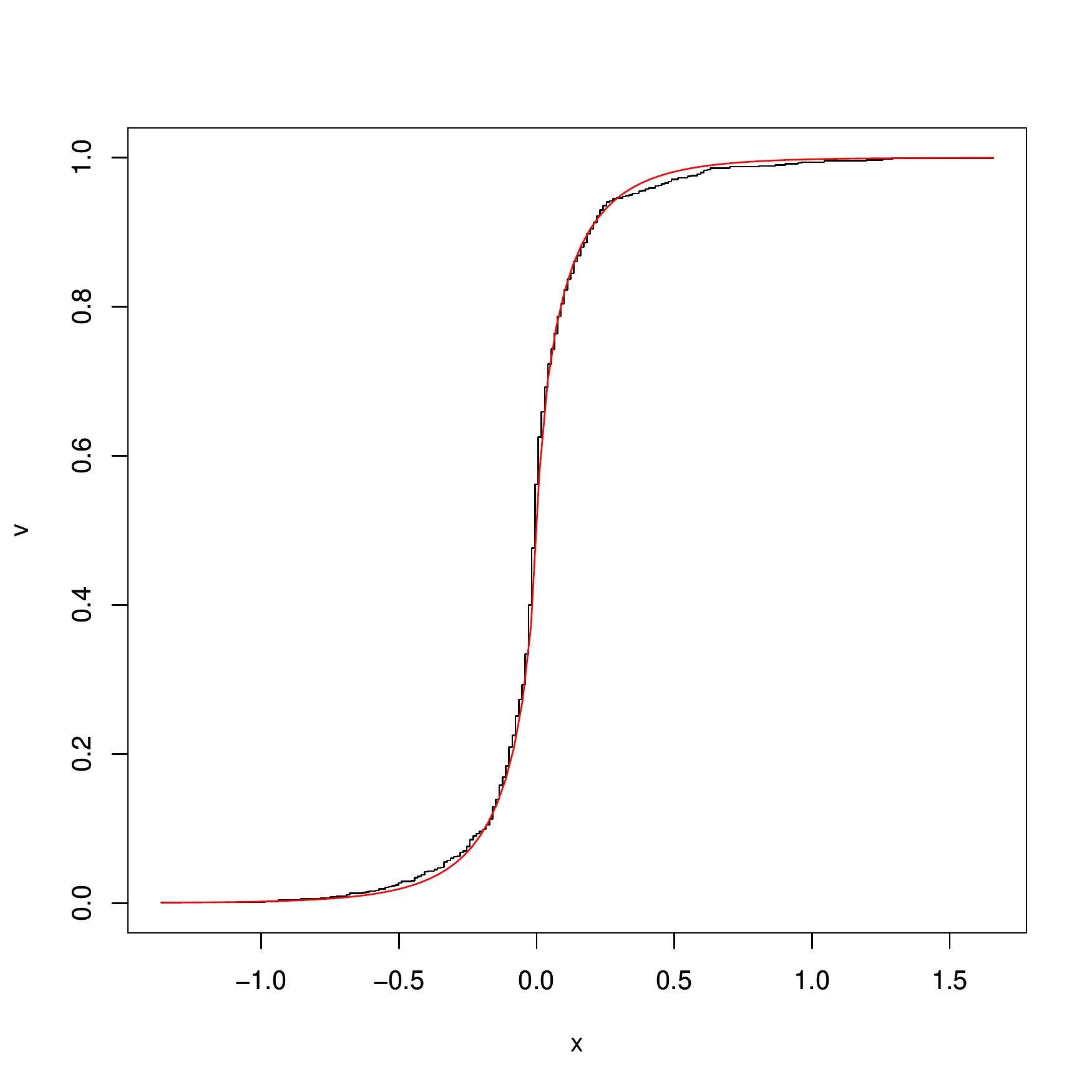}
\caption{
The figure on the left is the histogram of 
$T\Deb^2(\hat\tau_n^\beta-\tau_*^\beta)$ (black line) and the theoretical density function (red line).
The figure on the right is the empirical distribution function of 
$T\Deb^2(\hat\tau_n^\beta-\tau_*^\beta)$ (black line) and the theoretical distribution function (red line).
}\label{caseA-beta}
\end{center}
\end{figure}



\subsection{Case  B}

In this subsection, we consider the hyperbolic diffusion model
\begin{align*}
\dd X_t =
\left(
\beta-\frac{\gamma X_t}{\sqrt{1+X_t^2}}
\right)\dd t
+\alpha\dd W_t,\quad
X_0 = x_0.
\quad(\alpha>0,\beta\in\mathbb R, \gamma>|\beta|).
\end{align*}


For simulations of the estimator $\hat{\tau}_n^{\alpha}$ in Case B,
we study the following stochastic differential equation
\begin{align*}
X_t=
\begin{cases}
\displaystyle 
X_0+\int_0^t \left(\beta^*-\frac{\gamma^* X_s}{\sqrt{1+X_s^2}}\right)\dd s
+\alpha_1^* W_t,
\quad t\in[0,\tau_*^\beta T), \\
\displaystyle 
X_{\tau_*^\beta T}
+\int_{\tau_*^\beta T}^t \left(\beta^*-\frac{\gamma^* X_s}{\sqrt{1+X_s^2}}\right)\dd s
+\alpha_2^* (W_t-W_{\tau_*^\beta T}),
\quad t\in[\tau_*^\beta T,T], \\
\end{cases}
\end{align*}
where
$x_0=2$, 
$\beta^* =0$, $\gamma^* = 1$, 
$\tau_{*}^{\alpha} = 0.5$,
$\alpha_1^*=1$, 
$\alpha_2^*=2$.
The number of iteration is 1000.
We set that the sample size of the data $\{\Xt\}_{i=0}^n$ is $n=10^5$ or $10^6$ and
$h_n=n^{-2/3}$.
The existence of a change point in the intervals $[1/4T,T]$ and $[0, 3/4T]$ was investigated and the change point was detected. 
Therefore, we estimated $\hat{\alpha}_1$ and $\hat{\alpha}_2$ with $\underline\tau^\alpha=1/4$ and $\overline\tau^\alpha=3/4$, respectively.
The estimates of $\alpha_1^*$, $\alpha_2^*$ and $\tau_*^\alpha$ 
are reported in Table \ref{tab3}.
By Theorem \ref{th2}, one has that 
\begin{align*}
n(\hat\tau_n^\alpha-\tau_*^\alpha)=O_p(1).
\end{align*}
Since Figure \ref{caseB-alpha} shows 
that $n(\hat\tau_n^\alpha-\tau_*^\alpha)$ does not diverges  when increasing from $n = 10^5$ to $n = 10^6$,
it seems that 
$n(\hat\tau_n^\alpha-\tau_*^\alpha)$ is $O_p(1)$ in this example.

\begin{table}[h]
\caption{
Mean and standard deviation of the estimators under 
$\tau_*^\alpha=0.5$, $\alpha_1^*=1$, $\alpha_2^*=2$ 
in Case B.
}
\begin{center}
\begin{tabular*}{.7\textwidth}{@{\extracolsep{\fill}}cccccc}\hline
$n$ & $T$ & $h$ & $\hat\alpha_1$ & $\hat\alpha_2$ & $\hat\tau_n^\alpha$ 
\rule[0mm]{0cm}{4mm}\\\hline 
$10^5$ & $46.4$ & $4.64\times10^{-4}$ & 0.99976 & 2.00044 & 0.49521  
\rule[0mm]{0cm}{4mm}\\ 
&&& (0.00471) & (0.00927) & (0.00018) \\
$10^6$ & $10^2$ & $10^{-4}$ &  1.00003 & 1.99998 & 0.49966   \rule[0mm]{0cm}{4mm}\\ 
&&& (0.00143) & (0.00287) & (0.00012) \\\hline
\end{tabular*}
\end{center}
\label{tab3}
\end{table}

\begin{figure}[h] 
\begin{center}
\includegraphics[width=6cm,pagebox=cropbox,clip]{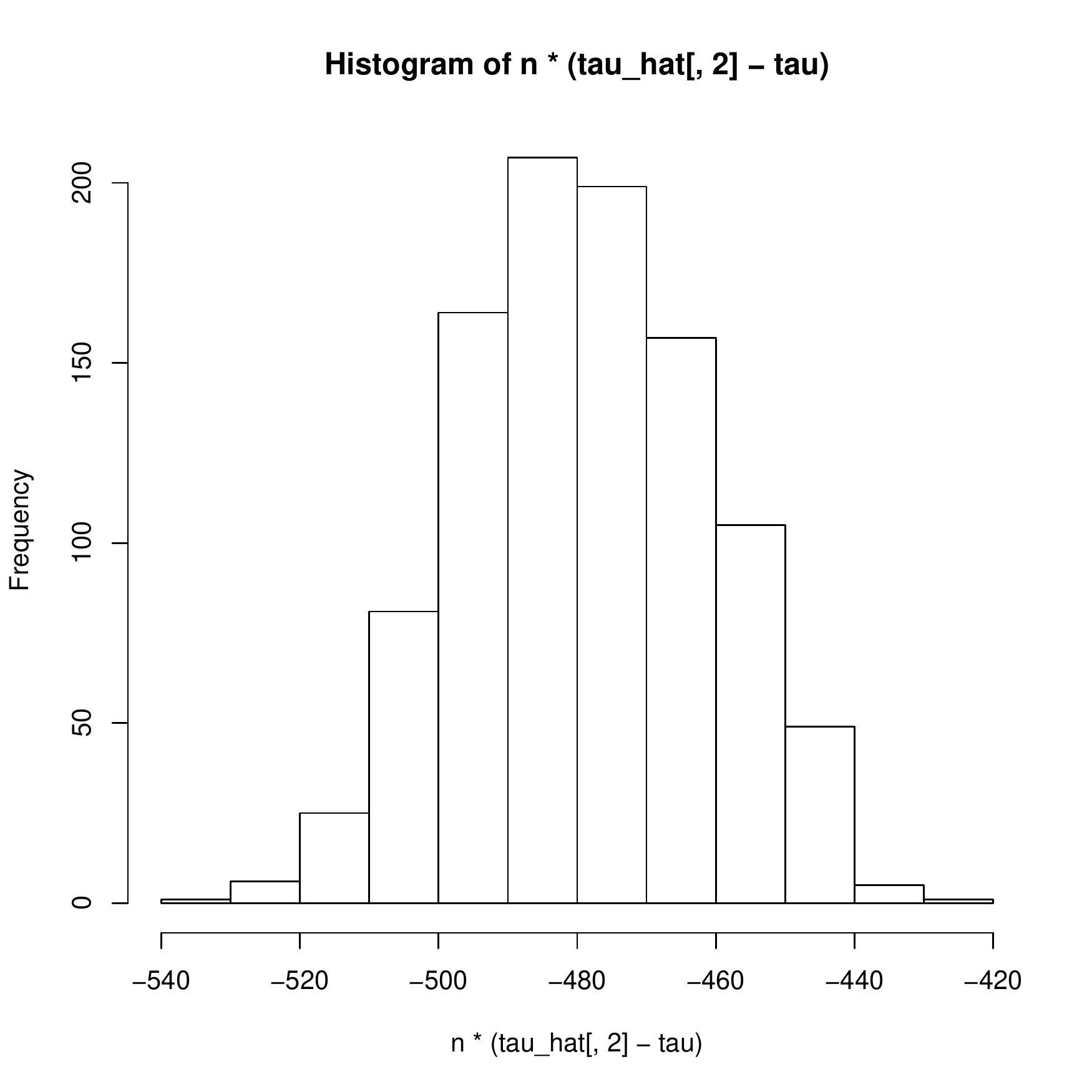}
\includegraphics[width=6cm,pagebox=cropbox,clip]{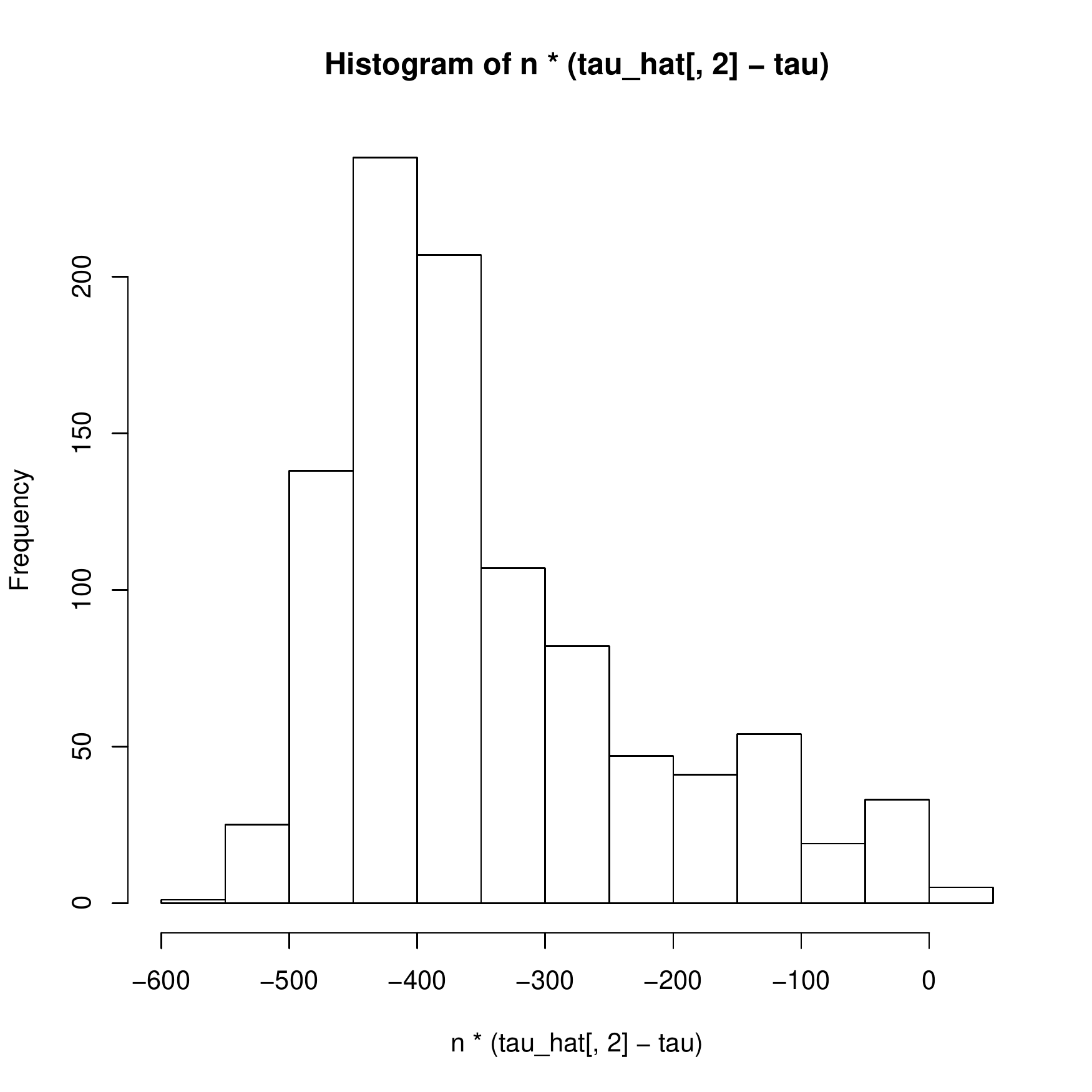}
\caption{The figure on the left is the histogram of 
$n(\hat\tau_n^\alpha-\tau_*^\alpha)$ with $n = 10^5$.
The figure on the right is the histogram of 
$n(\hat\tau_n^\alpha-\tau_*^\alpha)$ with $n = 10^6$.
}\label{caseB-alpha}
\end{center}
\end{figure}


In order to investigate the asymptotic performance of the estimator $\hat{\tau}_n^{\beta}$ in Case B,
we consider the stochastic differential equation as follows.
\begin{align*}
X_t=
\begin{cases}
\displaystyle 
X_0+\int_0^t \left(\beta_1^*-\frac{\gamma^* X_s}{\sqrt{1+X_s^2}}\right)\dd s
+\alpha^* W_t,
\quad t\in[0,\tau_*^\beta T), \\
\displaystyle 
X_{\tau_*^\beta T}
+\int_{\tau_*^\beta T}^t \left(\beta_2^*-\frac{\gamma^* X_s}{\sqrt{1+X_s^2}}\right)\dd s
+\alpha^* (W_t-W_{\tau_*^\beta T}),
\quad t\in[\tau_*^\beta T,T], \\
\end{cases}
\end{align*}
where
$x_0=0.25$, 
$\alpha^*=0.2$,  
$\gamma^* =1.2$, 
$\tau_{*}^{\beta} = 0.5$,  
$\beta_1^*=0.25$, 
$\beta_2^*=-0.25$.
The number of iteration is 1000.
We set that the sample size of the data $\{\Xt\}_{i=0}^n$ is $n=10^6$ or $10^7$ and
$h_n=n^{-4/7}$.
In all iterations, the change point was detected in the intervals $[1/4T,T]$ and 
$[0, 3/4T]$.
Therefore, we estimated $\hat{\beta}_1$ and $\hat{\beta}_2$ with $\underline\tau^\beta=1/4$ and $\overline\tau^\beta=3/4$, respectively.
The estimates of $\boldsymbol\beta_1^*$, $\boldsymbol\beta_2^*$ and $\tau_*^\beta$ 
are reported in Table \ref{tab4}.
It follows from Theorem \ref{th4} that
\begin{align*}
T(\hat\tau_n^\beta-\tau_*^\beta)=O_p(1).
\end{align*}
From Figure \ref{caseB-beta}, 
$T(\hat\tau_n^\beta-\tau_*^\beta)$ does not diverge 
when increasing from $n = 10^6$ to $n = 10^7$.
In this example, it appears that $T(\hat\tau_n^\beta-\tau_*^\beta)$ 
satisfies $O_p(1)$.

\begin{table}[h]
\caption{
Mean and standard deviation of the estimators under  
$\tau_*^\beta=0.5$, $\beta_1^*=0.25$, $\beta_2^*=-0.25$, $\gamma^*=1.2$ 
in Case B.
}
\begin{center}
\begin{tabular*}{.9\textwidth}{@{\extracolsep{\fill}}cccccccc}\hline
$n$ & $T$ & $h$ & $\hat\beta_1$ & $\hat\gamma_1$ & $\hat\beta_2$ & $\hat\gamma_2$ & 
$\hat\tau_n^\beta$ 
\rule[0mm]{0cm}{4mm}\\\hline 
$10^6$ & $3.73\times10^2$ & $3.73\times10^{-4}$ & 0.25964 & 1.24852 & $-0.26001$ & 1.24677 & 0.49884  \rule[0mm]{0cm}{4mm}\\ 
&&& (0.04142) & (0.16744) & (0.04451) & (0.18435) & (0.00191) \\
$10^7$ & $10^3$ & $10^{-4}$ & 0.25216 & 1.21212 & $-0.25489$ & 1.22162 & 0.50002  \rule[0mm]{0cm}{4mm}\\ 
&&& (0.02565) & (0.10281) & (0.02450) & (0.09966) & (0.00055) \\\hline
\end{tabular*}
\label{tab4}
\end{center}
\end{table}

\begin{figure}[h] 
\begin{center}
\includegraphics[width=5cm,pagebox=cropbox,clip]{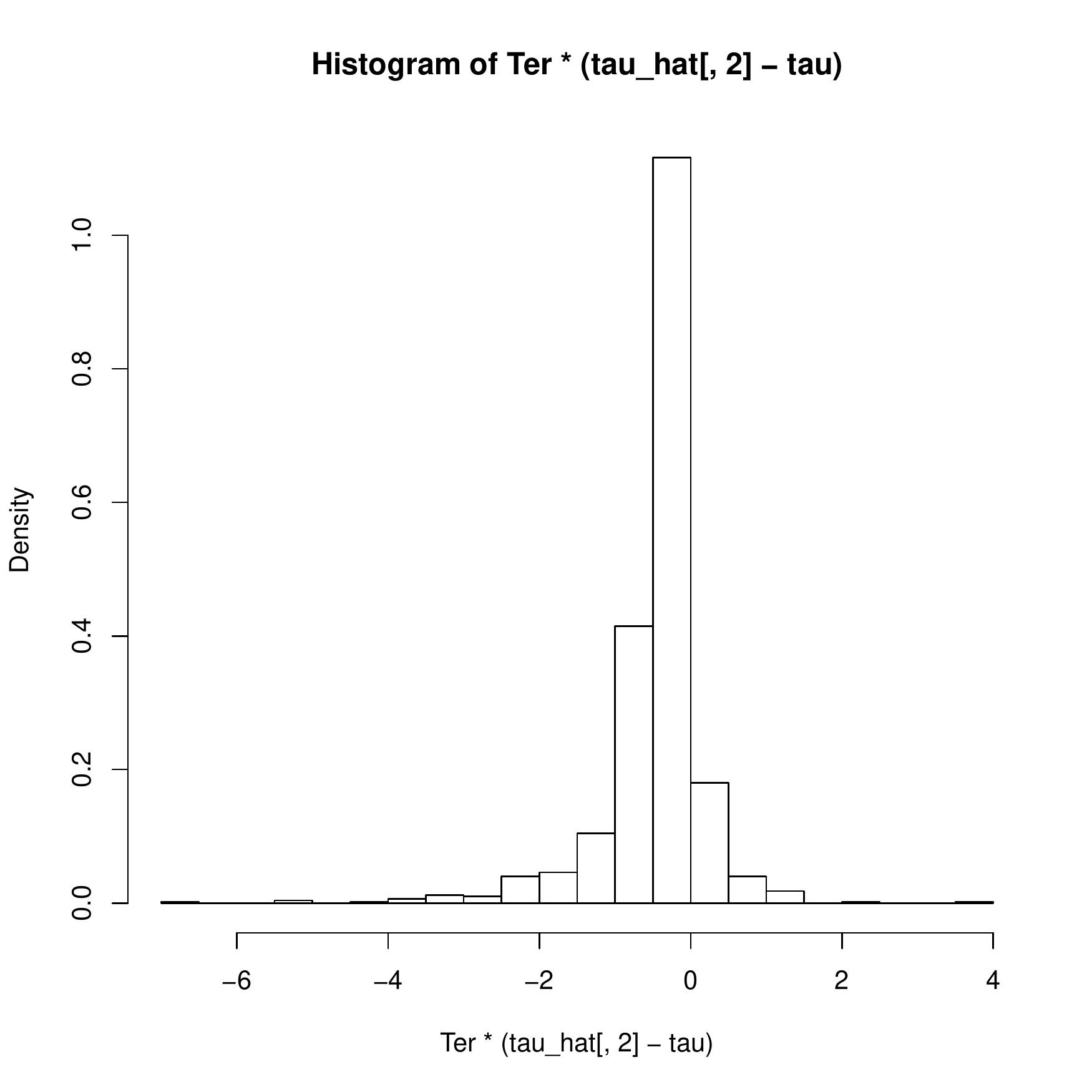}
\includegraphics[width=5cm,pagebox=cropbox,clip]{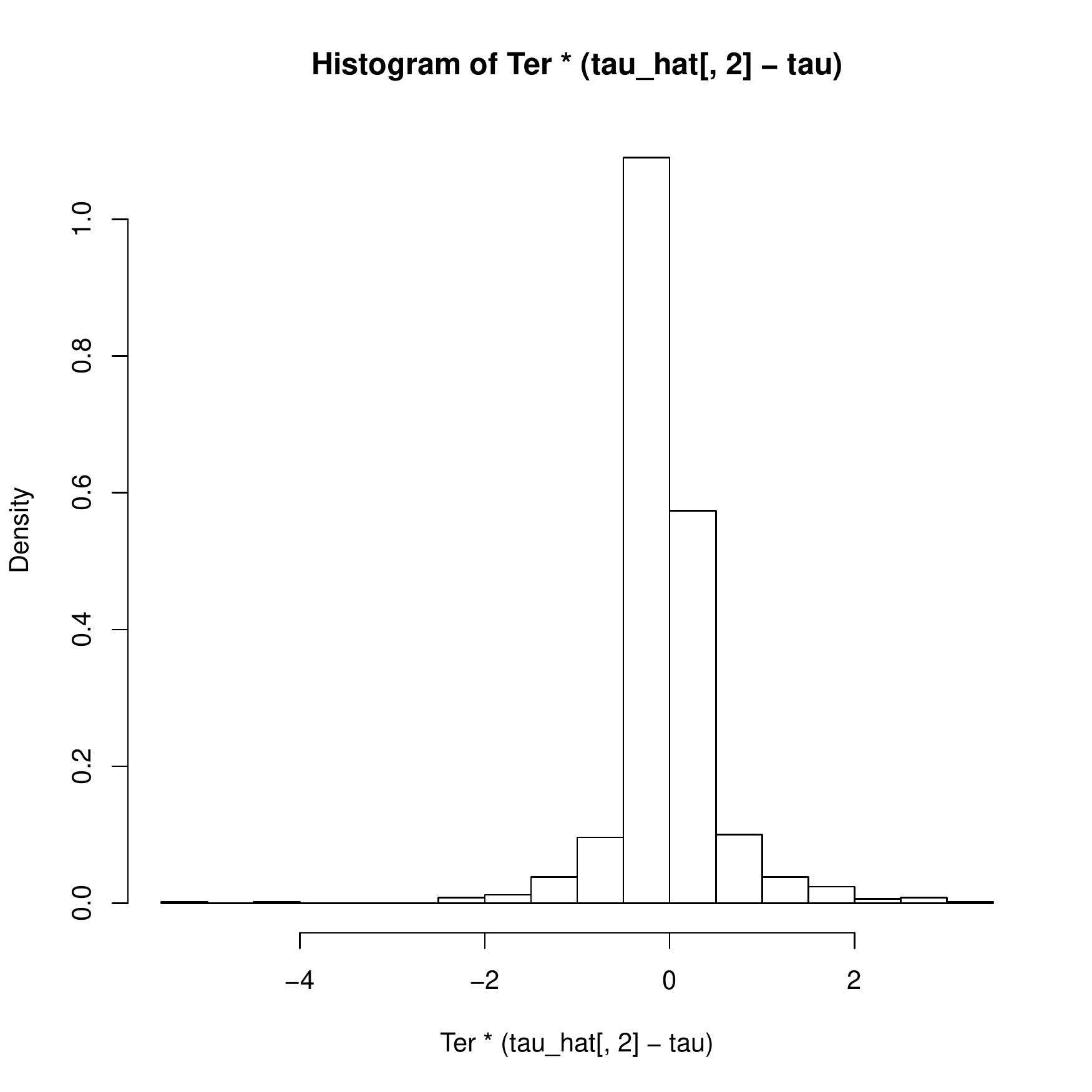}
\caption{The figure on the left is the histogram of 
$T(\hat\tau_n^\beta-\tau_*^\beta)$ with $n = 10^6$.
The figure on the right is the histogram of 
$T(\hat\tau_n^\beta-\tau_*^\beta)$ with $n = 10^7$.
}\label{caseB-beta}
\end{center}
\end{figure}


\section{Proofs}\label{sec7}
Let $\GG=\sigma\left[\{W_s\}_{s\le t_i^n}\right]$, and 
$C$, $C'>0$ denote universal constants.
If 
$f$ is a function on $\mathbb R^d\times \Theta$,
$f_{i-1}(\theta)$
denotes 
the value $f(\Xs,\theta)$.
If $\{u_n\}$ is a positive sequence, $R$ denotes a function 
on $\mathbb R^d\times\mathbb R_{+}\times\Theta$
for which there exists a constant $C>0$ such that 
\begin{align*}
\sup_{\theta\in\Theta}\| R(x,u_n,\theta)\|\le u_n C(1+\|x\|)^C.
\end{align*}
Let $R_{i-1}(u_n,\theta)=R(\Xs,u_n,\theta)$.
We set 
\begin{align*}
A\otimes x^{\otimes k}=\sum_{\ell_1,\ldots,\ell_k=1}^{d_1}
A^{\ell_1,\ldots,\ell_k}x^{\ell_1}\cdots x^{\ell_k},\quad
\text{ for }
A\in\underbrace{\mathbb R^{d_1}\otimes\cdots\otimes\mathbb R^{d_1}}_{k},\ x\in\mathbb R^{d_1}.
\end{align*} 
\begin{lem}\label{lem1}
Let $\Upsilon_n(\tau:\theta_1,\theta_2)$ be a contrast function,
and let $\hat\theta_1$, $\hat\theta_2$ be estimators of $\theta_1$, $\theta_2$,
respectively,
and let 
$\displaystyle
\hat\tau_n
=
\mathrm{argmin}_{\tau\in[0,1]}
\Upsilon _n(\tau:\hat\theta_1,\hat\theta_2)$
be the estimator of $\tau^*$, and let 
$\hat{\mathbb H}_n(v)=
\Upsilon_n(\tau^*+v/r_n:\hat\theta_1,\hat\theta_2)
-\Upsilon_n(\tau^*:\hat\theta_1,\hat\theta_2)$.
If there exist a positive sequence $\{r_n\}$ 
with $r_n\lto\infty$ 
and 
a random field $\mathbb H(v)$ 
that satisfy the following conditions, then 
$\displaystyle r_n(\hat\tau_n-\tau^*)
\dto
\mathrm{argmin}_{v\in\mathbb R}
\mathbb H(v)$.
\renewcommand{\labelenumi}{(\alph{enumi})}
\begin{enumerate}
\item 
$r_n(\hat\tau_n-\tau^*)=O_p(1)$,
\item 
For all $L>0$, 
\begin{align*}
\hat{\mathbb H}_n(v)\wto\mathbb H(v)\quad \text{in } \mathbb D[-L,L].
\end{align*}
\end{enumerate}
\end{lem}

\begin{proof}
Set 
$\displaystyle v^\dagger=
\mathrm{argmin}_{v\in\mathbb R}
\mathbb H(v)$.
For all $x\in\mathbb R$, 
\begin{align}
P(r_n(\hat\tau_n-\tau^*)\le x)
&\le
P\left(
r_n(\hat\tau_n-\tau^*)\le x,\  
r_n(\hat\tau_n-\tau^*)\in[-L,L],\  
\inf_{v\in[-L,x]}\hat{\mathbb H}_n(v)> \inf_{v\in[x,L]}\hat{\mathbb H}_n(v)
\right)\nonumber\\
&\quad+
P\Bigl(
r_n(\hat\tau_n-\tau^*)\notin[-L,L]
\Bigr)
+P\left(
\inf_{v\in[-L,x]}\hat{\mathbb H}_n(v)\le \inf_{v\in[x,L]}\hat{\mathbb H}_n(v)
\right).\label{eq-a1}
\end{align}
If 
$r_n(\hat\tau_n-\tau^*)\in[-L,L]$
and
$\inf_{v\in[-L,x]}\hat{\mathbb H}_n(v)> \inf_{v\in[x,L]}\hat{\mathbb H}_n(v)$,
then
\begin{align*}
\inf_{v\in[-L,x]}\Upsilon_n\left(\tau^*+\dfrac{v}{r_n}:\hat\theta_1,\hat\theta_2\right)
>\inf_{v\in[x,L]}\Upsilon_n\left(\tau^*+\dfrac{v}{r_n}:\hat\theta_1,\hat\theta_2\right),
\end{align*}
hence, from $\tau^*+\dfrac{x}{r_n}<\hat\tau_n\le \tau^*+\dfrac{L}{r_n}$,
$x< r_n(\hat\tau_n-\tau^*)\le L$ 
and
\begin{align*}
&P\left(
r_n(\hat\tau_n-\tau^*)\le x,\  
r_n(\hat\tau_n-\tau^*)\in[-L,L],\  
\inf_{v\in[-L,x]}\hat{\mathbb H}_n(v)> \inf_{v\in[x,L]}\hat{\mathbb H}_n(v)
\right)\\
&\le
P\Bigl(
r_n(\hat\tau_n-\tau^*)\le x,\ 
x< r_n(\hat\tau_n-\tau^*)\le L
\Bigr)\\
&=0.
\end{align*}
From (b), as $n\to\infty$, \eqref{eq-a1} is evaluated as follows.
\begin{align}
\varlimsup_{n\to\infty}P\Bigl(r_n(\hat\tau_n-\tau^*)\le x\Bigr)
&\le
\varlimsup_{n\to\infty}P\Bigl(
r_n(\hat\tau_n-\tau^*)\notin[-L,L]
\Bigr)
+\varlimsup_{n\to\infty}P\left(
\inf_{v\in[-L,x]}\hat{\mathbb H}_n(v)\le \inf_{v\in[x,L]}\hat{\mathbb H}_n(v)
\right)\nonumber\\
&\le
\sup_{n\in\mathbb N}P\Bigl(
r_n(\hat\tau_n-\tau^*)\notin[-L,L]
\Bigr)
+P\left(
\inf_{v\in[-L,x]}\mathbb H(v)\le \inf_{v\in[x,L]}\mathbb H(v)
\right).\label{eq-a2}
\end{align}
Now we have
\begin{align*}
P\left(
\inf_{v\in[-L,x]}\mathbb H(v)\le \inf_{v\in[x,L]}\mathbb H(v)
\right)
&\le 
P\left(
\inf_{v\in[-L,x]}\mathbb H(v)\le \inf_{v\in[x,L]}\mathbb H(v),\ 
v^\dagger\in[-L,L],\ 
v^\dagger >x
\right)\\
&\quad+
P(v^\dagger\notin[-L,L])
+P(v^\dagger \le x).
\end{align*}
Because if 
$\inf_{v\in[-L,x]}\mathbb H(v)\le \inf_{v\in[x,L]}\mathbb H(v)$ 
and $v^\dagger\in[-L,L]$,
then $-L\le v^\dagger \le x$,
we have
\begin{align*}
P\left(
\inf_{v\in[-L,x]}\mathbb H(v)\le \inf_{v\in[x,L]}\mathbb H(v),\ 
v^\dagger\in[-L,L],\ 
v^\dagger >x
\right)
\le 
P(-L\le v^\dagger \le x,\ v^\dagger >x)
=0.
\end{align*}
Therefore from
\begin{align*}
P\left(
\inf_{v\in[-L,x]}\mathbb H(v)\le \inf_{v\in[x,L]}\mathbb H(v)
\right)
\le 
P(v^\dagger\notin[-L,L])
+P(v^\dagger \le x)
\end{align*}
and (a), as $L\to\infty$, \eqref{eq-a2} is evaluated as follows.
\begin{align}
\varlimsup_{n\to\infty}P\Bigl(r_n(\hat\tau_n-\tau^*)\le x\Bigr)
&\le
\varlimsup_{L\to\infty}\sup_{n\in\mathbb N}
P\Bigl(
r_n(\hat\tau_n-\tau^*)\notin[-L,L]
\Bigr)
+\varlimsup_{L\to\infty}P(v^\dagger\notin[-L,L])
+P(v^\dagger \le x)
\nonumber\\
&=P(v^\dagger \le x).
\label{eq-a3}
\end{align}
In the same way, we have
\begin{align*}
\varlimsup_{n\to\infty}P\Bigl(r_n(\hat\tau_n-\tau^*)> x\Bigr)
&\le
\varlimsup_{n\to\infty}
P\Bigl(
r_n(\hat\tau_n-\tau^*)\notin[-L,L]
\Bigr)
+\varlimsup_{n\to\infty}P\left(
\inf_{v\in[-L,x]}\hat{\mathbb H}_n(v)\ge \inf_{v\in[x,L]}\hat{\mathbb H}_n(v)
\right)\\
&\le
\sup_{n\in\mathbb N}
P\Bigl(
r_n(\hat\tau_n-\tau^*)\notin[-L,L]
\Bigr)
+P\left(
\inf_{v\in[-L,x]}\mathbb H(v)\ge \inf_{v\in[x,L]}\mathbb H(v)
\right)\\
&\le
\sup_{n\in\mathbb N}
P\Bigl(
r_n(\hat\tau_n-\tau^*)\notin[-L,L]
\Bigr)
+P(v^\dagger>x)
+P(v^\dagger\notin[-L,L]).
\end{align*}
As $L\to\infty$, 
\begin{align*}
\varlimsup_{n\to\infty}P\Bigl(r_n(\hat\tau_n-\tau^*)> x\Bigr)
\le
P(v^\dagger>x),
\end{align*}
i.e.,
\begin{align}
\varliminf_{n\to\infty}P\Bigl(r_n(\hat\tau_n-\tau^*)\le x\Bigr)
&=\varliminf_{n\to\infty}\Bigl[1-P\Bigl(r_n(\hat\tau_n-\tau^*)> x\Bigr)\Bigr]
\nonumber\\
&=1-\varlimsup_{n\to\infty}P\Bigl(r_n(\hat\tau_n-\tau^*)> x\Bigr)
\nonumber\\
&\ge1-P(v^\dagger>x)
\nonumber\\
&=P(v^\dagger\le x).
\label{eq-a4}
\end{align}
From \eqref{eq-a3} and \eqref{eq-a4}, we obtain 
$\displaystyle\lim_{n\to\infty}P\Bigl(r_n(\hat\tau_n-\tau^*)\le x\Bigr)
=P(v^\dagger\le x)$ 
and $r_n(\hat\tau_n-\tau^*)\dto v^\dagger$.
\end{proof}

In Case A of Situation I, we set
\begin{align*}
\mathbb F_n(v)
&=\Phi_n\left(\tau_*^\alpha+\frac{v}{n\Dea^2}:\alpha_1^*,\alpha_2^*\right)
-\Phi_n\left(\tau_*^\alpha:\alpha_1^*,\alpha_2^*\right),\\
\hat{\mathbb F}_n(v)
&=\Phi_n\left(\tau_*^\alpha+\frac{v}{n\Dea^2}:\hat\alpha_1,\hat\alpha_2\right)
-\Phi_n\left(\tau_*^\alpha:\hat\alpha_1,\hat\alpha_2\right),\\
\mathcal D_n^\alpha(v)
&=\hat{\mathbb F}_n(v)-\mathbb F_n(v).
\end{align*}
\begin{lem}\label{lem2}
Suppose that \textbf{[C1]}-\textbf{[C5]}, \textbf{[C6-I]},
\textbf{[A1-I]} and \textbf{[A2-I]} hold.
Then, for all $L>0$, 
\begin{align*}
\sup_{v\in[-L,L]}|\mathcal D_n^\alpha(v)|\pto0
\end{align*}
as $n\to\infty$.
\end{lem}
\begin{proof}
We assume that $v>0$. Then, we can express 
\begin{align}
\mathcal D_n^\alpha(v)
&=\sum_{i=[n\tau_*^\alpha]+1}^{[n\tau_*^\alpha+v/\Dea^2]}
\Bigl(F_i(\hat\alpha_1)-F_i(\hat\alpha_2)\Bigr)
-\sum_{i=[n\tau_*^\alpha]+1}^{[n\tau_*^\alpha+v/\Dea^2]}
\Bigl(F_i(\alpha_1^*)-F_i(\alpha_2^*)\Bigr)
\nonumber\\
&=\sum_{i=[n\tau_*^\alpha]+1}^{[n\tau_*^\alpha+v/\Dea^2]}
\left[
\left(F_i(\alpha_1^*)+\partial_\alpha F_i(\alpha_1^*)(\hat\alpha_1-\alpha_1^*)
+\int_0^1\partial_\alpha^2 F_i(\alpha_1^*+u(\hat\alpha_1-\alpha_1^*))\dd u
\otimes(\hat\alpha_1-\alpha_1^*)^{\otimes2}
\right)\right.
\nonumber\\
&\left.\qquad\qquad
-\left(F_i(\alpha_2^*)
+\partial_\alpha F_i(\alpha_2^*)(\hat\alpha_2-\alpha_2^*)
+\int_0^1\partial_\alpha^2 F_i(\alpha_2^*+u(\hat\alpha_2-\alpha_2^*))\dd u
\otimes(\hat\alpha_2-\alpha_2^*)^{\otimes2}
\right)\right]
\nonumber\\
&\qquad
-\sum_{i=[n\tau_*^\alpha]+1}^{[n\tau_*^\alpha+v/\Dea^2]}
\Bigl(F_i(\alpha_1^*)-F_i(\alpha_2^*)\Bigr)
\nonumber\\
&=\sum_{i=[n\tau_*^\alpha]+1}^{[n\tau_*^\alpha+v/\Dea^2]}
\Bigl(
\partial_\alpha F_i(\alpha_1^*)(\hat\alpha_1-\alpha_1^*)
+\partial_\alpha F_i(\alpha_2^*)(\hat\alpha_2-\alpha_2^*)
\Bigr)
\nonumber\\
&\qquad+
\sum_{i=[n\tau_*^\alpha]+1}^{[n\tau_*^\alpha+v/\Dea^2]}
\left(
\int_0^1\partial_\alpha^2 F_i(\alpha_1^*+u(\hat\alpha_1-\alpha_1^*))\dd u
\otimes(\hat\alpha_1-\alpha_1^*)^{\otimes2}\right.
\nonumber\\
&\left.\qquad\qquad
-\int_0^1\partial_\alpha^2 F_i(\alpha_2^*+u(\hat\alpha_2-\alpha_2^*))\dd u
\otimes(\hat\alpha_2-\alpha_2^*)^{\otimes2}
\right).
\label{eq7.969}
\end{align}
Now we see
\begin{align}
&\sup_{v\in[0,L]}
\left|
\sum_{i=[n\tau_*^\alpha]+1}^{[n\tau_*^\alpha+v/\Dea^2]}
\int_0^1\partial_\alpha^2 F_i(\alpha_k^*+u(\hat\alpha_k-\alpha_k^*))\dd u
\otimes(\hat\alpha_k-\alpha_k^*)^{\otimes2}
\right|
\nonumber\\
&\le
\sum_{i=[n\tau_*^\alpha]+1}^{[n\tau_*^\alpha+L/\Dea^2]}
\left|
\int_0^1\partial_\alpha^2 F_i(\alpha_k^*+u(\hat\alpha_k-\alpha_k^*))\dd u
\right|
|\hat\alpha_k-\alpha_k^*|^2
\nonumber\\
&\le
\frac{1}{n}
\sum_{i=[n\tau_*^\alpha]+1}^{[n\tau_*^\alpha+L/\Dea^2]}
\sup_{\alpha\in\Theta_A}\left|
\partial_\alpha^2 F_i(\alpha)
\right|
\left(\sqrt{n}|\hat\alpha_k-\alpha_k^*|\right)^2
\nonumber\\
&=O_p\left(\frac{1}{n\Dea^2}\right)
=o_p(1)
\label{eq7.968}
\end{align}
and
\begin{align}
&\sum_{i=[n\tau_*^\alpha]+1}^{[n\tau_*^\alpha+v/\Dea^2]}
\partial_\alpha F_i(\alpha_k^*)(\hat\alpha_k-\alpha_k^*)
\nonumber\\
&=
\sum_{i=[n\tau_*^\alpha]+1}^{[n\tau_*^\alpha+v/\Dea^2]}
\Bigl(
\partial_\alpha F_i(\alpha_k^*)
-\EE_{\alpha_2^*}[\partial_\alpha F_i(\alpha_k^*)|\GG]
\Bigr)
(\hat\alpha_k-\alpha_k^*)
+\sum_{i=[n\tau_*^\alpha]+1}^{[n\tau_*^\alpha+v/\Dea^2]}
\EE_{\alpha_2^*}[\partial_\alpha F_i(\alpha_k^*)|\GG]
(\hat\alpha_k-\alpha_k^*).
\label{eq7.967}
\end{align}
By Theorem 2.11 of Hall and Heyde (1980), we have
\begin{align*}
&\EE_{\alpha_2^*}\left[
\frac1n
\sup_{v\in[0,L]}
\left|
\sum_{i=[n\tau_*^\alpha]+1}^{[n\tau_*^\alpha+v/\Dea^2]}
\left(\partial_\alpha F_i(\alpha_k^*)
-\EE_{\alpha_2^*}[\partial_\alpha F_i(\alpha_k^*)|\GG]\right)
\right|^2
\right]\\
&\le
\frac{C}{n}
\sum_{i=[n\tau_*^\alpha]+1}^{[n\tau_*^\alpha+L/\Dea^2]}
\EE_{\alpha_2^*}\left[
\EE_{\alpha_2^*}\left[
\left|\partial_\alpha F_i(\alpha_k^*)
-\EE_{\alpha_2^*}[\partial_\alpha F_i(\alpha_k^*)|\GG]\right|^2
\middle|\GG\right]\right]\\
&\le\frac{C'}{n\Dea^2}
\lto0
\end{align*}
and
\begin{align}
\frac1{\sqrt n}
\sup_{v\in[0,L]}
\left|
\sum_{i=[n\tau_*^\alpha]+1}^{[n\tau_*^\alpha+v/\Dea^2]}
\left(\partial_\alpha F_i(\alpha_k^*)
-\EE_{\alpha_2^*}[\partial_\alpha F_i(\alpha_k^*)|\GG]\right)
\right|=o_p(1).
\label{eq7.966}
\end{align}
Moreover, we see
\begin{align}
&\sup_{v\in[0,L]}
\left|
\sum_{i=[n\tau_*^\alpha]+1}^{[n\tau_*^\alpha+v/\Dea^2]}
\EE_{\alpha_2^*}[\partial_\alpha F_i(\alpha_k^*)|\GG]
\right|
|\hat\alpha_k-\alpha_k^*|
\nonumber
\\
&=\sup_{v\in[0,L]}
\left|
\sum_{i=[n\tau_*^\alpha]+1}^{[n\tau_*^\alpha+v/\Dea^2]}
\biggl(\left[\tr\Bigl(
A_{i-1}^{-1}\partial_{\alpha^\ell}A_{i-1}A_{i-1}^{-1}(\alpha_k^*)
(A_{i-1}(\alpha_2^*)-A_{i-1}(\alpha_k^*))
\Bigr)
\right]_{\ell}
+\Ri
\biggr)
\right|
|\hat\alpha_k-\alpha_k^*|
\nonumber
\\
&=
\sup_{v\in[0,L]}
\left|
\sum_{i=[n\tau_*^\alpha]+1}^{[n\tau_*^\alpha+v/\Dea^2]}
\Bigl(
\Xi_{i-1}^\alpha(\alpha_k^*)(\alpha_2^*-\alpha_k^*)
\right.
\nonumber
\\
&
\qquad+
\int_0^1
(1-u)
\left(
\tr
\left[
A_{i-1}^{-1}\partial_{\alpha^{\ell_1}}A_{i-1}A_{i-1}^{-1}(\alpha_k^*)
\partial_{\alpha^{\ell_3}}\partial_{\alpha^{\ell_2}}A_{i-1}
(\alpha_k^*+u(\alpha_2^*-\alpha_k^*))
\right]
\right)_{\ell_1,\ell_2,\ell_3}
\dd u
\otimes
(\alpha_2^*-\alpha_k^*)^{\otimes2}
\nonumber
\\
&\qquad
+\Ri
\Bigr)
\Biggr|
|\hat\alpha_k-\alpha_k^*|
\nonumber
\\
&=
\sup_{v\in[0,L]}
\left|
\sum_{i=[n\tau_*^\alpha]+1}^{[n\tau_*^\alpha+v/\Dea^2]}
\biggl(
\Xi_{i-1}^\alpha(\alpha_0)(\alpha_2^*-\alpha_k^*)
+\int_0^1
\partial_\alpha\Xi_{i-1}^\alpha(\alpha_0+u(\alpha_k^*-\alpha_0))
\dd u
\otimes(\alpha_2^*-\alpha_k^*)
\otimes(\alpha_k^*-\alpha_0)
\right.
\nonumber
\\
&
\qquad+
\int_0^1
(1-u)
\left(
\tr
\left[
A_{i-1}^{-1}\partial_{\alpha^{\ell_1}}A_{i-1}A_{i-1}^{-1}(\alpha_k^*)
\partial_{\alpha^{\ell_3}}\partial_{\alpha^{\ell_2}}A_{i-1}
(\alpha_k^*+u(\alpha_2^*-\alpha_k^*))
\right]
\right)_{\ell_1,\ell_2,\ell_3}
\dd u
\otimes
(\alpha_2^*-\alpha_k^*)^{\otimes2}
\nonumber
\\
&\qquad
+\Ri
\biggr)
\Biggr|
|\hat\alpha_k-\alpha_k^*|
\nonumber
\\
&\le
\frac{\Dea}{\sqrt{n}}
\sup_{v\in[0,L]}
\left|
\sum_{i=[n\tau_*^\alpha]+1}^{[n\tau_*^\alpha+v/\Dea^2]}
\Xi_{i-1}^\alpha(\alpha_0)
\right|
\Dea^{-1}|\alpha_2^*-\alpha_k^*|
\sqrt{n}|\hat\alpha_k-\alpha_k^*|
\nonumber
\\
&\qquad+
\frac{\Dea^2}{\sqrt n}
\sum_{i=[n\tau_*^\alpha]+1}^{[n\tau_*^\alpha+L/\Dea^2]}
\Ro
\Dea^{-1}|\alpha_2^*-\alpha_k^*|
\Dea^{-1}|\alpha_2^*-\alpha_0|
\sqrt{n}|\hat\alpha_k-\alpha_k^*|
\nonumber
\\
&\qquad+
\frac{\Dea^2}{\sqrt n}
\sum_{i=[n\tau_*^\alpha]+1}^{[n\tau_*^\alpha+L/\Dea^2]}
\Ro
(\Dea^{-1}|\alpha_2^*-\alpha_k^*|)^2
\sqrt{n}|\hat\alpha_k-\alpha_k^*|
+
\frac{1}{\sqrt n}
\sum_{i=[n\tau_*^\alpha]+1}^{[n\tau_*^\alpha+L/\Dea^2]}
\Ri
\sqrt{n}|\hat\alpha_k-\alpha_k^*|
\nonumber
\\
&=O_p\left(\frac{1}{\sqrt{n}\Dea}\right)
+O_p\left(\frac{1}{\sqrt{n}}\right)
+O_p\left(\frac{h}{\sqrt{n}\Dea^2}\right)
\nonumber
\\
&=o_p(1).
\label{eq7.965}
\end{align}
Therefore, from \eqref{eq7.969}-\eqref{eq7.965}, we have
\begin{align*}
\sup_{v\in[0,L]}
|{\mathcal D}_n^\alpha(v)|\pto0.
\end{align*}
By the similar proof, we see
$\sup_{v\in[-L,0]}
|{\mathcal D}_n^\alpha(v)|\pto0$ 
and this proof is complete.
\end{proof}

\begin{lem}\label{lem3}
Suppose that \textbf{[C1]}-\textbf{[C5]}, \textbf{[C6-I]},
\textbf{[A1-I]} and \textbf{[A2-I]} hold.
Then, for all $L>0$,
\begin{align*}
\mathbb F_n(v)\wto\mathbb F(v) \text{ in }\mathbb D[-L,L]
\end{align*}
as $n\to\infty$. 
\end{lem}
\begin{proof}
We consider $v>0$. We have
\begin{align*}
\mathbb F_n(v)
&=\sum_{i=[n\tau_*^\alpha]+1}^{[n\tau_*^\alpha+v/\Dea^2]}
\Bigl(F_i(\alpha_1^*)-F_i(\alpha_2^*)\Bigr)\\
&=\sum_{i=[n\tau_*^\alpha]+1}^{[n\tau_*^\alpha+v/\Dea^2]}
\biggl(
\partial_\alpha F_i(\alpha_2^*)(\alpha_1^*-\alpha_2^*)
+\frac12\partial_\alpha^2 F_i(\alpha_2^*)\otimes(\alpha_1^*-\alpha_2^*)^{\otimes2}
\\
&\qquad
+\int_0^1\frac{(1-u)^2}2\partial_\alpha^3 
F_i(\alpha_2^*+u(\alpha_1^*-\alpha_2^*))\dd u
\otimes(\alpha_1^*-\alpha_2^*)^{\otimes3}
\biggr)\\
&=\sum_{i=[n\tau_*^\alpha]+1}^{[n\tau_*^\alpha+v/\Dea^2]}
\left(
\partial_\alpha F_i(\alpha_2^*)(\alpha_1^*-\alpha_2^*)
+\frac12\partial_\alpha^2 F_i(\alpha_2^*)\otimes(\alpha_1^*-\alpha_2^*)^{\otimes2}
\right)+\bar o_p(1)\\
&=:\mathbb F_{1,n}(v)+\mathbb F_{2,n}(v)+\bar o_p(1),
\end{align*}
where $Y_n(v)=\bar o_p(1)$ denotes $\sup_{v\in[0,L]}|Y_n(v)|=o_p(1)$.
Now, we see
\begin{align}
\mathbb F_{1,n}(v)
=\sum_{i=[n\tau_*^\alpha]+1}^{[n\tau_*^\alpha+v/\Dea^2]}
\Bigl(
\partial_\alpha F_i(\alpha_2^*)
-\EE_{\alpha_2^*}[\partial_\alpha F_i(\alpha_2^*)|\GG]
\Bigr)
(\alpha_1^*-\alpha_2^*)
+\sum_{i=[n\tau_*^\alpha]+1}^{[n\tau_*^\alpha+v/\Dea^2]}
\EE_{\alpha_2^*}[\partial_\alpha F_i(\alpha_2^*)|\GG](\alpha_1^*-\alpha_2^*),
\label{eq7.959}
\end{align}
\begin{align}
\sup_{v\in[0,L]}
\left|
\sum_{i=[n\tau_*^\alpha]+1}^{[n\tau_*^\alpha+v/\Dea^2]}
\EE_{\alpha_2^*}[\partial_\alpha F_i(\alpha_2^*)|\GG](\alpha_1^*-\alpha_2^*)
\right|
\le
\Dea\sum_{i=[n\tau_*^\alpha]+1}^{[n\tau_*^\alpha+L/\Dea^2]}
\Ri
=O_p\left(\frac{h}{\Dea}\right)
=o_p(1),
\label{eq7.958}
\end{align}
\begin{align}
&\sum_{i=[n\tau_*^\alpha]+1}^{[n\tau_*^\alpha+v/\Dea^2]}
\EE_{\alpha_2^*}\left[
\biggl(
\Bigl(
\partial_\alpha F_i(\alpha_2^*)
-\EE_{\alpha_2^*}[\partial_\alpha F_i(\alpha_2^*)|\GG]
\Bigr)
(\alpha_1^*-\alpha_2^*)
\biggr)^2
\middle|\GG\right]
\nonumber
\\
&=(\alpha_1^*-\alpha_2^*)^\TT
\sum_{i=[n\tau_*^\alpha]+1}^{[n\tau_*^\alpha+v/\Dea^2]}
\EE_{\alpha_2^*}\biggl[
\Bigl(
\partial_\alpha F_i(\alpha_2^*)
-\EE_{\alpha_2^*}[\partial_\alpha F_i(\alpha_2^*)|\GG]
\Bigr)^\TT
\nonumber
\\
&\qquad\qquad
\Bigl(
\partial_\alpha F_i(\alpha_2^*)
-\EE_{\alpha_2^*}[\partial_\alpha F_i(\alpha_2^*)|\GG]
\Bigr)
\biggr|\GG\biggr]
(\alpha_1^*-\alpha_2^*)
\nonumber
\\
&=(\alpha_1^*-\alpha_2^*)^\TT
\sum_{i=[n\tau_*^\alpha]+1}^{[n\tau_*^\alpha+v/\Dea^2]}
\left(
2\Xi^\alpha_{i-1}(\alpha_2^*)
+\Ri
\right)
(\alpha_1^*-\alpha_2^*)
\nonumber
\\
&\pto
2e_\alpha^\TT 
\int_{\mathbb R^d}
\Xi^\alpha(x,\alpha_0)\dd\mu_{\alpha_0}(x)
e_\alpha v
=4\mathcal J_\alpha v
\label{eq7.957}
\end{align}
and
\begin{align}
&\sum_{i=[n\tau_*^\alpha]+1}^{[n\tau_*^\alpha+v/\Dea^2]}
\EE_{\alpha_2^*}\left[
\biggl(
\Bigl(
\partial_\alpha F_i(\alpha_2^*)
-\EE_{\alpha_2^*}[\partial_\alpha F_i(\alpha_2^*)|\GG]
\Bigr)
(\alpha_1^*-\alpha_2^*)
\biggr)^4
\middle|\GG\right]
=
\sum_{i=[n\tau_*^\alpha]+1}^{[n\tau_*^\alpha+v/\Dea^2]}
\Dea^4\Ro
\pto0.
\label{eq7.956}
\end{align}
According to Corollary 3.8 of McLeish (1974), we obtain, 
from \eqref{eq7.957} and \eqref{eq7.956},
\begin{align}
&\sum_{i=[n\tau_*^\alpha]+1}^{[n\tau_*^\alpha+v/\Dea^2]}
\Bigl(
\partial_\alpha F_i(\alpha_2^*)
-\EE_{\alpha_2^*}[\partial_\alpha F_i(\alpha_2^*)|\GG]
\Bigr)
(\alpha_1^*-\alpha_2^*)
\wto
-2\mathcal J_\alpha^{1/2}\mathcal W(v)
\text{ in }\mathbb D[0,L].
\label{eq7.955}
\end{align}
Further, from \eqref{eq7.959}, \eqref{eq7.958} and \eqref{eq7.955}, we have
$\mathbb F_{1,n}(v)\wto
-2\mathcal J_\alpha^{1/2}\mathcal W(v)$ in $\mathbb D[0,L]$.

Besides, we see 
\begin{align}
&
\sup_{v\in[0,L]}\left|\sum_{i=[n\tau_*^\alpha]+1}^{[n\tau_*^\alpha+v/\Dea^2]}
\partial_\alpha^2 F_i(\alpha_2^*)\otimes(\alpha_1^*-\alpha_2^*)^{\otimes2}
-2\mathcal J_\alpha v\right|
\nonumber
\\
&\le
\sup_{v\in[0,L]}\left|\sum_{i=[n\tau_*^\alpha]+1}^{[n\tau_*^\alpha+v/\Dea^2]}
\Bigl(
\partial_\alpha^2 F_i(\alpha_2^*)
-\EE_{\alpha_2^*}\left[\partial_\alpha^2 F_i(\alpha_2^*)\middle|\GG\right]
\Bigr)
\otimes(\alpha_1^*-\alpha_2^*)^{\otimes2}
\right|
\nonumber\\
&\qquad
+\sup_{v\in[0,L]}\left|\sum_{i=[n\tau_*^\alpha]+1}^{[n\tau_*^\alpha+v/\Dea^2]}
\EE_{\alpha_2^*}\left[\partial_\alpha^2 F_i(\alpha_2^*)\middle|\GG\right]
\otimes(\alpha_1^*-\alpha_2^*)^{\otimes2}
-2\mathcal J_\alpha v
\right|,
\label{eq7.949}
\end{align}
\begin{align}
\sup_{v\in[0,L]}
\left|
\sum_{i=[n\tau_*^\alpha]+1}^{[n\tau_*^\alpha+v/\Dea^2]}
\Bigl(
\partial_\alpha^2 F_i(\alpha_2^*)
-\EE_{\alpha_2^*}\left[\partial_\alpha^2 F_i(\alpha_2^*)\middle|\GG\right]
\Bigr)
\otimes(\alpha_1^*-\alpha_2^*)^{\otimes2}
\right|
\pto0
\label{eq7.948}
\end{align}
and
\begin{align}
&\sup_{v\in[0,L]}\left|\sum_{i=[n\tau_*^\alpha]+1}^{[n\tau_*^\alpha+v/\Dea^2]}
\EE_{\alpha_2^*}\left[\partial_\alpha^2 F_i(\alpha_2^*)\middle|\GG\right]
\otimes(\alpha_1^*-\alpha_2^*)^{\otimes2}
-2\mathcal J_\alpha v
\right|
\nonumber
\\
&=
\sup_{v\in[0,L]}\left|\sum_{i=[n\tau_*^\alpha]+1}^{[n\tau_*^\alpha+v/\Dea^2]}
\left(
\Xi^\alpha_{i-1}(\alpha_2^*)+\Ri
\right)\otimes(\alpha_1^*-\alpha_2^*)^{\otimes2}
-2\mathcal J_\alpha v
\right|
\nonumber
\\
&\le
\sup_{v\in[0,L]}\left|\sum_{i=[n\tau_*^\alpha]+1}^{[n\tau_*^\alpha+v/\Dea^2]}
\Xi^\alpha_{i-1}(\alpha_0)\otimes(\alpha_1^*-\alpha_2^*)^{\otimes2}
-2\mathcal J_\alpha v
\right|+o_p(1)
\nonumber
\\
&\pto0,
\label{eq7.947}
\end{align}
where \eqref{eq7.948} is obtained by
\begin{align*}
&\EE_{\alpha_2^*}\left[
\sup_{v\in[0,L]}
\left|
\sum_{i=[n\tau_*^\alpha]+1}^{[n\tau_*^\alpha+v/\Dea^2]}
\Bigl(
\partial_\alpha^2 F_i(\alpha_2^*)
-\EE_{\alpha_2^*}\left[\partial_\alpha^2 F_i(\alpha_2^*)\middle|\GG\right]
\Bigr)
\otimes(\alpha_1^*-\alpha_2^*)^{\otimes2}
\right|^2
\right]\\
&\le
C\Dea^4
\sum_{i=[n\tau_*^\alpha]+1}^{[n\tau_*^\alpha+L/\Dea^2]}
\EE_{\alpha_2^*}\left[
\left|
\partial_\alpha^2 F_i(\alpha_2^*)
-\EE_{\alpha_2^*}\left[\partial_\alpha^2 F_i(\alpha_2^*)\middle|\GG\right]
\right|^2
\right]\\
&\le
C'\Dea^2\lto0
\end{align*}
from Theorem 2.11 of Hall and Heyde (1980),
and \eqref{eq7.947} is obtained by 
\begin{align*}
&\sup_{v\in[0,L]}
\left|
\sum_{i=[n\tau_*^\alpha]+1}^{[n\tau_*^\alpha+v/\Dea^2]}
\Xi^\alpha_{i-1}(\alpha_0)\otimes(\alpha_1^*-\alpha_2^*)^{\otimes2}
-2\mathcal J_\alpha v
\right|\\
&\le
\sup_{v\in[0,\epsilon_n]}
\left|
\sum_{i=[n\tau_*^\alpha]+1}^{[n\tau_*^\alpha+v/\Dea^2]}
\Xi^\alpha_{i-1}(\alpha_0)\otimes(\alpha_1^*-\alpha_2^*)^{\otimes2}
-2\mathcal J_\alpha v
\right|\\
&\qquad
+\sup_{v\in[\epsilon_n,L]}
\left|
\sum_{i=[n\tau_*^\alpha]+1}^{[n\tau_*^\alpha+v/\Dea^2]}
\Xi^\alpha_{i-1}(\alpha_0)\otimes(\alpha_1^*-\alpha_2^*)^{\otimes2}
-2\mathcal J_\alpha v
\right|
\\
&\le
\epsilon_n\sup_{v\in[0,\epsilon_n]}
\left|
\frac1v\sum_{i=[n\tau_*^\alpha]+1}^{[n\tau_*^\alpha+v/\Dea^2]}
\Xi^\alpha_{i-1}(\alpha_0)\otimes(\alpha_1^*-\alpha_2^*)^{\otimes2}
-2\mathcal J_\alpha
\right|\\
&\qquad
+L\sup_{v\in[\epsilon_n,L]}
\left|
\frac1v\sum_{i=[n\tau_*^\alpha]+1}^{[n\tau_*^\alpha+v/\Dea^2]}
\Xi^\alpha_{i-1}(\alpha_0)\otimes(\alpha_1^*-\alpha_2^*)^{\otimes2}
-2\mathcal J_\alpha
\right|
\\
&\le
O_p(\epsilon_n)
+L\max_{[n^{1/r}]\le k\le n}
\left|
\frac1k\sum_{i=[n\tau_*^\alpha]+1}^{[n\tau_*^\alpha]+k}
\Xi^\alpha_{i-1}(\alpha_0)\otimes(\Dea^{-1}(\alpha_1^*-\alpha_2^*))^{\otimes2}
-2\mathcal J_\alpha
\right|
\\
&=o_p(1).
\end{align*}
Here $\{\epsilon_n\}_{n=1}^\infty$ is a positive sequence such that 
$\epsilon_n\lto0$ and $\epsilon_nh/\Dea^2\lto\infty$,
and 
$r$ is a constant with
$r\in(1,2)$ and $nh^r\lto\infty$.
From \eqref{eq7.949}-\eqref{eq7.947}, we have 
$\sup_{v\in[0,L]}|\mathbb F_{2,n}(v)-\mathcal J_\alpha v|\pto 0$.
Therefore we obtain 
\begin{align*}
\mathbb F_n(v)\wto-2\mathcal J_\alpha^{1/2}\mathcal W(v)+\mathcal J_\alpha v
\ \text{ in }\ \mathbb D[0,L].
\end{align*}
The argument for $v<0$ is proved as well.
\end{proof}

\begin{proof}[\bf{Proof of Theorem \ref{th1}}]
Let $M\ge1$. We have
\begin{align}
P\left(n\Dea^2|\hat\tau_n^\alpha-\tau_*^\alpha|>M\right)
&=P\left(n\Dea^2(\hat\tau_n^\alpha-\tau_*^\alpha)>M\right)
+P\left(n\Dea^2(\tau_*^\alpha-\hat\tau_n^\alpha)>M\right).
\label{eq7.979}
\end{align}
For $\tau>\tau_*^\alpha$, we have
\begin{align}
\Phi_n(\tau:\alpha_1,\alpha_2)-\Phi_n(\tau_*^\alpha:\alpha_1,\alpha_2)
&=\sum_{i=1}^{[n\tau]}F_i(\alpha_1)+\sum_{i=[n\tau]+1}^n F_i(\alpha_2)
-\sum_{i=1}^{[n\tau_*^\alpha]}F_i(\alpha_1)-\sum_{i=[n\tau_*^\alpha]+1}^n F_i(\alpha_2)
\nonumber
\\
&=\sum_{i=[n\tau_*^\alpha]+1}^{[n\tau]}\Bigl(F_i(\alpha_1)-F_i(\alpha_2)\Bigr).
\end{align}
Now, from
\begin{align*}
F_i(\alpha_1)-F_i(\alpha_2)
&= F_i(\alpha_1)-F_i(\alpha_2)-\EE_{\alpha_2^*}[F_i(\alpha_1)-F_i(\alpha_2)|\GG]\\
&\quad+\tr\left(A_{i-1}^{-1}(\alpha_1)A_{i-1}(\alpha_2)-I_d\right)
-\log\det A_{i-1}^{-1}(\alpha_1)A_{i-1}(\alpha_2)\\
&\quad-\tr\Bigl(
\left(A_{i-1}^{-1}(\alpha_1)-A_{i-1}^{-1}(\alpha_2)\right)
\left(A_{i-1}(\alpha_2)-h^{-1}\EE_{\alpha_2^*}[(\DeX)^{\otimes2}|\GG]\right)
\Bigr),
\end{align*}
we obtain
\begin{align*}
&\Phi_n(\tau:\alpha_1,\alpha_2)-\Phi_n(\tau_*^\alpha:\alpha_1,\alpha_2)
\\
&=\sum_{i=[n\tau_*^\alpha]+1}^{[n\tau]}
\Bigl(
F_i(\alpha_1)-F_i(\alpha_2)-\EE_{\alpha_2^*}[F_i(\alpha_1)-F_i(\alpha_2)|\GG]
\Bigr)\\
&\quad+\sum_{i=[n\tau_*^\alpha]+1}^{[n\tau]}
\Bigl(
\tr\left(A_{i-1}^{-1}(\alpha_1)A_{i-1}(\alpha_2)-I_d\right)
-\log\det A_{i-1}^{-1}(\alpha_1)A_{i-1}(\alpha_2)
\Bigr)\\
&\quad-\sum_{i=[n\tau_*^\alpha]+1}^{[n\tau]}
\tr\Bigl(
\left(A_{i-1}^{-1}(\alpha_1)-A_{i-1}^{-1}(\alpha_2)\right)
\left(A_{i-1}(\alpha_2)-h^{-1}\EE_{\alpha_2^*}[(\DeX)^{\otimes2}|\GG]\right)
\Bigr)\\
&=:
{\mathcal M}_n^\alpha(\tau:\alpha_1,\alpha_2)
+{\mathcal A}_n^\alpha(\tau:\alpha_1,\alpha_2)
+{\varrho}_n^\alpha(\tau:\alpha_1,\alpha_2).
\end{align*}
Let $D_{n,M}^\alpha=\{\tau\in[0,1]|n\Dea^2(\tau-\tau_*^\alpha)>M\}$.
For all $\delta>0$, 
we have
\begin{align*}
P\left(n\Dea^2(\hat\tau_n-\tau_*^\alpha)>M\right)
&\le
P\left(
\inf_{\tau\in D_{n,M}^\alpha}\Phi_n(\tau:\hat\alpha_1,\hat\alpha_2)
\le\Phi_n(\tau_*^\alpha:\hat\alpha_1,\hat\alpha_2)
\right)\\
&=
P\left(
\inf_{\tau\in D_{n,M}^\alpha}\Bigl(\Phi_n(\tau:\hat\alpha_1,\hat\alpha_2)
-\Phi_n(\tau_*^\alpha:\hat\alpha_1,\hat\alpha_2)\Bigr)
\le0
\right)\\
&=P\left(
\inf_{\tau\in D_{n,M}^\alpha}
\Bigl({\mathcal M}_n^\alpha(\tau:\hat\alpha_1,\hat\alpha_2)
+{\mathcal A}_n^\alpha(\tau:\hat\alpha_1,\hat\alpha_2)
+{\varrho}_n^\alpha(\tau:\hat\alpha_1,\hat\alpha_2)
\Bigr)\le0
\right)\\
&\le P\left(
\inf_{\tau\in D_{n,M}^\alpha}
\frac{
{\mathcal M}_n^\alpha(\tau:\hat\alpha_1,\hat\alpha_2)
+{\mathcal A}_n^\alpha(\tau:\hat\alpha_1,\hat\alpha_2)
+{\varrho}_n^\alpha(\tau:\hat\alpha_1,\hat\alpha_2)
}{\Dea^2([n\tau]-[n\tau_*^\alpha])}
\le0
\right)\\
&\le
P\left(
\inf_{\tau\in D_{n,M}^\alpha}
\frac{
{\mathcal M}_n^\alpha(\tau:\hat\alpha_1,\hat\alpha_2)
}{\Dea^2([n\tau]-[n\tau_*^\alpha])}
\le-\delta
\right) 
+P\left(
\inf_{\tau\in D_{n,M}^\alpha}
\frac{
{\mathcal A}_n^\alpha(\tau:\hat\alpha_1,\hat\alpha_2)
}{\Dea^2([n\tau]-[n\tau_*^\alpha])}
\le2\delta
\right)\\
&\quad+P\left(
\inf_{\tau\in D_{n,M}^\alpha}
\frac{
{\varrho}_n^\alpha(\tau:\hat\alpha_1,\hat\alpha_2)
}{\Dea^2([n\tau]-[n\tau_*^\alpha])}
\le-\delta
\right)\\
&\le
P\left(
\sup_{\tau\in D_{n,M}^\alpha}
\frac{
|{\mathcal M}_n^\alpha(\tau:\hat\alpha_1,\hat\alpha_2)|
}{\Dea^2([n\tau]-[n\tau_*^\alpha])}
\ge\delta
\right) 
+P\left(
\inf_{\tau\in D_{n,M}^\alpha}
\frac{
{\mathcal A}_n^\alpha(\tau:\hat\alpha_1,\hat\alpha_2)
}{\Dea^2([n\tau]-[n\tau_*^\alpha])}
\le2\delta
\right)\\
&\quad+P\left(
\sup_{\tau\in D_{n,M}^\alpha}
\frac{
|{\varrho}_n^\alpha(\tau:\hat\alpha_1,\hat\alpha_2)|
}{\Dea^2([n\tau]-[n\tau_*^\alpha])}
\ge\delta
\right)\\
&=:
P_{1,n}^\alpha+P_{2,n}^\alpha+P_{3,n}^\alpha.
\end{align*}
[i] Evaluation of $P_{1,n}^\alpha$. 
Choose $\epsilon>0$ as you like.  
Let 
$\mathcal O_{\alpha}$ be an open neighborhood of $\alpha$.
Because $\partial_\alpha F_i(\alpha)$ is continuous with respect to $\alpha\in\Theta_A$, 
we can choose $\bar\alpha\in\mathcal O_{\hat\alpha_2}$ so that 
\begin{align*}
{\mathcal M}_n^\alpha(\tau:\hat\alpha_1,\hat\alpha_2)
&=\sum_{i=[n\tau_*^\alpha]+1}^{[n\tau]}
\left(
F_i(\hat\alpha_1)-F_i(\hat\alpha_2)
-\left.\EE_{\alpha_2^*}[F_i(\alpha_1)-F_i(\alpha_2)|\GG]\right|_{\alpha_k=\hat\alpha_k}
\right)\\
&=\sum_{i=[n\tau_*^\alpha]+1}^{[n\tau]}
\left(
\partial_\alpha F_i(\bar\alpha)
-\left.\EE_{\alpha_2^*}[\partial_\alpha F_i(\alpha)|\GG]\right|_{\alpha=\bar\alpha}
\right)
(\hat\alpha_1-\hat\alpha_2).
\end{align*}
If $\hat\alpha_k\in\mathcal O_{\alpha_0}$, then
\begin{align*}
|{\mathcal M}_n^\alpha(\tau:\hat\alpha_1,\hat\alpha_2)|
&\le
\left|\sum_{i=[n\tau_*^\alpha]+1}^{[n\tau]}
\left(
\partial_\alpha F_i(\bar\alpha)
-\left.\EE_{\alpha_2^*}[\partial_\alpha F_i(\alpha)|\GG]\right|_{\alpha=\bar\alpha}
\right)\right|
|\hat\alpha_1-\hat\alpha_2|\\
&\le
\sup_{\alpha\in\Theta_A}\left|\sum_{i=[n\tau_*^\alpha]+1}^{[n\tau]}
\left(
\partial_\alpha F_i(\alpha)-\EE_{\alpha_2^*}[\partial_\alpha F_i(\alpha)|\GG]
\right)\right|
|\hat\alpha_1-\hat\alpha_2|\\
&=:
\sup_{\alpha\in\Theta_A}|\mathbb M_n^\alpha(\tau:\alpha)||\hat\alpha_1-\hat\alpha_2|.
\end{align*}
Hence we have
\begin{align}
P_{1,n}^\alpha
&=
P\left(
\sup_{\tau\in D_{n,M}^\alpha}
\frac{
|{\mathcal M}_n^\alpha(\tau:\hat\alpha_1,\hat\alpha_2)|
}{\Dea^2([n\tau]-[n\tau_*^\alpha])}
\ge\delta
\right)
\nonumber\\ 
&\le
P\left(
\sup_{\tau\in D_{n,M}^\alpha}
\frac{|{\mathcal M}_n^\alpha(\tau:\hat\alpha_1,\hat\alpha_2)|}
{\Dea^2([n\tau]-[n\tau_*^\alpha])}
\ge\delta,\ 
|\hat\alpha_1-\hat\alpha_2|\le2\Dea,\ 
\hat\alpha_1,\hat\alpha_2\in\mathcal O_{\alpha_0}
\right)
\nonumber\\
&\qquad+P(|\hat\alpha_1-\hat\alpha_2|>2\Dea)
+P(\hat\alpha_1\notin\mathcal O_{\alpha_0})
+P(\hat\alpha_2\notin\mathcal O_{\alpha_0})
\nonumber\\
&\le
P\left(
\sup_{\tau\in D_{n,M}^\alpha}
\frac{
\sup_{\alpha\in\Theta_A}|{\mathbb M}_n^\alpha(\tau:\alpha)|
}{\Dea^2([n\tau]-[n\tau_*^\alpha])}
|\hat\alpha_1-\hat\alpha_2|
\ge\delta,\ 
|\hat\alpha_1-\hat\alpha_2|\le2\Dea
\right)
\nonumber\\
&\qquad+P(|\hat\alpha_1-\hat\alpha_2|>2\Dea)
+P(\hat\alpha_1\notin\mathcal O_{\alpha_0})
+P(\hat\alpha_2\notin\mathcal O_{\alpha_0})
\nonumber\\
&\le
P\left(
\sup_{\tau\in D_{n,M}^\alpha}
\frac{
2 \sup_{\alpha\in\Theta_A}|{\mathbb M}_n^\alpha(\tau:\alpha)|
}{\Dea([n\tau]-[n\tau_*^\alpha])}
\ge\delta
\right)
\nonumber\\
&\qquad+P(|\hat\alpha_1-\hat\alpha_2|>2\Dea)
+P(\hat\alpha_1\notin\mathcal O_{\alpha_0})
+P(\hat\alpha_2\notin\mathcal O_{\alpha_0}).
\label{eq7.999}
\end{align}
By the uniform version on the H\'ajek-Renyi inequality in Lemma 2 
of Iacus and Yoshida (2012), 
we obtain
\begin{align}
&P\left(
\sup_{\tau\in D_{n,M}^\alpha}
\frac{
2 \sup_{\alpha\in\Theta_A}|{\mathbb M}_n^\alpha(\tau:\alpha)|
}{\Dea([n\tau]-[n\tau_*^\alpha])}
\ge\delta
\right)
\nonumber\\
&=
P\left(
\sup_{\tau\in D_{n,M}^\alpha}
\frac{
2 \sup_{\alpha\in\Theta_A}|{\mathbb M}_n^\alpha(\tau:\alpha)|
}{[n\tau]-[n\tau_*^\alpha]}
\ge\delta\Dea
\right)
\nonumber\\
&\le
P\left(
\max_{j>M/\Dea^2-1}
\frac{2}{j}
\sup_{\alpha}\left|
\sum_{i=[n\tau_*^\alpha]+1}^{[n\tau_*^\alpha]+j}
\left(
\partial_\alpha F_i(\alpha)-
\EE[\partial_\alpha F_i(\alpha)|\GG]
\right)
\right|
\ge\delta\Dea
\right)
\nonumber\\
&\le
\sum_{j>M/\Dea^2-1}
\frac{C}{(\delta\Dea j)^2}
\nonumber\\
&\le
\frac{C'}{(\delta\Dea)^2}\frac{\Dea^2}{M}
=\frac{C'}{\delta^2M}
=:\gamma_\alpha(M),
\label{eq7.998}
\end{align}
Noting that if
$|\hat\alpha_k-\alpha_k^*|\le\Dea/2$ holds, 
then 
$|\hat\alpha_1-\hat\alpha_2|
\le
|\hat\alpha_1-\alpha_1^*|
+|\alpha_1^*-\alpha_2^*|
+|\hat\alpha_2-\alpha_2^*|
\le 2\Dea$, 
i.e., 
\begin{align*}
\{|\hat\alpha_1-\hat\alpha_2|>2\Dea\}
\subset
\left\{|\hat\alpha_1-\alpha_1^*|>\frac{\Dea}{2}\right\}
\cup\left\{|\hat\alpha_2-\alpha_2^*|>\frac{\Dea}{2}\right\},
\end{align*}
and for sufficiently large $n$ so that 
$P\left(
|\hat\alpha_k-\alpha_k^*|>\Dea/2
\right)<\epsilon/2$
because of $\Dea^{-1}(\hat\alpha_k-\alpha_k^*)=o_p(1)$,
we see 
\begin{align}
P(|\hat\alpha_1-\hat\alpha_2|>2\Dea)
&\le
P\left(|\hat\alpha_1-\alpha_1^*|>\frac{\Dea}{2}
\right)
+P\left(
|\hat\alpha_2-\alpha_2^*|>\frac{\Dea}{2}
\right)
<\epsilon.
\label{eq7.997}
\end{align}
From \textbf{[C6-I]} and \textbf{[A2-I]}, 
we have $\hat\alpha_k\pto\alpha_0$ as $n\to\infty$ and 
$P(\hat\alpha_k\notin\mathcal O_{\alpha_0})<\epsilon$ for large $n$.
Therefore, from \eqref{eq7.999}-\eqref{eq7.997} and this, 
we have
$P_{1,n}^\alpha\le\gamma_\alpha(M)+3\epsilon$ for large $n$.

[ii] Evaluation of $P_{2,n}^\alpha$.
If $\hat\alpha_k\in\mathcal O_{\alpha_0}$, then we have
\begin{align*}
&\tr\left(
A_{i-1}^{-1}(\hat\alpha_1)A_{i-1}(\hat\alpha_2)-I_d 
\right)
-\log\det A_{i-1}^{-1}(\hat\alpha_1)A_{i-1}(\hat\alpha_2)\\
&=
\tr\left(
A_{i-1}^{-1}(\hat\alpha_2)A_{i-1}(\hat\alpha_2)-I_d 
\right)
-\log\det A_{i-1}^{-1}(\hat\alpha_2)A_{i-1}(\hat\alpha_2)\\
&\qquad+\left.\partial_\alpha
\Bigl(
\tr\left(
A_{i-1}^{-1}(\alpha)A_{i-1}(\hat\alpha_2)
\right)
-\log\det A_{i-1}^{-1}(\alpha)A_{i-1}(\hat\alpha_2)
\Bigr)
\right|_{\alpha=\hat\alpha_2}
(\hat\alpha_1-\hat\alpha_2)\\
&\qquad+\frac12\left.\partial_\alpha^2
\Bigl(
\tr\left(
A_{i-1}^{-1}(\alpha)A_{i-1}(\hat\alpha_2)
\right)
-\log\det A_{i-1}^{-1}(\alpha)A_{i-1}(\hat\alpha_2)
\Bigr)
\right|_{\alpha=\hat\alpha_2}
\otimes(\hat\alpha_1-\hat\alpha_2)^{\otimes2}\\
&\qquad+
\frac{1}{3!}
\left.\partial_\alpha^3
\Bigl(
\tr\left(
A_{i-1}^{-1}(\alpha)A_{i-1}(\hat\alpha_2)
\right)
-\log\det A_{i-1}^{-1}(\alpha)A_{i-1}(\hat\alpha_2)
\Bigr)
\right|_{\alpha=\hat\alpha_2}
\otimes(\hat\alpha_1-\hat\alpha_2)^{\otimes3}\\
&\qquad+
\int_0^1\frac{(1-u)^3}{3!}
\left.\partial_\alpha^4
\Bigl(
\tr\left(
A_{i-1}^{-1}(\alpha)A_{i-1}(\hat\alpha_2)
\right)
-\log\det A_{i-1}^{-1}(\alpha)A_{i-1}(\hat\alpha_2)
\Bigr)
\right|_{\alpha=\hat\alpha_2+u(\hat\alpha_1-\hat\alpha_2)}\dd u
\otimes(\hat\alpha_1-\hat\alpha_2)^{\otimes4}\\
&=\frac12
\Xi_{i-1}^\alpha(\hat\alpha_2)
\otimes(\hat\alpha_1-\hat\alpha_2)^{\otimes2}\\
&\qquad+
\frac{1}{3!}
\left.\partial_\alpha^3
\Bigl(
\tr\left(
A_{i-1}^{-1}(\alpha)A_{i-1}(\hat\alpha_2)
\right)
-\log\det A_{i-1}^{-1}(\alpha)A_{i-1}(\hat\alpha_2)
\Bigr)
\right|_{\alpha=\hat\alpha_2}
\otimes(\hat\alpha_1-\hat\alpha_2)^{\otimes3}\\
&\qquad+
\int_0^1\frac{(1-u)^3}{3!}
\left.\partial_\alpha^4
\Bigl(
\tr\left(
A_{i-1}^{-1}(\alpha)A_{i-1}(\hat\alpha_2)
\Bigr)
-\log\det A_{i-1}^{-1}(\alpha)A_{i-1}(\hat\alpha_2)
\right)
\right|_{\alpha=\hat\alpha_2+u(\hat\alpha_1-\hat\alpha_2)}\dd u
\otimes(\hat\alpha_1-\hat\alpha_2)^{\otimes4}\\
&=\frac12
\Xi_{i-1}^\alpha(\alpha_0)
\otimes(\hat\alpha_1-\hat\alpha_2)^{\otimes2}
+\frac12
\partial_\alpha\Xi_{i-1}^\alpha(\alpha_0)
\otimes(\hat\alpha_1-\hat\alpha_2)^{\otimes2}
\otimes(\hat\alpha_2-\alpha_0)
\\
&\qquad+
\frac12
\int_0^1(1-u)
\partial_\alpha^2\Xi_{i-1}^\alpha(\alpha_0+u(\hat\alpha_2-\alpha_0))\dd u
\otimes(\hat\alpha_1-\hat\alpha_2)^{\otimes2}
\otimes(\hat\alpha_2-\alpha_0)^{\otimes2}
\\
&\qquad+
\frac{1}{3!}
\left.\partial_\alpha^3
\Bigl(
\tr\left(
A_{i-1}^{-1}(\alpha)A_{i-1}(\hat\alpha_2)
\right)
-\log\det A_{i-1}^{-1}(\alpha)A_{i-1}(\hat\alpha_2)
\Bigr)
\right|_{\alpha=\hat\alpha_2}
\otimes(\hat\alpha_1-\hat\alpha_2)^{\otimes3}\\
&\qquad+
\int_0^1\frac{(1-u)^3}{3!}
\left.\partial_\alpha^4
\Bigl(
\tr\left(
A_{i-1}^{-1}(\alpha)A_{i-1}(\hat\alpha_2)
\right)
-\log\det A_{i-1}^{-1}(\alpha)A_{i-1}(\hat\alpha_2)
\Bigr)
\right|_{\alpha=\hat\alpha_2+u(\hat\alpha_1-\hat\alpha_2)}\dd u
\otimes(\hat\alpha_1-\hat\alpha_2)^{\otimes4}\\
&\ge
\left(
\frac12\lambda_1[\Xi_{i-1}^\alpha(\alpha_0)]+r_{i-1}
\right)
|\hat\alpha_1-\hat\alpha_2|^2,
\end{align*}
where 
$\lambda_1[N]$ denotes the minimum eigenvalue of a symmetric matrix $N$,
and $r_{i-1}$ satisfies 
\begin{align*}
\sup_{\tau\in D_{n,M}^\alpha}
\left|
\frac{1}{[n\tau]-[n\tau_*^\alpha]}
\sum_{i=[n\tau_*^\alpha]+1}^{[n\tau]}r_{i-1}
\right|
=o_p(1)
\end{align*}
from \textbf{[A1-I]} and \textbf{[A2-I]}. 
Therefore we obtain
\begin{align}
P_{2,n}^\alpha
&=P\left(
\inf_{\tau\in D_{n,M}^\alpha}
\frac{
{\mathcal A}_n^\alpha(\tau:\hat\alpha_1,\hat\alpha_2)
}{\Dea^2([n\tau]-[n\tau_*^\alpha])}
\le2\delta
\right)
\nonumber\\
&\le
P\left(
\inf_{\tau\in D_{n,M}^\alpha}
\frac{
{\mathcal A}_n^\alpha(\tau:\hat\alpha_1,\hat\alpha_2)
}{\Dea^2([n\tau]-[n\tau_*^\alpha])}
\le2\delta,\ 
|\hat\alpha_1-\hat\alpha_2|\ge\frac{\Dea}{2},\ 
\hat\alpha_k\in\mathcal O_{\alpha_0}
\right)
\nonumber\\
&\qquad+P\left(
|\hat\alpha_1-\hat\alpha_2|<\frac{\Dea}{2}
\right)
+P\left(
\hat\alpha_1\notin\mathcal O_{\alpha_0}
\right)
+P\left(
\hat\alpha_2\notin\mathcal O_{\alpha_0}
\right)
\nonumber\\
&\le
P\left(
\inf_{\tau\in D_{n,M}^\alpha}
\frac{1}{\Dea^2([n\tau]-[n\tau_*^\alpha])}
\sum_{i=[n\tau_*^\alpha]+1}^{[n\tau]}
\left(
\frac12\lambda_1[\Xi_{i-1}^\alpha(\alpha_0)]+r_{i-1}
\right)
|\hat\alpha_1-\hat\alpha_2|^2
\le2\delta,\ 
|\hat\alpha_1-\hat\alpha_2|\ge\frac{\Dea}{2}
\right)
\nonumber\\
&\qquad+P\left(
|\hat\alpha_1-\hat\alpha_2|<\frac{\Dea}{2}
\right)
+P\left(
\hat\alpha_1\notin\mathcal O_{\alpha_0}
\right)
+P\left(
\hat\alpha_2\notin\mathcal O_{\alpha_0}
\right)
\nonumber\\
&\le
P\left(
\inf_{\tau\in D_{n,M}^\alpha}
\frac{1}{[n\tau]-[n\tau_*^\alpha]}
\sum_{i=[n\tau_*^\alpha]+1}^{[n\tau]}
\left(
\frac12\lambda_1[\Xi_{i-1}^\alpha(\alpha_0)]+r_{i-1}
\right)
\le8\delta 
\right)
\nonumber\\
&\qquad+P\left(
|\hat\alpha_1-\hat\alpha_2|<\frac{\Dea}{2}
\right)
+P\left(
\hat\alpha_1\notin\mathcal O_{\alpha_0}
\right)
+P\left(
\hat\alpha_2\notin\mathcal O_{\alpha_0}
\right).
\label{eq7.989}
\end{align}
According to \textbf{[A2-I]},
if we set
\begin{align*}
\delta=\frac{1}{19}\int_{\mathbb R^d}\lambda_1[\Xi^\alpha(x,\alpha_0)]
\dd\mu_{\alpha_0}(x)>0,
\end{align*}
then for large $n$, 
\begin{align}
&P\left(
\inf_{\tau\in D_{n,M}^\alpha}
\frac{1}{[n\tau]-[n\tau_*^\alpha]}
\sum_{i=[n\tau_*^\alpha]+1}^{[n\tau]}
\left(
\frac12\lambda_1[\Xi_{i-1}^\alpha(\alpha_0)]+r_{i-1}
\right)
\le8\delta
\right)
\nonumber\\
&\le
P\left(
\inf_{\tau\in D_{n,M}^\alpha}
\frac{1}{[n\tau]-[n\tau_*^\alpha]}
\sum_{i=[n\tau_*^\alpha]+1}^{[n\tau]}
\frac12\lambda_1[\Xi_{i-1}^\alpha(\alpha_0)]
\le9\delta
\right)
+
P\left(
\inf_{\tau\in D_{n,M}^\alpha}
\frac{1}{[n\tau]-[n\tau_*^\alpha]}
\sum_{i=[n\tau_*^\alpha]+1}^{[n\tau]}
r_{i-1}
\le-\delta
\right)
\nonumber\\
&\le
P\left(
\inf_{\tau\in D_{n,M}^\alpha}
\frac{1}{[n\tau]-[n\tau_*^\alpha]}
\sum_{i=[n\tau_*^\alpha]+1}^{[n\tau]}
\lambda_1[\Xi_{i-1}^\alpha(\alpha_0)]
\le18\delta
\right)
+
P\left(
\sup_{\tau\in D_{n,M}^\alpha}
\left|
\frac{1}{[n\tau]-[n\tau_*^\alpha]}
\sum_{i=[n\tau_*^\alpha]+1}^{[n\tau]}
r_{i-1}
\right|
\ge\delta
\right)
\nonumber\\
&\le
P\left(
\inf_{\tau\in D_{n,M}^\alpha}
\left(
\frac{1}{[n\tau]-[n\tau_*^\alpha]}
\sum_{i=[n\tau_*^\alpha]+1}^{[n\tau]}
\lambda_1[\Xi_{i-1}^\alpha(\alpha_0)]
-19\delta
\right)
\le-\delta
\right)
+\epsilon
\nonumber\\
&\le
P\left(
\sup_{\tau\in D_{n,M}^\alpha}
\left|
\frac{1}{[n\tau]-[n\tau_*^\alpha]}
\sum_{i=[n\tau_*^\alpha]+1}^{[n\tau]}
\lambda_1[\Xi_{i-1}^\alpha(\alpha_0)]
-19\delta
\right|
\ge\delta
\right)
+\epsilon
\nonumber\\
&\le
P\left(
\sup_{k> M/\Dea^2-1}
\left|
\frac{1}{k}
\sum_{i=[n\tau_*^\alpha]+1}^{[n\tau_*^\alpha]+k}
\lambda_1[\Xi_{i-1}^\alpha(\alpha_0)]
-19\delta
\right|
\ge\delta
\right)
+\epsilon
\nonumber\\
&\le
P\left(
\max_{[n^{1/r}]\le k\le n}
\left|
\frac{1}{k}
\sum_{i=[n\tau_*^\alpha]+1}^{[n\tau_*^\alpha]+k}
\lambda_1[\Xi_{i-1}^\alpha(\alpha_0)]
-19\delta
\right|
\ge\delta
\right)
+\epsilon
\nonumber\\
&\le 2\epsilon.
\label{eq7.988}
\end{align}

Noting that if
$|\hat\alpha_k-\alpha_k^*|\le\Dea/4$ holds, 
then 
$\Dea=|\alpha_1^*-\alpha_2^*|
\le
|\alpha_1^*-\hat\alpha_1|
+|\hat\alpha_1-\hat\alpha_2|
+|\hat\alpha_2-\alpha_2^*|
\le |\hat\alpha_1-\hat\alpha_2|+\Dea/2$, 
i.e., 
\begin{align*}
\left\{|\hat\alpha_1-\hat\alpha_2|<\frac\Dea2\right\}
\subset
\left\{|\hat\alpha_1-\alpha_1^*|>\frac{\Dea}{4}\right\}
\cup\left\{|\hat\alpha_2-\alpha_2^*|>\frac{\Dea}{4}\right\}
\end{align*}
and for sufficiently large $n$ so that 
$P\left(
|\hat\alpha_k-\alpha_k^*|>\Dea/4
\right)<\epsilon/2$
because of $\Dea^{-1}(\hat\alpha_k-\alpha_k^*)=o_p(1)$,
we see
\begin{align}
P\left(
|\hat\alpha_1-\hat\alpha_2|<\frac{\Dea}{2}
\right)
&\le
P\left(|\hat\alpha_1-\alpha_1^*|>\frac{\Dea}{4}
\right)
+P\left(
|\hat\alpha_2-\alpha_2^*|>\frac{\Dea}{4}
\right)
<\epsilon.
\label{eq7.987}
\end{align}
Therefore, from \eqref{eq7.989}-\eqref{eq7.987}, we obtain
$P_{2,n}^\alpha\le5\epsilon$ for large $n$.


[iii] Evaluation of $P_{3,n}^\alpha$.
We have, for large $n$,
\begin{align*}
&\tr
\Bigl(
\left(
A_{i-1}^{-1}(\hat\alpha_1)-A_{i-1}^{-1}(\hat\alpha_2)
\right)
\left(
A_{i-1}(\hat\alpha_2)-h^{-1}\EE_{\alpha_2^*}[(\DeX)^{\otimes2}|\GG]
\right)
\Bigr)\\
&=
\tr
\Bigl(
\left(
A_{i-1}^{-1}(\hat\alpha_1)-A_{i-1}^{-1}(\hat\alpha_2)
\right)
\left(
A_{i-1}(\hat\alpha_2)-A_{i-1}(\alpha_2^*)+\Ri
\right)
\Bigr)\\
&\le
\tr
\Bigl(
\left(
A_{i-1}^{-1}(\hat\alpha_1)-A_{i-1}^{-1}(\hat\alpha_2)
\right)
\left(
A_{i-1}(\hat\alpha_2)-A_{i-1}(\alpha_2^*)
\right)
\Bigr)
+
\left|
A_{i-1}^{-1}(\hat\alpha_1)-A_{i-1}^{-1}(\hat\alpha_2)
\right|
\left|\Ri\right|
\\
&\le
\left[
\tr
\Bigl(
\partial_{\alpha^{\ell_1}}A_{i-1}^{-1}(\hat\alpha_2)
(A_{i-1}(\hat\alpha_2)-A_{i-1}(\alpha_2^*))
\Bigr)
\right]_{\ell_1}
(\hat\alpha_1-\hat\alpha_2)\\
&\qquad+\int_0^1
(1-u)
\left[
\tr
\Bigl(
\partial_{\alpha^{\ell_2}}\partial_{\alpha^{\ell_1}}
A_{i-1}^{-1}(\hat\alpha_2+u(\hat\alpha_1-\hat\alpha_2))
(A_{i-1}(\hat\alpha_2)-A_{i-1}(\alpha_2^*))
\Bigr)
\right]_{\ell_1,\ell_2}
\dd u
\otimes(\hat\alpha_1-\hat\alpha_2)^{\otimes2}\\
&\qquad+
\left|
\int_0^1
\partial_{\alpha}A_{i-1}^{-1}(\hat\alpha_2+u(\hat\alpha_1-\hat\alpha_2))
\dd u
\right|
|\Ri| |\hat\alpha_1-\hat\alpha_2|\\
&=
\left[
\tr\left(
\partial_{\alpha^{\ell_1}}A_{i-1}^{-1}(\hat\alpha_2)
\partial_{\alpha^{\ell_2}}A_{i-1}(\alpha_2^*)
\right)
\right]_{\ell_1,\ell_2}
\otimes(\hat\alpha_1-\hat\alpha_2)
\otimes(\hat\alpha_2-\alpha_2^*)\\
&\qquad+\int_0^1
(1-u)
\left[
\tr\Bigl(
\partial_{\alpha^{\ell_1}}
A_{i-1}^{-1}(\hat\alpha_2)
\partial_{\alpha^{\ell_3}}\partial_{\alpha^{\ell_2}}
A_{i-1}(\alpha_2^*+u(\hat\alpha_2-\alpha_2^*))
\Bigr)
\right]_{\ell_1,\ell_2,\ell_3}
\dd u
\otimes(\hat\alpha_1-\hat\alpha_2)
\otimes(\hat\alpha_2-\alpha_2^*)^{\otimes2}\\
&\qquad+\int_0^1(1-v)\int_0^1
(1-u)
\left[
\tr
\Bigl(
\partial_{\alpha^{\ell_2}}\partial_{\alpha^{\ell_1}}
A_{i-1}^{-1}(\hat\alpha_2+u(\hat\alpha_1-\hat\alpha_2))
\partial_{\alpha^{\ell_3}}A_{i-1}(\alpha_2^*+v(\hat\alpha_2-\alpha_2^*))
\Bigr)
\right]_{\ell_1,\ell_2,\ell_3}
\dd u\dd v\\
&\qquad\qquad
\otimes(\hat\alpha_1-\hat\alpha_2)^{\otimes2}
\otimes(\hat\alpha_2-\alpha_2^*)\\
&\qquad+
\left|
\int_0^1
\partial_{\alpha}A_{i-1}^{-1}(\hat\alpha_2+u(\hat\alpha_1-\hat\alpha_2))
\dd u
\right|
|\Ri| |\hat\alpha_1-\hat\alpha_2|\\
&=
\Xi_{i-1}^\alpha(\alpha_2^*)
\otimes(\hat\alpha_1-\hat\alpha_2)
\otimes(\hat\alpha_2-\alpha_2^*)\\
&\qquad+\int_0^1
\left[
\tr\Bigl(
\partial_{\alpha^{\ell_3}}\partial_{\alpha^{\ell_1}}
A_{i-1}^{-1}(\alpha_2^*+u(\hat\alpha_2-\alpha_2^*))
\partial_{\alpha^{\ell_2}}A_{i-1}(\alpha_2^*)
\Bigr)
\right]_{\ell_1,\ell_2,\ell_3}
\dd u
\otimes(\hat\alpha_1-\hat\alpha_2)
\otimes(\hat\alpha_2-\alpha_2^*)^{\otimes2}\\
&\qquad+\int_0^1
(1-u)
\left[
\tr\Bigl(
\partial_{\alpha^{\ell_1}}
A_{i-1}^{-1}(\hat\alpha_2)
\partial_{\alpha^{\ell_3}}\partial_{\alpha^{\ell_2}}
A_{i-1}(\alpha_2^*+u(\hat\alpha_2-\alpha_2^*))
\Bigr)
\right]_{\ell_1,\ell_2,\ell_3}
\dd u
\otimes(\hat\alpha_1-\hat\alpha_2)
\otimes(\hat\alpha_2-\alpha_2^*)^{\otimes2}\\
&\qquad+\int_0^1(1-v)\int_0^1
(1-u)
\left[
\tr\Bigl(
\partial_{\alpha^{\ell_2}}\partial_{\alpha^{\ell_1}}
A_{i-1}^{-1}(\hat\alpha_2+u(\hat\alpha_1-\hat\alpha_2))
\partial_{\alpha^{\ell_3}}A_{i-1}(\alpha_2^*+v(\hat\alpha_2-\alpha_2^*))
\Bigr)
\right]_{\ell_1,\ell_2,\ell_3}
\dd u\dd v\\
&\qquad\qquad
\otimes(\hat\alpha_1-\hat\alpha_2)^{\otimes2}
\otimes(\hat\alpha_2-\alpha_2^*)\\
&\qquad+
\left|
\int_0^1
\partial_{\alpha}A_{i-1}^{-1}(\hat\alpha_2+u(\hat\alpha_1-\hat\alpha_2))
\dd u
\right|
|\Ri| |\hat\alpha_1-\hat\alpha_2|\\
&=
\Xi_{i-1}^\alpha(\alpha_0)
\otimes(\hat\alpha_1-\hat\alpha_2)
\otimes(\hat\alpha_2-\alpha_2^*)
+\int_0^1
\partial_{\alpha}\Xi_{i-1}^\alpha(\alpha_0+u(\alpha_2^*-\alpha_0))
\dd u
\otimes(\hat\alpha_1-\hat\alpha_2)
\otimes(\hat\alpha_2-\alpha_2^*)
\otimes(\alpha_2^*-\alpha_0)
\\
&\qquad+\int_0^1
\left[
\tr\Bigl(
\partial_{\alpha^{\ell_3}}\partial_{\alpha^{\ell_1}}
A_{i-1}^{-1}(\alpha_2^*+u(\hat\alpha_2-\alpha_2^*))
\partial_{\alpha^{\ell_2}}A_{i-1}(\alpha_2^*)
\Bigr)
\right]_{\ell_1,\ell_2,\ell_3}
\dd u
\otimes(\hat\alpha_1-\hat\alpha_2)
\otimes(\hat\alpha_2-\alpha_2^*)^{\otimes2}\\
&\qquad+\int_0^1
(1-u)
\left[
\tr\Bigl(
\partial_{\alpha^{\ell_1}}
A_{i-1}^{-1}(\hat\alpha_2)
\partial_{\alpha^{\ell_3}}\partial_{\alpha^{\ell_2}}
A_{i-1}(\alpha_2^*+u(\hat\alpha_2-\alpha_2^*))
\Bigr)
\right]_{\ell_1,\ell_2,\ell_3}
\dd u
\otimes(\hat\alpha_1-\hat\alpha_2)
\otimes(\hat\alpha_2-\alpha_2^*)^{\otimes2}\\
&\qquad+\int_0^1(1-v)\int_0^1
(1-u)
\left[
\tr\Bigl(
\partial_{\alpha^{\ell_2}}\partial_{\alpha^{\ell_1}}
A_{i-1}^{-1}(\hat\alpha_2+u(\hat\alpha_1-\hat\alpha_2))
\partial_{\alpha^{\ell_3}}A_{i-1}(\alpha_2^*+v(\hat\alpha_2-\alpha_2^*))
\Bigr)
\right]_{\ell_1,\ell_2,\ell_3}
\dd u\dd v\\
&\qquad\qquad
\otimes(\hat\alpha_1-\hat\alpha_2)^{\otimes2}
\otimes(\hat\alpha_2-\alpha_2^*)\\
&\qquad+
\left|
\int_0^1
\partial_{\alpha}A_{i-1}^{-1}(\hat\alpha_2+u(\hat\alpha_1-\hat\alpha_2))
\dd u
\right|
|\Ri| |\hat\alpha_1-\hat\alpha_2|\\
&\le
\Xi_{i-1}^\alpha(\alpha_0)
\otimes(\hat\alpha_1-\hat\alpha_2)
\otimes(\hat\alpha_2-\alpha_2^*)
\\
&\qquad+
\frac{\Dea^2}{\sqrt n}
\left|\int_0^1
\partial_{\alpha}\Xi_{i-1}^\alpha(\alpha_0+u(\alpha_2^*-\alpha_0))
\dd u
\right|
\Dea^{-1}|\hat\alpha_1-\hat\alpha_2|
\sqrt n|\hat\alpha_2-\alpha_2^*|
\Dea^{-1}|\alpha_2^*-\alpha_0|
\\
&\qquad+
\frac{\Dea}{n}
\left|
\int_0^1
\left[
\tr\Bigl(
\partial_{\alpha^{\ell_3}}\partial_{\alpha^{\ell_1}}
A_{i-1}^{-1}(\alpha_2^*+u(\hat\alpha_2-\alpha_2^*))
\partial_{\alpha^{\ell_2}}A_{i-1}(\alpha_2^*)
\Bigr)
\right]_{\ell_1,\ell_2,\ell_3}
\dd u
\right|
\Dea^{-1}|\hat\alpha_1-\hat\alpha_2|
\left(\sqrt{n}|\hat\alpha_2-\alpha_2^*|\right)^2\\
&\qquad+
\frac{\Dea}{n}
\left|
\int_0^1
(1-u)
\left[
\tr\Bigl(
\partial_{\alpha^{\ell_1}}
A_{i-1}^{-1}(\hat\alpha_2)
\partial_{\alpha^{\ell_3}}\partial_{\alpha^{\ell_2}}
A_{i-1}(\alpha_2^*+u(\hat\alpha_2-\alpha_2^*))
\Bigr)
\right]_{\ell_1,\ell_2,\ell_3}
\dd u
\right|
\Dea^{-1}|\hat\alpha_1-\hat\alpha_2|
\left(\sqrt{n}|\hat\alpha_2-\alpha_2^*|\right)^2\\
&\qquad+
\frac{\Dea^2}{\sqrt n}
\left|
\int_0^1(1-v)\int_0^1
(1-u)
\left[
\tr\Bigl(
\partial_{\alpha^{\ell_2}}\partial_{\alpha^{\ell_1}}
A_{i-1}^{-1}(\hat\alpha_2+u(\hat\alpha_1-\hat\alpha_2))
\partial_{\alpha^{\ell_3}}A_{i-1}(\alpha_2^*+v(\hat\alpha_2-\alpha_2^*))
\Bigr)
\right]_{\ell_1,\ell_2,\ell_3}
\dd u\dd v
\right|\\
&\qquad\qquad
\left(\Dea^{-1}|\hat\alpha_1-\hat\alpha_2|\right)^2
\sqrt{n}|\hat\alpha_2-\alpha_2^*|\\
&\qquad+
h\Dea\left|
\int_0^1
\partial_{\alpha}A_{i-1}^{-1}(\hat\alpha_2+u(\hat\alpha_1-\hat\alpha_2))
\dd u
\right|
\Ro \Dea^{-1}|\hat\alpha_1-\hat\alpha_2|\\
&\le
\frac{\Dea}{\sqrt n}\Xi_{i-1}^\alpha(\alpha_0)
\otimes\Dea^{-1}(\hat\alpha_1-\hat\alpha_2)
\otimes\sqrt{n}(\hat\alpha_2-\alpha_2^*)
+
\left(
\frac{\Dea^2}{\sqrt n}
+\frac{\Dea}{n}
+h\Dea\right)\Ro
\end{align*}
and
\begin{align*}
\sup_{\tau\in D_{n,M}^\alpha}
\frac{|\varrho_n^\alpha(\tau:\hat\alpha_1,\hat\alpha_2)|}
{\Dea^2([n\tau]-[n\tau_*^\alpha])}
&\le
\frac{1}{\sqrt{n}\Dea}
\sup_{\tau\in D_{n,M}^\alpha}
\left|
\frac{1}{[n\tau]-[n\tau_*^\alpha]}
\sum_{i=[n\tau_*^\alpha]+1}^{[n\tau]}
\Xi_{i-1}^\alpha(\alpha_0)
\right|
\Dea^{-1}|\hat\alpha_1-\hat\alpha_2|
\sqrt n|\hat\alpha_2-\alpha_2^*|
\\
&\qquad+
\left(
\frac{\Dea^2}{\sqrt n}
+\frac{\Dea}{n}
+h\Dea\right)
\sup_{\tau\in D_{n,M}^\alpha}
\frac{1}{\Dea^2([n\tau]-[n\tau_*^\alpha])}
\sum_{i=[n\tau_*^\alpha]+1}^{[n\tau]}
\Ro
\\
&\le
O_p\left(
\frac{1}{\sqrt{n}\Dea}
\right)
+O_p(\sqrt{n}\Dea^2)
+O_p(\Dea)
+O_p(T\Dea)
=o_p(1).
\end{align*}
Hence, we see 
\begin{align*}
P_{3,n}^\alpha
&\le
P\left(
\sup_{\tau\in D_{n,M}^\alpha}
\frac{|\varrho_n^\alpha(\tau:\hat\alpha_1,\hat\alpha_2)|}
{\Dea^2([n\tau]-[n\tau_*^\alpha])}
\ge\delta,\ 
\hat\alpha_1,\hat\alpha_2\in\mathcal O_{\alpha_0}
\right)
+P(\hat\alpha_1\notin\mathcal O_{\alpha_0})
+P(\hat\alpha_2\notin\mathcal O_{\alpha_0})\\
&\le 3\epsilon
\end{align*}
for large $n$.


[iv] From the evaluations in Steps [i]-[iii], we have
\begin{align*}
\varlimsup_{n\to\infty}
P(n\Dea^2(\hat\tau_n^\alpha-\tau_*^\alpha)>M)
\le
\gamma_\alpha(M)+11\epsilon
\end{align*}
for any $M\ge1$ and $\epsilon>0$. 
Therefore
\begin{align}
\varlimsup_{M\to\infty}
\varlimsup_{n\to\infty}
P(n\Dea^2(\hat\tau_n^\alpha-\tau_*^\alpha)>M)
\le
11\epsilon.
\label{eq7.978}
\end{align}


Since, for $\tau<\tau_*^\alpha$, 
\begin{align*}
&\Phi_n(\tau:\alpha_1,\alpha_2)-\Phi_n(\tau_*^\alpha:\alpha_1,\alpha_2)\\
&=\sum_{i=[n\tau]+1}^{[n\tau_*^\alpha]}
\Bigl(
F_i(\alpha_2)-F_i(\alpha_1)-\EE_{\alpha_1^*}[F_i(\alpha_2)-F_i(\alpha_1)|\GG]
\Bigr)\\
&\quad+\sum_{i=[n\tau]+1}^{[n\tau_*^\alpha]}
\Bigl(
\tr\left(A_{i-1}^{-1}(\alpha_2)A_{i-1}(\alpha_1)-I_d\right)
-\log\det A_{i-1}^{-1}(\alpha_2)A_{i-1}(\alpha_1)
\Bigr)\\
&\quad
-\sum_{i=[n\tau]+1}^{[n\tau_*^\alpha]}
\tr\Bigl(
\bigl(A_{i-1}^{-1}(\alpha_2)-A_{i-1}^{-1}(\alpha_1)\bigl)
\left(A_{i-1}(\alpha_1)-h^{-1}\EE_{\alpha_1^*}[(\DeX)^{\otimes2}|\GG]\right)
\Bigr),
\end{align*}
we obtain, in the same way as above,
\begin{align}
\varlimsup_{M\to\infty}
\varlimsup_{n\to\infty}
P(n\Dea^2(\tau_*^\alpha-\hat\tau_n^\alpha)>M)
\le
11\epsilon.
\label{eq7.976}
\end{align}
We see, from \eqref{eq7.979}, \eqref{eq7.978} and \eqref{eq7.976},
$\displaystyle\varlimsup_{M\to\infty}
\varlimsup_{n\to\infty}
P(n\Dea^2|\hat\tau_n^\alpha-\tau_*^\alpha|>M)
\le
22\epsilon$,
which shows 
\begin{align}\label{eq7.97}
n\Dea^2(\hat\tau_n^\alpha-\tau_*^\alpha)=O_p(1).
\end{align}
From Lemmas \ref{lem1}-\ref{lem3} and \eqref{eq7.97},  
we obtain
\begin{align*}
n\Dea^2(\hat\tau_n^\alpha-\tau_*^\alpha)
\dto
\underset{v\in\mathbb R}{\mathrm{argmin}}\,\mathbb F(v).
\end{align*}
This completes the proof of Theorem \ref{th1}.
\end{proof}

\begin{proof}[\bf{Proof of Theorem \ref{th2}}]
As with the proof of Theorem \ref{th1},
it is sufficient to show 
\begin{align*}
\varlimsup_{M\to\infty}
\varlimsup_{n\to\infty}
P(n(\hat\tau_n^\alpha-\tau_*^\alpha)>M)
=0.
\end{align*} 

Let $D_{n,M}^\alpha=\{\tau\in[0,1]|n(\tau-\tau_*^\alpha)>M\}$.
For all $\delta>0$, 
we have
\begin{align*}
P\left(n(\hat\tau_n-\tau_*^\alpha)>M\right)
&\le
P\left(
\sup_{\tau\in D_{n,M}^\alpha}
\frac{
|{\mathcal M}_n^\alpha(\tau:\hat\alpha_1,\hat\alpha_2)|
}{[n\tau]-[n\tau_*^\alpha]}
\ge\delta
\right) 
+P\left(
\inf_{\tau\in D_{n,M}^\alpha}
\frac{
{\mathcal A}_n^\alpha(\tau:\hat\alpha_1,\hat\alpha_2)
}{[n\tau]-[n\tau_*^\alpha]}
\le2\delta
\right)\\
&\quad+P\left(
\sup_{\tau\in D_{n,M}^\alpha}
\frac{
|{\varrho}_n^\alpha(\tau:\hat\alpha_1,\hat\alpha_2)|
}{[n\tau]-[n\tau_*^\alpha]}
\ge\delta
\right)\\
&=:
P_{1,n}^\alpha+P_{2,n}^\alpha+P_{3,n}^\alpha.
\end{align*}

[i] Evaluation of $P_{1,n}^\alpha$.
For large $n$, we have
\begin{align*}
P_{1,n}^\alpha
&=
P\left(
\sup_{\tau\in D_{n,M}^\alpha}
\frac{
|{\mathcal M}_n^\alpha(\tau:\hat\alpha_1,\hat\alpha_2)|
}{[n\tau]-[n\tau_*^\alpha]}
\ge\delta
\right)\\
&\le
P\left(
\sup_{\tau\in D_{n,M}^\alpha}
\frac{
|{\mathcal M}_n^\alpha(\tau:\hat\alpha_1,\hat\alpha_2)|
}{[n\tau]-[n\tau_*^\alpha]}
\ge\delta,\ 
\hat\alpha_1\in\mathcal O_{\alpha_1^*},\ 
\hat\alpha_2\in\mathcal O_{\alpha_2^*}
\right)
+P(\hat\alpha_1\notin\mathcal O_{\alpha_1^*})
+P(\hat\alpha_2\notin\mathcal O_{\alpha_2^*})\\
&\le
P\left(
\sup_{\tau\in D_{n,M}^\alpha}
\frac{
\sup_{\alpha_k\in\mathcal O_{\alpha_k^*}}|{\mathcal M}_n^\alpha(\tau:\alpha_1,\alpha_2)|
}{[n\tau]-[n\tau_*^\alpha]}
\ge\delta
\right)
+P(\hat\alpha_1\notin\mathcal O_{\alpha_1^*})
+P(\hat\alpha_2\notin\mathcal O_{\alpha_2^*}).
\end{align*}
By the uniform version on the H\'ajek-Renyi inequality in Lemma 2 
of Iacus and Yoshida (2012), 
we obtain
\begin{align*}
&P\left(
\sup_{\tau\in D_{n,M}^\alpha}
\frac{1}{[n\tau]-[n\tau_*^\alpha]}
\sup_{\alpha_k\in\mathcal O_{\alpha_k^*}}|{\mathcal M}_n^\alpha(\tau:\alpha_1,\alpha_2)|
\ge\delta
\right)\\
&\le
P\left(
\max_{j>M-1}
\frac{1}{j}
\sup_{\alpha_k\in\mathcal O_{\alpha_k^*}}\left|
\sum_{i=[n\tau_*^\alpha]+1}^{[n\tau_*^\alpha]+j}
\left(
F_i(\alpha_1)-F_i(\alpha_2)-
\EE[F_i(\alpha_1)-F_i(\alpha_2)|\GG]
\right)
\right|
\ge\delta
\right)\\
&\le
\sum_{j>M-1}
\frac{C}{(\delta j)^2}\\
&\le
\frac{C'}{\delta^2M}
=:\gamma_\alpha(M).
\end{align*}
From \textbf{[C6-I]}, 
we have  
$P(\hat\alpha_k\notin\mathcal O_{\alpha_k^*})<\epsilon$ for large $n$.
Therefore $P_{1,n}^\alpha\le\gamma_\alpha(M)+2\epsilon$ for large $n$.

[ii] Evaluation of $P_{2,n}^\alpha$.
If $\hat\alpha_k\in\mathcal O_{\alpha_k^*}$, 
then there exists a positive constant $c$ independent of $i$ such that
\begin{align*}
\Gamma_{i-1}^\alpha(\hat\alpha_1,\hat\alpha_2)
&=
\Gamma_{i-1}^\alpha(\alpha_1^*,\alpha_2^*)
+\int_0^1
\left.
\partial_{(\alpha_1,\alpha_2)}
\Gamma_{i-1}^\alpha(\alpha_1,\alpha_2)
\right|_{\alpha_k=\alpha_k^*+u(\hat\alpha_k-\alpha_k^*)}
\dd u
\left(
\begin{array}{c}
\hat\alpha_1-\alpha_1^*\\
\hat\alpha_2-\alpha_2^*
\end{array}
\right)\\
&\ge
\Gamma_{i-1}^\alpha(\alpha_1^*,\alpha_2^*)
-c(|\hat\alpha_1-\alpha_1^*|+|\hat\alpha_2-\alpha_2^*|).
\end{align*}
According to \textbf{[B1-I]},
if we set
\begin{align*}
\delta=\frac{1}{4}\inf_{x}\Gamma^\alpha(x,\alpha_1^*,\alpha_2^*)>0,
\end{align*}
then for large $n$, 
\begin{align*}
P_{2,n}^\alpha
&\le
P\left(
\inf_{\tau\in D_{n,M}^\alpha}
\frac{
{\mathcal A}_n^\alpha(\tau:\hat\alpha_1,\hat\alpha_2)
}{[n\tau]-[n\tau_*^\alpha]}
\le2\delta,\ 
\hat\alpha_1\in\mathcal O_{\alpha_1^*},\ 
\hat\alpha_2\in\mathcal O_{\alpha_2^*}
\right)
+P\left(
\hat\alpha_1\notin\mathcal O_{\alpha_1^*}
\right)
+P\left(
\hat\alpha_2\notin\mathcal O_{\alpha_2^*}
\right)\\
&\le
P\left(
\inf_{\tau\in D_{n,M}^\alpha}
\frac{1}{[n\tau]-[n\tau_*^\alpha]}
\sum_{i=[n\tau_*^\alpha]+1}^{[n\tau]}
\Bigl(
\Gamma_{i-1}^\alpha(\alpha_1^*,\alpha_2^*)
-c(|\hat\alpha_1-\alpha_1^*|+|\hat\alpha_2-\alpha_2^*|)
\Bigr)
\le2\delta 
\right)\\
&\qquad
+P\left(
\hat\alpha_1\notin\mathcal O_{\alpha_1^*}
\right)
+P\left(
\hat\alpha_2\notin\mathcal O_{\alpha_2^*}
\right)\\
&\le
P\left(
\inf_{\tau\in D_{n,M}^\alpha}
\frac{1}{[n\tau]-[n\tau_*^\alpha]}
\sum_{i=[n\tau_*^\alpha]+1}^{[n\tau]}
\Gamma_{i-1}^\alpha(\alpha_1^*,\alpha_2^*)
-c(|\hat\alpha_1-\alpha_1^*|+|\hat\alpha_2-\alpha_2^*|)
\le2\delta
\right)\\
&\qquad
+P\left(
\hat\alpha_1\notin\mathcal O_{\alpha_1^*}
\right)
+P\left(
\hat\alpha_2\notin\mathcal O_{\alpha_2^*}
\right)\\
&\le
P\left(
\inf_{\tau\in D_{n,M}^\alpha}
\frac{1}{[n\tau]-[n\tau_*^\alpha]}
\sum_{i=[n\tau_*^\alpha]+1}^{[n\tau]}
\Gamma_{i-1}^\alpha(\alpha_1^*,\alpha_2^*)
\le3\delta
\right)
+
P\Bigl(
-c(|\hat\alpha_1-\alpha_1^*|+|\hat\alpha_2-\alpha_2^*|)
\le-\delta
\Bigr)
\\
&\qquad
+P\left(
\hat\alpha_1\notin\mathcal O_{\alpha_1^*}
\right)
+P\left(
\hat\alpha_2\notin\mathcal O_{\alpha_2^*}
\right)\\
&\le
P\left(
\inf_{x}\Gamma^\alpha(x,\alpha_1^*,\alpha_2^*)
\le3\delta
\right)
+
P\left(
|\hat\alpha_1-\alpha_1^*|+|\hat\alpha_2-\alpha_2^*|
\ge\frac{\delta}{c}
\right)
+P\left(
\hat\alpha_1\notin\mathcal O_{\alpha_1^*}
\right)
+P\left(
\hat\alpha_2\notin\mathcal O_{\alpha_2^*}
\right)\\
&\le 4\epsilon
\end{align*}
thanks to
\begin{align*}
P\left(
|\hat\alpha_1-\alpha_1^*|+|\hat\alpha_2-\alpha_2^*|
\ge\frac{\delta}{c}
\right)
\le 
P\left(
|\hat\alpha_1-\alpha_1^*|
\ge\frac{\delta}{2c}
\right)
+
P\left(
|\hat\alpha_2-\alpha_2^*|
\ge\frac{\delta}{2c}
\right)
\le 2\epsilon
\end{align*}
from \textbf{[C6-I]}.

[iii] Evaluation of $P_{3,n}^\alpha$.
\begin{align*}
&\tr\Bigl(
\left(
A_{i-1}^{-1}(\hat\alpha_1)-A_{i-1}^{-1}(\hat\alpha_2)
\right)
\left(
A_{i-1}(\hat\alpha_2)-h^{-1}\EE_{\alpha_2^*}[(\DeX)^{\otimes2}|\GG]
\right)
\Bigr)\\
&=
\tr\Bigl(
\left(
A_{i-1}^{-1}(\hat\alpha_1)-A_{i-1}^{-1}(\hat\alpha_2)
\right)
\left(
A_{i-1}^{-1}(\hat\alpha_2)-A_{i-1}^{-1}(\alpha_2^*)
-hQ_{i-1}(\theta^*)+\Rd
\right)
\Bigr)\\
&\le
\int_0^1
\left.
\left[\tr\Bigl(
(A_{i-1}^{-1}(\hat\alpha_1)-A_{i-1}^{-1}(\hat\alpha_2))
\partial_{\alpha^\ell}A_{i-1}(\alpha)
\Bigr)
\right]_{\ell}
\right|_{\alpha=\alpha_2^*+u(\hat\alpha_2-\alpha_2^*)}
\dd u
(\hat\alpha_2-\alpha_2^*)\\
&\qquad+h|A_{i-1}^{-1}(\hat\alpha_1)-A_{i-1}^{-1}(\hat\alpha_2)|
|Q_{i-1}(\theta^*)|+\Rd
\end{align*}
and
\begin{align*}
&\sup_{\tau\in D_{n,M}^\alpha}
\frac{|\varrho_n^\alpha(\tau:\hat\alpha_1,\hat\alpha_2)|}
{[n\tau]-[n\tau_*^\alpha]}\\
&\le
\frac{1}{\sqrt{n}}
\sup_{\tau\in D_{n,M}^\alpha}
\left|
\frac{1}{[n\tau]-[n\tau_*^\alpha]}
\sum_{i=[n\tau_*^\alpha]+1}^{[n\tau]}
\int_0^1
\left.
\left[
\tr\Bigl(
(A_{i-1}^{-1}(\hat\alpha_1)-A_{i-1}^{-1}(\hat\alpha_2))
\partial_{\alpha^\ell}A_{i-1}(\alpha)
\Bigr)
\right]_{\ell}
\right|_{\alpha=\alpha_2^*+u(\hat\alpha_2-\alpha_2^*)}
\dd u
\right|
\\
&\qquad\times
\sqrt n|\hat\alpha_2-\alpha_2^*|
\\
&\qquad+
\sup_{\tau\in D_{n,M}^\alpha}
\frac{h}{[n\tau]-[n\tau_*^\alpha]}
\sum_{i=[n\tau_*^\alpha]+1}^{[n\tau]}
|A_{i-1}^{-1}(\hat\alpha_1)-A_{i-1}^{-1}(\hat\alpha_2)|
|Q_{i-1}(\theta^*)|
\\
&\qquad+
\sup_{\tau\in D_{n,M}^\alpha}
\frac{h^2}{[n\tau]-[n\tau_*^\alpha]}
\sum_{i=[n\tau_*^\alpha]+1}^{[n\tau]}
\Ro
\\
&\le
\frac{1}{\sqrt{n}}
\sup_{x,\alpha_k}
\left|
\left[
\tr\Bigl(
(A^{-1}(x,\alpha_1)-A^{-1}(x,\alpha_2))
\partial_{\alpha^\ell}A(x,\alpha_3)
\Bigr)
\right]_{\ell}
\right|
\sqrt n|\hat\alpha_2-\alpha_2^*|
\\
&\qquad+
h\sup_{x,\alpha_k}
\bigl|A^{-1}(x,\alpha_1)-A^{-1}(x,\alpha_2)\bigr|
\sup_{x,\theta}|Q(x,\theta)|
+
\frac{h^2}{M}
\sum_{i=[n\tau_*^\alpha]+1}^{n}
\Ro
\\
&=o_p(1).
\end{align*}
Hence, we see
\begin{align*}
P_{3,n}^\alpha
&\le
P\left(
\sup_{\tau\in D_{n,M}^\alpha}
\frac{|\varrho_n^\alpha(\tau:\hat\alpha_1,\hat\alpha_2)|}
{[n\tau]-[n\tau_*^\alpha]}
\ge\delta,\ 
\hat\alpha_1\in\mathcal O_{\alpha_1^*},\ 
\hat\alpha_2\in\mathcal O_{\alpha_2^*}
\right)
+P(\hat\alpha_1\notin\mathcal O_{\alpha_1^*})
+P(\hat\alpha_2\notin\mathcal O_{\alpha_2^*})\\
&\le 3\epsilon
\end{align*}
for large $n$.

[iv] From the evaluations in Steps [i]-[iii], we have
\begin{align*}
\varlimsup_{n\to\infty}
P(n(\hat\tau_n^\alpha-\tau_*^\alpha)>M)
\le
\gamma_\alpha(M)+9\epsilon
\end{align*}
for any $M\ge1$ and $\epsilon>0$. 
Therefore
\begin{align*}
\varlimsup_{M\to\infty}
\varlimsup_{n\to\infty}
P(n(\hat\tau_n^\alpha-\tau_*^\alpha)>M)
\le
9\epsilon.
\end{align*}
\end{proof}

\begin{proof}[\bf{Proof of Corollary \ref{cor1}}]
It is sufficient to show, for all $\epsilon_1\in[0,1)$ and $M>0$, 
\begin{align*}
\varlimsup_{n\to\infty}
P(n^{\epsilon_1}(\hat\tau_n^\alpha-\tau_*^\alpha)>M)
=0.
\end{align*} 

Let $D_{n,M}^\alpha=\{\tau\in[0,1]|n^{\epsilon_1}(\tau-\tau_*^\alpha)>M\}$.
Suppose that there exists $\delta_1\in(0,\frac{1-\epsilon_1}{2})$ such that 
$nh^{1/(\epsilon_1+\delta_1)}\lto0$.
We have
\begin{align*}
P\left(n^{\epsilon_1}(\hat\tau_n-\tau_*^\alpha)>M\right)
&\le
P\left(
\sup_{\tau\in D_{n,M}^\alpha}
\frac{
|{\mathcal M}_n^\alpha(\tau:\hat\alpha_1,\hat\alpha_2)|
}{n^{-\delta_1}([n\tau]-[n\tau_*^\alpha])}
\ge1
\right) 
+P\left(
\inf_{\tau\in D_{n,M}^\alpha}
\frac{
{\mathcal A}_n^\alpha(\tau:\hat\alpha_1,\hat\alpha_2)
}{n^{-\delta_1}([n\tau]-[n\tau_*^\alpha])}
\le2
\right)\\
&\quad+P\left(
\sup_{\tau\in D_{n,M}^\alpha}
\frac{
|{\varrho}_n^\alpha(\tau:\hat\alpha_1,\hat\alpha_2)|
}{n^{-\delta_1}([n\tau]-[n\tau_*^\alpha])}
\ge1
\right)\\
&=:
P_{1,n}^\alpha+P_{2,n}^\alpha+P_{3,n}^\alpha.
\end{align*}

[i] Evaluation of $P_{1,n}^\alpha$.
For large $n$, we have
\begin{align*}
P_{1,n}^\alpha
&=
P\left(
\sup_{\tau\in D_{n,M}^\alpha}
\frac{
|{\mathcal M}_n^\alpha(\tau:\hat\alpha_1,\hat\alpha_2)|
}{n^{-\delta_1}([n\tau]-[n\tau_*^\alpha])}
\ge1
\right)\\
&\le
P\left(
\sup_{\tau\in D_{n,M}^\alpha}
\frac{
|{\mathcal M}_n^\alpha(\tau:\hat\alpha_1,\hat\alpha_2)|
}{n^{-\delta_1}([n\tau]-[n\tau_*^\alpha])}
\ge1,\ 
\hat\alpha_1\in\mathcal O_{\alpha_1^*},\ 
\hat\alpha_2\in\mathcal O_{\alpha_2^*}
\right)
+P(\hat\alpha_1\notin\mathcal O_{\alpha_1^*})
+P(\hat\alpha_2\notin\mathcal O_{\alpha_2^*})\\
&\le
P\left(
\sup_{\tau\in D_{n,M}^\alpha}
\frac{
\sup_{\alpha_k\in\mathcal O_{\alpha_k^*}}|{\mathcal M}_n^\alpha(\tau:\alpha_1,\alpha_2)|
}{[n\tau]-[n\tau_*^\alpha]}
\ge n^{-\delta_1}
\right)
+P(\hat\alpha_1\notin\mathcal O_{\alpha_1^*})
+P(\hat\alpha_2\notin\mathcal O_{\alpha_2^*}).
\end{align*}
By the uniform version on the H\'ajek-Renyi inequality in Lemma 2 
of Iacus and Yoshida (2012), 
we obtain
\begin{align*}
&P\left(
\sup_{\tau\in D_{n,M}^\alpha}
\frac{1}{[n\tau]-[n\tau_*^\alpha]}
\sup_{\alpha_k\in\mathcal O_{\alpha_k^*}}|{\mathcal M}_n^\alpha(\tau:\alpha_1,\alpha_2)|
\ge n^{-\delta_1}
\right)\\
&\le
P\left(
\max_{j>M/n^{\epsilon_1-1}-1}
\frac{1}{j}
\sup_{\alpha_k\in\mathcal O_{\alpha_k^*}}\left|
\sum_{i=[n\tau_*^\alpha]+1}^{[n\tau_*^\alpha]+j}
\left(
F_i(\alpha_1)-F_i(\alpha_2)-
\EE[F_i(\alpha_1)-F_i(\alpha_2)|\GG]
\right)
\right|
\ge n^{-\delta_1}
\right)\\
&\le
\sum_{j>M/n^{\epsilon_1-1}-1}
\frac{C}{( n^{-\delta_1} j)^2}\\
&\le
\frac{C'}{n^{-2\delta_1}}
\frac{n^{\epsilon_1-1}}{M}
=\frac{C'}{M}
n^{\epsilon_1+2\delta_1-1}
\lto0.
\end{align*}
From \textbf{[C6-I]}, 
we have  
$P(\hat\alpha_k\notin\mathcal O_{\alpha_k^*})<\epsilon$ for large $n$.
Therefore $P_{1,n}^\alpha\le3\epsilon$ for large $n$.

[ii] Evaluation of $P_{2,n}^\alpha$.
If $\hat\alpha_k\in\mathcal O_{\alpha_k^*}$, 
then there exists a positive constant $c$ independent of $i$ such that
\begin{align*}
\Gamma_{i-1}^\alpha(\hat\alpha_1,\hat\alpha_2)
&=
\Gamma_{i-1}^\alpha(\alpha_1^*,\alpha_2^*)
+\int_0^1
\left.
\partial_{(\alpha_1,\alpha_2)}
\Gamma_{i-1}^\alpha(\alpha_1,\alpha_2)
\right|_{\alpha_k=\alpha_k^*+u(\hat\alpha_k-\alpha_k^*)}
\dd u
\left(
\begin{array}{c}
\hat\alpha_1-\alpha_1^*\\
\hat\alpha_2-\alpha_2^*
\end{array}
\right)\\
&\ge
\Gamma_{i-1}^\alpha(\alpha_1^*,\alpha_2^*)
-c(|\hat\alpha_1-\alpha_1^*|+|\hat\alpha_2-\alpha_2^*|).
\end{align*}
We see, for large $n$, 
\begin{align*}
P_{2,n}^\alpha
&\le
P\left(
\inf_{\tau\in D_{n,M}^\alpha}
\frac{
{\mathcal A}_n^\alpha(\tau:\hat\alpha_1,\hat\alpha_2)
}{n^{-\delta_1}([n\tau]-[n\tau_*^\alpha])}
\le2,\ 
\hat\alpha_1\in\mathcal O_{\alpha_1^*},\ 
\hat\alpha_2\in\mathcal O_{\alpha_2^*}
\right)
+P\left(
\hat\alpha_1\notin\mathcal O_{\alpha_1^*}
\right)
+P\left(
\hat\alpha_2\notin\mathcal O_{\alpha_2^*}
\right)\\
&\le
P\left(
\inf_{\tau\in D_{n,M}^\alpha}
\frac{1}{[n\tau]-[n\tau_*^\alpha]}
\sum_{i=[n\tau_*^\alpha]+1}^{[n\tau]}
\Bigl(
\Gamma_{i-1}^\alpha(\alpha_1^*,\alpha_2^*)
-c(|\hat\alpha_1-\alpha_1^*|+|\hat\alpha_2-\alpha_2^*|)
\Bigr)
\le2 n^{-\delta_1} 
\right)\\
&\qquad
+P\left(
\hat\alpha_1\notin\mathcal O_{\alpha_1^*}
\right)
+P\left(
\hat\alpha_2\notin\mathcal O_{\alpha_2^*}
\right)\\
&\le
P\left(
\inf_{\tau\in D_{n,M}^\alpha}
\frac{1}{[n\tau]-[n\tau_*^\alpha]}
\sum_{i=[n\tau_*^\alpha]+1}^{[n\tau]}
\Gamma_{i-1}^\alpha(\alpha_1^*,\alpha_2^*)
-c(|\hat\alpha_1-\alpha_1^*|+|\hat\alpha_2-\alpha_2^*|)
\le2 n^{-\delta_1}
\right)\\
&\qquad
+P\left(
\hat\alpha_1\notin\mathcal O_{\alpha_1^*}
\right)
+P\left(
\hat\alpha_2\notin\mathcal O_{\alpha_2^*}
\right)\\
&\le
P\left(
\inf_{\tau\in D_{n,M}^\alpha}
\frac{1}{[n\tau]-[n\tau_*^\alpha]}
\sum_{i=[n\tau_*^\alpha]+1}^{[n\tau]}
\Gamma_{i-1}^\alpha(\alpha_1^*,\alpha_2^*)
\le3 n^{-\delta_1}
\right)
+
P\Bigl(
-c(|\hat\alpha_1-\alpha_1^*|+|\hat\alpha_2-\alpha_2^*|)
\le- n^{-\delta_1}
\Bigr)
\\
&\qquad
+P\left(
\hat\alpha_1\notin\mathcal O_{\alpha_1^*}
\right)
+P\left(
\hat\alpha_2\notin\mathcal O_{\alpha_2^*}
\right)\\
&\le
P\left(
\inf_{x}\Gamma^\alpha(x,\alpha_1^*,\alpha_2^*)
\le3 {n^{-\delta_1}}
\right)
+
P\left(
|\hat\alpha_1-\alpha_1^*|+|\hat\alpha_2-\alpha_2^*|
\ge\frac{n^{-\delta_1}}{c}
\right)
+P\left(
\hat\alpha_1\notin\mathcal O_{\alpha_1^*}
\right)
+P\left(
\hat\alpha_2\notin\mathcal O_{\alpha_2^*}
\right)\\
&\le 5\epsilon
\end{align*}
thanks to
\begin{align*}
P\left(
|\hat\alpha_1-\alpha_1^*|+|\hat\alpha_2-\alpha_2^*|
\ge\frac{n^{-\delta_1}}{c}
\right)
\le 
P\left(
n^{\delta_1-1/2}\sqrt n|\hat\alpha_1-\alpha_1^*|
\ge\frac{1}{2c}
\right)
+
P\left(
n^{\delta_1-1/2}\sqrt n|\hat\alpha_2-\alpha_2^*|
\ge\frac{1}{2c}
\right)
\le 2\epsilon
\end{align*}
from \textbf{[C6-I]}.

[iii] Evaluation of $P_{3,n}^\alpha$.
\begin{align*}
&\tr\Bigl(
\left(
A_{i-1}^{-1}(\hat\alpha_1)-A_{i-1}^{-1}(\hat\alpha_2)
\right)
\left(
A_{i-1}(\hat\alpha_2)-h^{-1}\EE_{\alpha_2^*}[(\DeX)^{\otimes2}|\GG]
\right)
\Bigr)\\
&=
\tr\Bigl(
\left(
A_{i-1}^{-1}(\hat\alpha_1)-A_{i-1}^{-1}(\hat\alpha_2)
\right)
\left(
A_{i-1}(\hat\alpha_2)-A_{i-1}(\alpha_2^*)
\right)
\Bigr)
+\Ri
\\
&
\le
\int_0^1
\left.
\left[\tr\Bigl(
(A_{i-1}^{-1}(\hat\alpha_1)-A_{i-1}^{-1}(\hat\alpha_2))
\partial_{\alpha^\ell}A_{i-1}(\alpha)
\Bigr)
\right]_{\ell}
\right|_{\alpha=\alpha_2^*+u(\hat\alpha_2-\alpha_2^*)}
\dd u
(\hat\alpha_2-\alpha_2^*)
+\Ri
\end{align*}
and
\begin{align*}
&\sup_{\tau\in D_{n,M}^\alpha}
\frac{|\varrho_n^\alpha(\tau:\hat\alpha_1,\hat\alpha_2)|}
{n^{-\delta_1}([n\tau]-[n\tau_*^\alpha])}\\
&\le
\frac{1}{n^{-\delta_1+1/2}}
\sup_{\tau\in D_{n,M}^\alpha}
\left|
\frac{1}{[n\tau]-[n\tau_*^\alpha]}
\sum_{i=[n\tau_*^\alpha]+1}^{[n\tau]}
\int_0^1
\left.
\left[
\tr\Bigl(
(A_{i-1}^{-1}(\hat\alpha_1)-A_{i-1}^{-1}(\hat\alpha_2))
\partial_{\alpha^\ell}A_{i-1}(\alpha)
\Bigr)
\right]_{\ell}
\right|_{\alpha=\alpha_2^*+u(\hat\alpha_2-\alpha_2^*)}
\dd u
\right|
\\
&\qquad\times
\sqrt n|\hat\alpha_2-\alpha_2^*|
\\
&
\qquad+
\frac{h}{n^{-\delta_1}}
\sup_{\tau\in D_{n,M}^\alpha}
\frac{1}{[n\tau]-[n\tau_*^\alpha]}
\sum_{i=[n\tau_*^\alpha]+1}^{[n\tau]}
\Ro
\\
&\le
\frac{1}{n^{-\delta_1+1/2}}
\sup_{x,\alpha_k}
\left|
\left[
\tr\Bigl(
(A^{-1}(x,\alpha_1)-A^{-1}(x,\alpha_2))
\partial_{\alpha^\ell}A(x,\alpha_3)
\Bigr)
\right]_{\ell}
\right|
\sqrt n|\hat\alpha_2-\alpha_2^*|
\\
&\qquad+
\frac{h}{n^{-\delta_1}}
\frac{n^{\epsilon_1-1}}{M}
\sum_{i=[n\tau_*^\alpha]+1}^{n}
\Ro
\\
&=
O_p(n^{\delta_1-1/2})
+O_p(n^{\epsilon_1+\delta_1}h)
=
O_p(n^{\delta_1-1/2})
+O_p(nh^{1/(\epsilon_1+\delta_1)})
=o_p(1).
\end{align*}
Hence, we see
\begin{align*}
P_{3,n}^\alpha
&\le
P\left(
\sup_{\tau\in D_{n,M}^\alpha}
\frac{|\varrho_n^\alpha(\tau:\hat\alpha_1,\hat\alpha_2)|}
{n^{-\delta_1}([n\tau]-[n\tau_*^\alpha])}
\ge1,\ 
\hat\alpha_1\in\mathcal O_{\alpha_1^*},\ 
\hat\alpha_2\in\mathcal O_{\alpha_2^*}
\right)
+P(\hat\alpha_1\notin\mathcal O_{\alpha_1^*})
+P(\hat\alpha_2\notin\mathcal O_{\alpha_2^*})\\
&\le 3\epsilon
\end{align*}
for large $n$.

[iv] From the evaluations in Steps [i]-[iii], we have
\begin{align*}
\varlimsup_{n\to\infty}
P(n^{\epsilon_1}(\hat\tau_n^\alpha-\tau_*^\alpha)>M)
\le
11\epsilon
\end{align*}
for any $M>0$ and $\epsilon>0$. 
\end{proof}

In Case A of Situation II, we set
\begin{align*}
\mathbb G_n(v)
&=\Psi_n\left(\tau_*^\beta+\frac{v}{T\Deb^2}
:\beta_1^*,\beta_2^*\middle|\alpha^*\right)
-\Psi_n\left(\tau_*^\beta:\beta_1^*,\beta_2^*\middle|\alpha^*\right),\\
\hat{\mathbb G}_n(v)
&=\Psi_n\left(\tau_*^\beta+\frac{v}{T\Deb^2}:\hat\beta_1,\hat\beta_2\middle|\hat\alpha\right)
-\Psi_n\left(\tau_*^\beta:\hat\beta_1,\hat\beta_2\middle|\hat\alpha\right),\\
\mathcal D_n^\beta(v)
&=\hat{\mathbb G}_n(v)-\mathbb G_n(v).
\end{align*}
\begin{lem}\label{lem4}
Suppose that \textbf{[C1]}-\textbf{[C5]}, \textbf{[C6-II]},
\textbf{[A1-II]}-\textbf{[A3-II]} hold.
Then, for all $L>0$, 
\begin{align*}
\sup_{v\in[-L,L]}|\mathcal D_n^\beta(v)|\pto0
\end{align*}
as $n\to\infty$.
\end{lem}
\begin{proof}
We assume that $v>0$. Then, we can express 
\begin{align}
\mathcal D_n^\beta(v)
&=\sum_{i=[n\tau_*^\beta]+1}^{[n\tau_*^\beta+v/h\Deb^2]}
\left(
G_i(\hat\beta_1|\hat\alpha)-G_i(\hat\beta_2|\hat\alpha)
\right)
-\sum_{i=[n\tau_*^\beta]+1}^{[n\tau_*^\beta+v/h\Deb^2]}
\Bigl(
G_i(\beta_1^*|\alpha^*)-G_i(\beta_2^*|\alpha^*)
\Bigr)
\nonumber
\\
&=
\sum_{i=[n\tau_*^\beta]+1}^{[n\tau_*^\beta+v/h\Deb^2]}
\biggl(
G_i(\hat\beta_1|\alpha^*)-G_i(\hat\beta_2|\alpha^*)
+
\int_0^1
\left.\left(
\partial_\alpha G_i(\hat\beta_1|\alpha)
-\partial_\alpha G_i(\hat\beta_2|\alpha)
\right)
\right|_{\alpha=\alpha^*+u(\hat\alpha-\alpha^*)}
\dd u
(\hat\alpha-\alpha^*)
\biggr)
\nonumber
\\
&\qquad-
\sum_{i=[n\tau_*^\beta]+1}^{[n\tau_*^\beta+v/h\Deb^2]}
\Bigl(
G_i(\beta_1^*|\alpha^*)-G_i(\beta_2^*|\alpha^*)
\Bigr)
\nonumber
\\
&=
\sum_{i=[n\tau_*^\beta]+1}^{[n\tau_*^\beta+v/h\Deb^2]}
\Biggl(
G_i(\beta_1^*|\alpha^*)-G_i(\beta_2^*|\alpha^*)
+
\sum_{j=1}^{m-1}
\frac{1}{j!}
\left(
\partial_\beta^j G_i(\beta_1^*|\alpha^*)
\otimes(\hat\beta_1-\beta_1^*)^{\otimes j}
-\partial_\beta^j G_i(\beta_2^*|\alpha^*)
\otimes(\hat\beta_2-\beta_2^*)^{\otimes j}
\right)
\nonumber
\\
&\qquad\qquad+
\int_0^1
\frac{(1-u)^{m-1}}{(m-1)!}
\partial_\beta^{m} G_i(\beta_1^*+u(\hat\beta_1-\beta_1^*)|\alpha^*)
\dd u
\otimes(\hat\beta_1-\beta_1^*)^{\otimes m}
\nonumber
\\
&\qquad\qquad-
\int_0^1
\frac{(1-u)^{m-1}}{(m-1)!}
\partial_\beta^{m} G_i(\beta_2^*+u(\hat\beta_2-\beta_2^*)|\alpha^*)
\dd u
\otimes(\hat\beta_2-\beta_2^*)^{\otimes m}
\nonumber
\\
&\qquad\qquad+
\int_0^1\int_0^1
\partial_\beta\partial_\alpha 
G_i(\hat\beta_2+v(\hat\beta_1-\hat\beta_2)|\alpha^*+u(\hat\alpha-\alpha^*))
\dd u\dd v
\otimes(\hat\alpha-\alpha^*)
\otimes(\hat\beta_1-\hat\beta_2)
\Biggr)
\nonumber
\\
&\qquad-
\sum_{i=[n\tau_*^\beta]+1}^{[n\tau_*^\beta+v/h\Deb^2]}
\Bigl(
G_i(\beta_1^*|\alpha^*)-G_i(\beta_2^*|\alpha^*)
\Bigr)
\nonumber
\\
&=
\sum_{i=[n\tau_*^\beta]+1}^{[n\tau_*^\beta+v/h\Deb^2]}
\sum_{j=1}^{m-1}
\frac{1}{j!}
\left(
\partial_\beta^j G_i(\beta_1^*|\alpha^*)
\otimes(\hat\beta_1-\beta_1^*)^{\otimes j}
-\partial_\beta^j G_i(\beta_2^*|\alpha^*)
\otimes(\hat\beta_2-\beta_2^*)^{\otimes j}
\right)
\nonumber
\\
&\qquad+
\sum_{i=[n\tau_*^\beta]+1}^{[n\tau_*^\beta+v/h\Deb^2]}
\biggl(
\int_0^1
\frac{(1-u)^{m-1}}{(m-1)!}
\partial_\beta^{m} G_i(\beta_1^*+u(\hat\beta_1-\beta_1^*)|\alpha^*)
\dd u
\otimes(\hat\beta_1-\beta_1^*)^{\otimes m}
\nonumber
\\
&\qquad\qquad-
\int_0^1
\frac{(1-u)^{m-1}}{(m-1)!}
\partial_\beta^{m} G_i(\beta_2^*+u(\hat\beta_2-\beta_2^*)|\alpha^*)
\dd u
\otimes(\hat\beta_2-\beta_2^*)^{\otimes m}
\nonumber
\\
&\qquad\qquad+
\int_0^1\int_0^1
\partial_\beta\partial_\alpha 
G_i(\hat\beta_2+v(\hat\beta_1-\hat\beta_2)|\alpha^*+u(\hat\alpha-\alpha^*))
\dd u\dd v
\otimes(\hat\alpha-\alpha^*)
\otimes(\hat\beta_1-\hat\beta_2)
\biggr).
\label{eq7.399}
\end{align}
Now we see
\begin{align}
&\sup_{v\in[0,L]}
\left|
\sum_{i=[n\tau_*^\beta]+1}^{[n\tau_*^\beta+v/h\Deb^2]}
\int_0^1
\frac{(1-u)^{m-1}}{(m-1)!}
\partial_\beta^{m} G_i(\beta_j^*+u(\hat\beta_j-\beta_j^*)|\alpha^*)
\dd u
\otimes(\hat\beta_k-\beta_k^*)^{\otimes m}
\right|
\nonumber
\\
&\le
\sum_{i=[n\tau_*^\beta]+1}^{[n\tau_*^\beta+L/h\Deb^2]}
\left|
\int_0^1
\frac{(1-u)^{m-1}}{(m-1)!}
\partial_\beta^{m} G_i(\beta_k^*+u(\hat\beta_k-\beta_k^*)|\alpha^*)
\dd u
\right|
|\hat\beta_k-\beta_k^*|^{m}
\nonumber
\\
&\le
\frac{1}{T^{m/2}}
\sum_{i=[n\tau_*^\beta]+1}^{[n\tau_*^\beta+L/h\Deb^2]}
\sup_{\beta\in\Theta_B}\left|
\partial_\beta^{m} G_i(\beta|\alpha^*)
\right|
\left(\sqrt{T}|\hat\beta_k-\beta_k^*|\right)^{m}
\nonumber
\\
&
=O_p\left(
\frac{1}{T^{m/2}h^{1/2}\Deb^2}
\right)
=O_p\left(
\frac{1}{n^{m/2-1}h^{(m-1)/2}}
\frac{1}{T\Deb^2}
\right)
\nonumber
\\
&=o_p(1),
\label{eq7.398}
\end{align}
\begin{align}
&\sup_{v\in[0,L]}
\left|
\sum_{i=[n\tau_*^\beta]+1}^{[n\tau_*^\beta+v/h\Deb^2]}
\int_0^1\int_0^1
\partial_\beta\partial_\alpha 
G_i(\hat\beta_2+v(\hat\beta_1-\hat\beta_2)|\alpha^*+u(\hat\alpha-\alpha^*))
\dd u\dd v
\otimes(\hat\alpha-\alpha^*)
\otimes(\hat\beta_1-\hat\beta_2)
\right|
\nonumber
\\
&\le
\sum_{i=[n\tau_*^\beta]+1}^{[n\tau_*^\beta+L/h\Deb^2]}
\left|
\int_0^1\int_0^1
\partial_\beta\partial_\alpha 
G_i(\hat\beta_2+v(\hat\beta_1-\hat\beta_2)|\alpha^*+u(\hat\alpha-\alpha^*))
\dd u\dd v
\right|
|\hat\alpha-\alpha^*|
|\hat\beta_1-\hat\beta_2|
\nonumber
\\
&\le
\frac{\Deb}{\sqrt{n}}
\sum_{i=[n\tau_*^\beta]+1}^{[n\tau_*^\beta+L/h\Deb^2]}
\sup_{(\alpha,\beta)\in\Theta}
\left|
\partial_\beta\partial_\alpha G_i(\beta|\alpha)
\right|
\Deb^{-1}|\hat\beta_1-\hat\beta_2|
\sqrt{n}|\hat\alpha-\alpha^*|
\nonumber
\\
&=
O_p\left(
\frac{1}{\sqrt{T}\Deb}
\right)
=o_p(1)
\label{eq7.397}
\end{align}
and
\begin{align}
&\sum_{i=[n\tau_*^\beta]+1}^{[n\tau_*^\beta+v/h\Deb^2]}
\partial_\beta^j G_i(\beta_k^*|\alpha^*)
\otimes(\hat\beta_k-\beta_k^*)^{\otimes j}
\nonumber
\\
&=
\sum_{i=[n\tau_*^\beta]+1}^{[n\tau_*^\beta+v/h\Deb^2]}
\left(
\partial_\beta^j G_i(\beta_k^*|\alpha^*)
-\EE_{\beta_2^*}[\partial_\beta^j G_i(\beta_k^*|\alpha^*)|\GG]
\right)
\otimes(\hat\beta_k-\beta_k^*)^{\otimes j}
\nonumber
\\
&\qquad
+
\sum_{i=[n\tau_*^\beta]+1}^{[n\tau_*^\beta+v/h\Deb^2]}
\EE_{\beta_2^*}[\partial_\beta^j G_i(\beta_k^*|\alpha^*)|\GG]
\otimes(\hat\beta_k-\beta_k^*)^{\otimes j}.
\label{eq7.396}
\end{align}
By Theorem 2.11 of Hall and Heyde (1980), we have
\begin{align*}
&\EE_{\beta_2^*}\left[
\frac1{T^j}
\sup_{v\in[0,L]}
\left|
\sum_{i=[n\tau_*^\beta]+1}^{[n\tau_*^\beta+v/h\Deb^2]}
\left(
\partial_\beta^j G_i(\beta_k^*|\alpha^*)
-\EE_{\beta_2^*}[\partial_\beta^j G_i(\beta_k^*|\alpha^*)|\GG]
\right)\right|^2
\right]\\
&\le
\frac{C}{T^j}
\sum_{i=[n\tau_*^\beta]+1}^{[n\tau_*^\beta+L/h\Deb^2]}
\EE_{\beta_2^*}\left[
\EE_{\beta_2^*}\left[
\left|
\partial_\beta^j G_i(\beta_k^*|\alpha^*)
-\EE[\partial_\beta^j G_i(\beta_k^*|\alpha^*)|\GG]
\right|^2
\middle|\GG\right]\right]\\
&\le\frac{C'}{T}\frac{h}{h\Deb^2}
=\frac{C'}{T\Deb^2}
\lto0
\end{align*}
and
\begin{align}
\frac1{T^{j/2}}
\sup_{v\in[0,L]}
\left|
\sum_{i=[n\tau_*^\beta]+1}^{[n\tau_*^\beta+v/h\Deb^2]}
\left(
\partial_\beta^j G_i(\beta_k^*|\alpha^*)
-\EE_{\beta_2^*}[\partial_\beta^j G_i(\beta_k^*|\alpha^*)|\GG]
\right)
\right|=o_p(1).
\label{eq7.395}
\end{align}
Moreover, we see
\begin{align}
&\sup_{v\in[0,L]}
\left|
\sum_{i=[n\tau_*^\beta]+1}^{[n\tau_*^\beta+v/h\Deb^2]}
\EE_{\beta_2^*}[\partial_\beta G_i(\beta_k^*|\alpha^*)|\GG]
\right|
|\hat\beta_k-\beta_k^*|
\nonumber
\\
&=
\sup_{v\in[0,L]}
\left|
\sum_{i=[n\tau_*^\beta]+1}^{[n\tau_*^\beta+v/h\Deb^2]}
\left(
-2h
\Bigl[
\partial_{\beta^{\ell_1}}b_{i-1}(\beta_k^*)^\TT
A_{i-1}^{-1}(\alpha^*)
(b_{i-1}(\beta_2^*)-b_{i-1}(\beta_k^*))
\Bigr]_{\ell_1}
+\Rd
\right)
\right|
|\hat\beta_k-\beta_k^*|
\nonumber
\\
&=
\sup_{v\in[0,L]}
\left|
\sum_{i=[n\tau_*^\beta]+1}^{[n\tau_*^\beta+v/h\Deb^2]}
\biggl(
-2h
\Bigl(
\Xi_{i-1}^\beta(\alpha^*,\beta_k^*)(\beta_2^*-\beta_k^*)
\right.
\nonumber
\\
&\qquad
+\int_0^1(1-u)
\Bigl[\partial_{\beta^{\ell_1}} b_{i-1}(\beta_k^*)^\TT A_{i-1}^{-1}(\alpha^*)
\partial_{\beta^{\ell_3}}\partial_{\beta^{\ell_2}} 
b_{i-1}(\beta_k^*+u(\beta_2^*-\beta_k^*))
\Bigr]_{\ell_1,\ell_2,\ell_3}
\dd u
\otimes(\beta_2^*-\beta_k^*)^{\otimes2}
\Bigr)
\nonumber
\\
&\qquad
+\Rd
\biggr)
\Biggr|
|\hat\beta_k-\beta_k^*|
\nonumber
\\
&=
\sup_{v\in[0,L]}
\left|
\sum_{i=[n\tau_*^\beta]+1}^{[n\tau_*^\beta+v/h\Deb^2]}
\biggl(
-2h
\Bigl(
\Xi_{i-1}^\beta(\alpha^*,\beta_0)(\beta_2^*-\beta_k^*)
\right.
\nonumber
\\
&\qquad+
\int_0^1
\partial_{\beta}\Xi_{i-1}^\beta(\alpha^*,\beta_0+u(\beta_k^*-\beta_0))
\dd u
\otimes(\beta_2^*-\beta_k^*)
\otimes(\beta_k^*-\beta_0)
\nonumber
\\
&\qquad
+\int_0^1(1-u)
\Bigl[\partial_{\beta^{\ell_1}} b_{i-1}(\beta_k^*)^\TT A_{i-1}^{-1}(\alpha^*)
\partial_{\beta^{\ell_3}}\partial_{\beta^{\ell_2}} 
b_{i-1}(\beta_k^*+u(\beta_2^*-\beta_k^*))
\Bigr]_{\ell_1,\ell_2,\ell_3}
\dd u
\otimes(\beta_2^*-\beta_k^*)^{\otimes2}
\Bigr)
\nonumber
\\
&\qquad
+\Rd
\biggr)
\Biggr|
|\hat\beta_k-\beta_k^*|
\nonumber
\\
&\le
\frac{2h\Deb}{\sqrt{T}}
\sup_{v\in[0,L]}
\left|
\sum_{i=[n\tau_*^\beta]+1}^{[n\tau_*^\beta+v/h\Deb^2]}
\Xi_{i-1}^\beta(\alpha^*,\beta_0)
\right|
\Deb^{-1}|\beta_2^*-\beta_k^*|
\sqrt{T}|\hat\beta_k-\beta_k^*|
\nonumber
\\
&\qquad+
\frac{h\Deb^2}{\sqrt T}
\sum_{i=[n\tau_*^\beta]+1}^{[n\tau_*^\beta+L/h\Deb^2]}
\Ro
\Deb^{-1}|\beta_2^*-\beta_k^*|
\Deb^{-1}|\beta_k^*-\beta_0|
\sqrt{T}|\hat\beta_k-\beta_k^*|
\nonumber
\\
&\qquad+
\frac{h\Deb^2}{\sqrt T}
\sum_{i=[n\tau_*^\beta]+1}^{[n\tau_*^\beta+L/h\Deb^2]}
\Ro
(\Deb^{-1}|\beta_2^*-\beta_k^*|)^2
\sqrt{T}|\hat\beta_k-\beta_k^*|
+\frac{1}{\sqrt T}
\sum_{i=[n\tau_*^\beta]+1}^{[n\tau_*^\beta+L/h\Deb^2]}
\Rd
\sqrt{T}|\hat\beta_k-\beta_k^*|
\nonumber
\\
&=O_p\left(\frac{1}{\sqrt{T}\Deb}\right)
+O_p\left(\frac{1}{\sqrt T}\right)
+O_p\left(\frac{h}{\sqrt{T}\Deb^2}\right)
\nonumber
\\
&=o_p(1)
\label{eq7.394}
\end{align}
and, for $j\ge 2$,
\begin{align}
\sup_{v\in[0,L]}
\left|
\sum_{i=[n\tau_*^\beta]+1}^{[n\tau_*^\beta+v/h\Deb^2]}
\EE_{\beta_2^*}[\partial_\beta^j G_i(\beta_k^*|\alpha^*)|\GG]
\right|
|\hat\beta_k-\beta_k^*|^j
&\le
\sum_{i=[n\tau_*^\beta]+1}^{[n\tau_*^\beta+L/h\Deb^2]}
\Ri
|\hat\beta_j-\beta_j^*|^j
\nonumber
\\
&=
\frac{h}{T^{j/2}}\sum_{i=[n\tau_*^\beta]+1}^{[n\tau_*^\beta+L/h\Deb^2]}
\Ro
\left(\sqrt{T}|\hat\beta_k-\beta_k^*|\right)^j
\nonumber
\\
&=O_p\left(\frac{1}{T\Deb^2}\right)
\nonumber
\\
&=o_p(1).
\label{eq7.393}
\end{align}
Therefore, from \eqref{eq7.399}-\eqref{eq7.393}, we have
\begin{align*}
\sup_{v\in[0,L]}
|{\mathcal D}_n^\beta(v)|\pto0.
\end{align*}
By the similar proof, we see
$\sup_{v\in[-L,0]}
|{\mathcal D}_n^\beta(v)|\pto0$, 
and this proof is complete.
\end{proof}

\begin{lem}\label{lem5}
Suppose that \textbf{[C1]}-\textbf{[C5]}, \textbf{[C6-II]},
\textbf{[A1-II]}-\textbf{[A3-II]} hold.
Then, for all $L>0$,
\begin{align*}
\mathbb G_n(v)\wto\mathbb G(v) \text{ in }\mathbb D[-L,L]
\end{align*}
as $n\to\infty$.
\end{lem}
\begin{proof}
We consider $v>0$. We have 
\begin{align*}
\mathbb G_n(v)
&=\sum_{i=[n\tau_*^\beta]+1}^{[n\tau_*^\beta+v/h\Deb^2]}
\Bigl(G_i(\beta_1^*|\alpha^*)-G_i(\beta_2^*|\alpha^*)\Bigr)\\
&=\sum_{i=[n\tau_*^\beta]+1}^{[n\tau_*^\beta+v/h\Deb^2]}
\Biggl(
\sum_{j=1}^{m-1}\frac{1}{j!}
\partial_\beta^j G_i(\beta_2^*|\alpha^*)\otimes(\beta_1^*-\beta_2^*)^{\otimes j}
\\
&\qquad
+\int_0^1\frac{(1-u)^{m-1}}{(m-1)!}\partial_\beta^m 
G_i(\beta_2^*+u(\beta_1^*-\beta_2^*)|\alpha^*)\dd u
\otimes(\beta_1^*-\beta_2^*)^{\otimes m}
\Biggr)\\
&=\sum_{i=[n\tau_*^\beta]+1}^{[n\tau_*^\beta+v/h\Deb^2]}
\left(
\partial_\beta G_i(\beta_2^*|\alpha^*)(\beta_1^*-\beta_2^*)
+\frac12\partial_\beta^2 G_i(\beta_2^*|\alpha^*)\otimes(\beta_1^*-\beta_2^*)^{\otimes2}
\right)\\
&\qquad+
\sum_{i=[n\tau_*^\beta]+1}^{[n\tau_*^\beta+v/h\Deb^2]}
\Biggl(
\sum_{j=3}^{m-1}\frac{1}{j!}
\partial_\beta^j G_i(\beta_2^*|\alpha^*)\otimes(\beta_1^*-\beta_2^*)^{\otimes j}
\\
&\qquad\qquad
+\int_0^1\frac{(1-u)^{m-1}}{(m-1)!}\partial_\beta^{m} 
G_i(\beta_2^*+u(\beta_1^*-\beta_2^*)|\alpha^*)\dd u
\otimes(\beta_1^*-\beta_2^*)^{\otimes m}
\Biggr)\\
&=\sum_{i=[n\tau_*^\beta]+1}^{[n\tau_*^\beta+v/h\Deb^2]}
\left(
\partial_\beta G_i(\beta_2^*|\alpha^*)(\beta_1^*-\beta_2^*)
+\frac12\partial_\beta^2 G_i(\beta_2^*|\alpha^*)\otimes(\beta_1^*-\beta_2^*)^{\otimes2}
\right)+\bar o_p(1)\\
&=:\mathbb G_{1,n}(v)+\mathbb G_{2,n}(v)+\bar o_p(1),
\end{align*}
where, for $j\ge 3$,
\begin{align*}
&\Deb^j \sup_{v\in[0,L]}
\left|
\sum_{i=[n\tau_*^\beta]+1}^{[n\tau_*^\beta+v/h\Deb^2]}
\left(
\partial_\beta^j G_i(\beta_2^*|\alpha^*)
-\EE[\partial_\beta^j G_i(\beta_2^*|\alpha^*)|\GG]
\right)
\right|=o_p(1),\\
&\Deb^j\sup_{v\in[0,L]}
\left|
\sum_{i=[n\tau_*^\beta]+1}^{[n\tau_*^\beta+v/h\Deb^2]}
\EE_{\beta_2^*}[\partial_\beta^j G_i(\beta_2^*|\alpha^*)|\GG]
\right|
=o_p(1),\\
&\sup_{v\in[0,L]}
\left|
\sum_{i=[n\tau_*^\beta]+1}^{[n\tau_*^\beta+v/h\Deb^2]}
\int_0^1\frac{(1-u)^{m-1}}{(m-1)!}\partial_\beta^{m} 
G_i(\beta_2^*+u(\beta_1^*-\beta_2^*)|\alpha^*)\dd u
\otimes(\beta_1^*-\beta_2^*)^{\otimes m}
\right|=o_p(1).
\end{align*}
Now, we see
\begin{align}
\mathbb G_{1,n}(v)
&=\sum_{i=[n\tau_*^\beta]+1}^{[n\tau_*^\beta+v/h\Deb^2]}
\Bigl(
\partial_\beta G_i(\beta_2^*|\alpha^*)
-\EE_{\beta_2^*}[\partial_\beta G_i(\beta_2^*|\alpha^*)|\GG]
\Bigr)
(\beta_1^*-\beta_2^*)
\nonumber
\\
&\qquad+
\sum_{i=[n\tau_*^\beta]+1}^{[n\tau_*^\beta+v/h\Deb^2]}
\EE_{\beta_2^*}[\partial_\beta G_i(\beta_2^*|\alpha^*)|\GG](\beta_1^*-\beta_2^*),
\label{eq7.499}
\end{align}
\begin{align}
\sup_{v\in[0,L]}
\left|
\sum_{i=[n\tau_*^\beta]+1}^{[n\tau_*^\beta+v/h\Deb^2]}
\EE_{\beta_2^*}[\partial_\beta G_i(\beta_2^*|\alpha^*)|\GG](\beta_1^*-\beta_2^*)
\right|
\le
\Deb\sum_{i=[n\tau_*^\beta]+1}^{[n\tau_*^\beta+L/h\Deb^2]}
\Rd
=O_p\left(\frac{h}{\Deb}\right)
=o_p(1),
\label{eq7.498}
\end{align}
\begin{align}
&\sum_{i=[n\tau_*^\beta]+1}^{[n\tau_*^\beta+v/h\Deb^2]}
\EE_{\beta_2^*}\left[
\Bigl(
\left(
\partial_\beta G_i(\beta_2^*|\alpha^*)
-\EE_{\beta_2^*}[\partial_\beta G_i(\beta_2^*|\alpha^*)|\GG]
\right)
(\beta_1^*-\beta_2^*)
\Bigr)^2
\middle|\GG\right]
\nonumber
\\
&=(\beta_1^*-\beta_2^*)^\TT
\sum_{i=[n\tau_*^\beta]+1}^{[n\tau_*^\beta+v/h\Deb^2]}
\EE_{\beta_2^*}\biggl[
\Bigl(
\partial_\beta G_i(\beta_2^*|\alpha^*)
-\EE_{\beta_2^*}[\partial_\beta G_i(\beta_2^*|\alpha^*)|\GG]
\Bigr)^\TT
\nonumber
\\
&\qquad\qquad
\Bigl(
\partial_\beta G_i(\beta_2^*|\alpha^*)
-\EE_{\beta_2^*}[\partial_\beta G_i(\beta_2^*|\alpha^*)|\GG]
\Bigr)
\biggr|\GG\biggr]
(\beta_1^*-\beta_2^*)
\nonumber
\\
&=(\beta_1^*-\beta_2^*)^\TT
\sum_{i=[n\tau_*^\beta]+1}^{[n\tau_*^\beta+v/h\Deb^2]}
\left(
4h\Xi^\beta_{i-1}(\alpha^*,\beta_2^*)
+\Rd
\right)
(\beta_1^*-\beta_2^*)
\nonumber
\\
&\pto
4e_\beta^\TT 
\int_{\mathbb R^d}
\Xi^\beta(x,\alpha^*,\beta_0)\dd\mu_{(\alpha^*,\beta_0)}(x)
e_\beta v
=4\mathcal J_\beta v
\label{eq7.497}
\end{align}
and
\begin{align}
&\sum_{i=[n\tau_*^\beta]+1}^{[n\tau_*^\beta+v/h\Deb^2]}
\EE\left[
\Bigl(
\left(
\partial_\beta G_i(\beta_2^*|\alpha^*)
-\EE_{\beta_2^*}[\partial_\beta G_i(\beta_2^*|\alpha^*)|\GG]
\right)
(\beta_1^*-\beta_2^*)
\Bigr)^4
\middle|\GG\right]
\nonumber
\\
&=
\sum_{i=[n\tau_*^\beta]+1}^{[n\tau_*^\beta+v/h\Deb^2]}
\Deb^4\Rd
\pto0.
\label{eq7.496}
\end{align}
According to Corollary 3.8 of McLeish (1974),
we obtain, from \eqref{eq7.497} and \eqref{eq7.496}, 
\begin{align}
&\sum_{i=[n\tau_*^\beta]+1}^{[n\tau_*^\beta+v/h\Deb^2]}
\Bigl(
\partial_\beta G_i(\beta_2^*|\alpha^*)
-\EE_{\beta_2^*}[\partial_\beta G_i(\beta_2^*|\alpha^*)|\GG]
\Bigr)
(\beta_1^*-\beta_2^*)
\wto
-2\mathcal J_\beta^{1/2}\mathcal W(v)
\text{ in }\mathbb D[0,L].
\label{eq7.495}
\end{align}
Further, from \eqref{eq7.499}, \eqref{eq7.498} and \eqref{eq7.495}, 
we have
$\mathbb G_{1,n}(v)\wto
-2\mathcal J_\beta^{1/2}\mathcal W(v)$ in $\mathbb D[0,L]$.

Besides, we see
\begin{align}
&
\sup_{v\in[0,L]}\left|\sum_{i=[n\tau_*^\beta]+1}^{[n\tau_*^\beta+v/h\Deb^2]}
\partial_\beta^2 G_i(\beta_2^*|\alpha^*)
\otimes(\beta_1^*-\beta_2^*)^{\otimes2}
-2\mathcal J_\beta v\right|
\nonumber
\\
&\le
\sup_{v\in[0,L]}\left|\sum_{i=[n\tau_*^\beta]+1}^{[n\tau_*^\beta+v/h\Deb^2]}
\Bigl(\partial_\beta^2 G_i(\beta_2^*|\alpha^*)
-\EE_{\beta_2^*}\left[\partial_\beta^2 G_i(\beta_2^*|\alpha^*)\middle|\GG\right]
\Bigr)
\otimes(\beta_1^*-\beta_2^*)^{\otimes2}
\right|
\nonumber
\\
&\qquad
+\sup_{v\in[0,L]}\left|\sum_{i=[n\tau_*^\beta]+1}^{[n\tau_*^\beta+v/h\Deb^2]}
\EE_{\beta_2^*}\left[\partial_\beta^2 G_i(\beta_2^*|\alpha^*)\middle|\GG\right]
\otimes(\beta_1^*-\beta_2^*)^{\otimes2}
-2\mathcal J_\beta v
\right|,
\label{eq7.599}
\end{align}
\begin{align}
\sup_{v\in[0,L]}
\left|
\sum_{i=[n\tau_*^\beta]+1}^{[n\tau_*^\beta+v/h\Deb^2]}
\Bigl(
\partial_\beta^2 G_i(\beta_2^*|\alpha^*)
-\EE_{\beta_2^*}\left[\partial_\beta^2 G_i(\beta_2^*|\alpha^*)\middle|\GG\right]
\Bigr)
\otimes(\beta_1^*-\beta_2^*)^{\otimes2}
\right|
\pto0
\label{eq7.598}
\end{align}
and
\begin{align}
&\sup_{v\in[0,L]}\left|\sum_{i=[n\tau_*^\beta]+1}^{[n\tau_*^\beta+v/h\Deb^2]}
\EE_{\beta_2^*}\left[\partial_\beta^2 G_i(\beta_2^*|\alpha^*)\middle|\GG\right]
\otimes(\beta_1^*-\beta_2^*)^{\otimes2}
-2\mathcal J_\beta v
\right|
\nonumber
\\
&=
\sup_{v\in[0,L]}\left|\sum_{i=[n\tau_*^\beta]+1}^{[n\tau_*^\beta+v/h\Deb^2]}
\left(
2h\Xi^\beta_{i-1}(\alpha^*,\beta_2^*)+\Rd
\right)\otimes(\beta_1^*-\beta_2^*)^{\otimes2}
-2\mathcal J_\beta v
\right|
\nonumber
\\
&\le
2\sup_{v\in[0,L]}\left|\sum_{i=[n\tau_*^\beta]+1}^{[n\tau_*^\beta+v/h\Deb^2]}
h\Xi^\beta_{i-1}(\alpha^*,\beta_0)\otimes(\beta_1^*-\beta_2^*)^{\otimes2}
-\mathcal J_\beta v
\right|+o_p(1)
\nonumber
\\
&\pto 0,
\label{eq7.597}
\end{align}
where \eqref{eq7.598} is obtained by
\begin{align*}
&\EE_{\beta_2^*}\left[
\sup_{v\in[0,L]}
\left|
\sum_{i=[n\tau_*^\beta]+1}^{[n\tau_*^\beta+v/h\Deb^2]}
\left(
\partial_\beta^2 G_i(\beta_2^*|\alpha^*)
-\EE_{\beta_2^*}\left[\partial_\beta^2 G_i(\beta_2^*|\alpha^*)\middle|\GG\right]
\right)
\otimes(\beta_1^*-\beta_2^*)^{\otimes2}
\right|^2
\right]\\
&\le
C\Deb^4
\sum_{i=[n\tau_*^\beta]+1}^{[n\tau_*^\beta+L/h\Deb^2]}
\EE_{\beta_2^*}\left[
\left|
\partial_\beta^2 G_i(\beta_2^*|\alpha^*)
-\EE_{\beta_2^*}\left[\partial_\beta^2 G_i(\beta_2^*|\alpha^*)\middle|\GG\right]
\right|^2
\right]\\
&\le
C'\Deb^2\lto0
\end{align*}
from Theorem 2.11 of Hall and Heyde (1980), 
and \eqref{eq7.597} is obtained by
\begin{align*}
&\sup_{v\in[0,L]}
\left|
\sum_{i=[n\tau_*^\beta]+1}^{[n\tau_*^\beta+v/h\Deb^2]}
h\Xi^\beta_{i-1}(\alpha^*,\beta_0)\otimes(\beta_1^*-\beta_2^*)^{\otimes2}
-\mathcal J_\beta v
\right|\\
&\le
\sup_{v\in[0,\epsilon_n]}
\left|
\sum_{i=[n\tau_*^\beta]+1}^{[n\tau_*^\beta+v/h\Deb^2]}
h\Xi^\beta_{i-1}(\alpha^*,\beta_0)\otimes(\beta_1^*-\beta_2^*)^{\otimes2}
-\mathcal J_\beta v
\right|\\
&\qquad
+\sup_{v\in[\epsilon_n,L]}
\left|
\sum_{i=[n\tau_*^\beta]+1}^{[n\tau_*^\beta+v/h\Deb^2]}
h\Xi^\beta_{i-1}(\alpha^*,\beta_0)\otimes(\beta_1^*-\beta_2^*)^{\otimes2}
-\mathcal J_\beta v
\right|
\\
&\le
\epsilon_n\sup_{v\in[0,\epsilon_n]}
\left|
\frac1v\sum_{i=[n\tau_*^\beta]+1}^{[n\tau_*^\beta+v/h\Deb^2]}
h\Xi^\beta_{i-1}(\alpha^*,\beta_0)\otimes(\beta_1^*-\beta_2^*)^{\otimes2}
-\mathcal J_\beta 
\right|\\
&\qquad
+L\sup_{v\in[\epsilon_n,L]}
\left|
\frac1v\sum_{i=[n\tau_*^\beta]+1}^{[n\tau_*^\beta+v/h\Deb^2]}
h\Xi^\beta_{i-1}(\alpha^*,\beta_0)\otimes(\beta_1^*-\beta_2^*)^{\otimes2}
-\mathcal J_\beta 
\right|\\
&\le
O_p(\epsilon_n)
+L\max_{[n^{1/r}]\le k\le n}
\left|
\frac1k\sum_{i=[n\tau_*^\beta]+1}^{[n\tau_*^\beta]+k}
h\Xi^\beta_{i-1}(\alpha^*,\beta_0)\otimes(\Deb^{-1}(\beta_1^*-\beta_2^*))^{\otimes2}
-\mathcal J_\beta 
\right|\\
&=o_p(1)
\end{align*}
Here $\{\epsilon_n\}_{n=1}^\infty$ is a positive sequence such that 
$\epsilon_n\lto0$ and $\epsilon_n/\Deb^2\lto\infty$, 
and $r$ is a constant with $r\in(1,2)$ and $nh^r\lto\infty$.
From \eqref{eq7.599}-\eqref{eq7.597},
we have 
$\sup_{v\in[0,L]}|\mathbb G_{2,n}(v)-\mathcal J_\beta v|\pto 0$.
Therefore we obtain 
\begin{align*}
\mathbb G_n(v)\wto-2\mathcal J_\beta^{1/2}\mathcal W(v)+\mathcal J_\beta v
\ \text{ in }\ \mathbb D[0,L].
\end{align*}
The argument for $v<0$ is proved as well.
\end{proof}

\begin{proof}[\bf{Proof of Theorem \ref{th3}}]
It is enough to show
\begin{align}\label{7.3-1}
\varlimsup_{M\to\infty}
\varlimsup_{n\to\infty}
P\left(T\Deb^2(\hat\tau_n^\beta-\tau_*^\beta)>M\right)=0.
\end{align}
Because in the same way as the proof of Theorem \ref{th1}, 
\eqref{7.3-1} leads to 
$T\Deb^2(\hat\tau_n^\beta-\tau_*^\beta)=O_p(1)$, 
and this and Lemmas \ref{lem1}, \ref{lem4}, \ref{lem5} yield
\begin{align*}
T\Deb^2(\hat\tau_n^\beta-\tau_*^\beta)
\dto
\arg\min_{v\in\mathbb R}
\mathbb G(v).
\end{align*}

We show \eqref{7.3-1} below. For $\tau>\tau_*^\beta$, we have 
\begin{align*}
\Psi_n(\tau:\beta_1,\beta_2|\alpha)-\Psi_n(\tau_*^\beta:\beta_1,\beta_2|\alpha)
&=\sum_{i=1}^{[n\tau]}G_i(\beta_1|\alpha)
+\sum_{i=[n\tau]+1}^n G_i(\beta_2|\alpha)
-\sum_{i=1}^{[n\tau_*^\beta]}G_i(\beta_1|\alpha)
-\sum_{i=[n\tau_*^\beta]+1}^n G_i(\beta_2|\alpha)\\
&=\sum_{i=[n\tau_*^\beta]+1}^{[n\tau]}
\Big(
G_i(\beta_1|\alpha)-G_i(\beta_2|\alpha)
\Bigr).
\end{align*}
Now, from
\begin{align*}
G_i(\beta_1|\alpha)-G_i(\beta_2|\alpha)
&= G_i(\beta_1|\alpha)-G_i(\beta_2|\alpha)
-\EE_{\beta_2^*}[G_i(\beta_1|\alpha)-G_i(\beta_2|\alpha)|\GG]\\
&\quad
+h \tr\Bigl(
A_{i-1}^{-1}(\alpha)
\left(b_{i-1}(\beta_1)-b_{i-1}(\beta_2)\right)^{\otimes2}
\Bigr)\\
&\quad+
2\tr\Bigl(
A_{i-1}^{-1}(\alpha)
\left(hb_{i-1}(\beta_2)-\EE_{\beta_2^*}[\DeX|\GG]\right)
\left(b_{i-1}(\beta_1)-b_{i-1}(\beta_2)\right)^\TT
\Bigr),
\end{align*}
we see
\begin{align*}
&\Psi_n(\tau:\beta_1,\beta_2|\alpha)-\Psi_n(\tau_*^\beta:\beta_1,\beta_2|\alpha)
\\
&=\sum_{i=[n\tau_*^\beta]+1}^{[n\tau]}
\Bigl(
G_i(\beta_1|\alpha)-G_i(\beta_2|\alpha)
-\EE_{\beta_2^*}[G_i(\beta_1|\alpha)-G_i(\beta_2|\alpha)|\GG]
\Bigr)\\
&\qquad
+h\sum_{i=[n\tau_*^\beta]+1}^{[n\tau]}
\tr\Bigl(
A_{i-1}^{-1}(\alpha)
\left(b_{i-1}(\beta_1)-b_{i-1}(\beta_2)\right)^{\otimes2}
\Bigr)\\
&\qquad+
2\sum_{i=[n\tau_*^\beta]+1}^{[n\tau]}
\tr\Bigl(
A_{i-1}^{-1}(\alpha)
\left(hb_{i-1}(\beta_2)-\EE_{\beta_2^*}[\DeX|\GG]\right)
\left(b_{i-1}(\beta_1)-b_{i-1}(\beta_2)\right)^\TT
\Bigr)\\
&=:
{\mathcal M}_n^\beta(\tau:\beta_1,\beta_2|\alpha)
+{\mathcal A}_n^\beta(\tau:\beta_1,\beta_2|\alpha)
+{\varrho}_n^\beta(\tau:\beta_1,\beta_2|\alpha).
\end{align*}
Let $M\ge 1$, 
$D_{n,M}^\beta=\{\tau\in[0,1]|T\Deb^2(\tau-\tau_*^\beta)>M\}$.
For all $\delta>0$, we have
\begin{align*}
P\left(T\Deb^2(\hat\tau_n^\beta-\tau_*^\beta)>M\right)
&\le
P\left(
\inf_{\tau\in D_{n,M}^\beta}
\Psi_n(\tau:\hat\beta_1,\hat\beta_2|\hat\alpha)
\le\Psi_n(\tau_*^\beta:\hat\beta_1,\hat\beta_2|\hat\alpha)
\right)\\
&=
P\left(
\inf_{\tau\in D_{n,M}^\beta}
\left(\Psi_n(\tau:\hat\beta_1,\hat\beta_2|\hat\alpha)
-\Psi_n(\tau_*^\beta:\hat\beta_1,\hat\beta_2|\hat\alpha)\right)
\le0
\right)\\
&=P\left(
\inf_{\tau\in D_{n,M}^\beta}
\left({\mathcal M}_n^\beta(\tau:\hat\beta_1,\hat\beta_2|\hat\alpha)
+{\mathcal A}_n^\beta(\tau:\hat\beta_1,\hat\beta_2|\hat\alpha)
+{\varrho}_n^\beta(\tau:\hat\beta_1,\hat\beta_2|\hat\alpha)
\right)\le0
\right)\\
&\le P\left(
\inf_{\tau\in D_{n,M}^\beta}
\frac{
{\mathcal M}_n^\beta(\tau:\hat\beta_1,\hat\beta_2|\hat\alpha)
+{\mathcal A}_n^\beta(\tau:\hat\beta_1,\hat\beta_2|\hat\alpha)
+{\varrho}_n^\beta(\tau:\hat\beta_1,\hat\beta_2|\hat\alpha)
}{h\Deb^2([n\tau]-[n\tau_*^\beta])}
\le0
\right)\\
&\le
P\left(
\inf_{\tau\in D_{n,M}^\beta}
\frac{
{\mathcal M}_n^\beta(\tau:\hat\beta_1,\hat\beta_2|\hat\alpha)
}{h\Deb^2([n\tau]-[n\tau_*^\beta])}
\le-\delta
\right) 
+P\left(
\inf_{\tau\in D_{n,M}^\beta}
\frac{
{\mathcal A}_n^\beta(\tau:\hat\beta_1,\hat\beta_2|\hat\alpha)
}{h\Deb^2([n\tau]-[n\tau_*^\beta])}
\le2\delta
\right)\\
&\quad+P\left(
\inf_{\tau\in D_{n,M}^\beta}
\frac{
{\varrho}_n^\beta(\tau:\hat\beta_1,\hat\beta_2|\hat\alpha)
}{h\Deb^2([n\tau]-[n\tau_*^\beta])}
\le-\delta
\right)\\
&\le
P\left(
\sup_{\tau\in D_{n,M}^\beta}
\frac{
|{\mathcal M}_n^\beta(\tau:\hat\beta_1,\hat\beta_2|\hat\alpha)|
}{h\Deb^2([n\tau]-[n\tau_*^\beta])}
\ge\delta
\right) 
+P\left(
\inf_{\tau\in D_{n,M}^\beta}
\frac{
{\mathcal A}_n^\beta(\tau:\hat\beta_1,\hat\beta_2|\hat\alpha)
}{h\Deb^2([n\tau]-[n\tau_*^\beta])}
\le2\delta
\right)\\
&\quad+P\left(
\sup_{\tau\in D_{n,M}^\beta}
\frac{
|{\varrho}_n^\beta(\tau:\hat\beta_1,\hat\beta_2|\hat\alpha)|
}{h\Deb^2([n\tau]-[n\tau_*^\beta])}
\ge\delta
\right)\\
&=:
P_{1,n}^\beta+P_{2,n}^\beta+P_{3,n}^\beta.
\end{align*}

[i] Evaluation of $P_{1,n}^\beta$. 
Choose $\epsilon>0$ as you like.
Because $\partial_\beta G_i(\beta|\alpha)$ 
is continuous with respect to $\beta\in\Theta_B$, 
we can choose $\bar\beta\in\mathcal O_{\hat\beta_2}$ so that 
\begin{align*}
{\mathcal M}_n^\beta(\tau:\hat\beta_1,\hat\beta_2|\alpha)
&=\sum_{i=[n\tau_*^\beta]+1}^{[n\tau]}
\left(
G_i(\hat\beta_1|\alpha)-G_i(\hat\beta_2|\alpha)
-\EE_{\beta_2^*}[G_i(\hat\beta_1|\alpha)-G_i(\hat\beta_2|\alpha)|\GG]
\right)\\
&=\sum_{i=[n\tau_*^\beta]+1}^{[n\tau]}
\left(
\partial_\beta G_i(\bar\beta|\alpha)
-\left.\EE_{\beta_2^*}[\partial_\beta G_i(\beta|\alpha)|\GG]
\right|_{\beta=\bar\beta}
\right)
(\hat\beta_1-\hat\beta_2).
\end{align*}
If $\hat\alpha\in\mathcal O_{\alpha^*}$ and
$\hat\beta_k\in\mathcal O_{\beta_0}$, then
\begin{align*}
|{\mathcal M}_n^\beta(\tau:\hat\beta_1,\hat\beta_2|\hat\alpha)|
&\le
\left|\sum_{i=[n\tau_*^\beta]+1}^{[n\tau]}
\left(
\partial_\beta G_i(\bar\beta|\hat\alpha)
-\EE_{\beta_2^*}[\partial_\beta G_i(\bar\beta|\hat\alpha)|\GG]
\right)\right|
|\hat\beta_1-\hat\beta_2|\\
&\le
\sup_{(\alpha,\beta)\in\Theta}
\left|\sum_{i=[n\tau_*^\beta]+1}^{[n\tau]}
\left(
\partial_\beta G_i(\beta|\alpha)
-\EE_{\beta_2^*}[\partial_\beta G_i(\beta|\alpha)|\GG]
\right)\right|
|\hat\beta_1-\hat\beta_2|\\
&=:
\sup_{(\alpha,\beta)\in\Theta}
|\mathbb M_n^\beta(\tau:\beta|\alpha)||\hat\beta_1-\hat\beta_2|.
\end{align*}
Hence we have
\begin{align}
P_{1,n}^\beta
&=
P\left(
\sup_{\tau\in D_{n,M}^\beta}
\frac{
|{\mathcal M}_n^\beta(\tau:\hat\beta_1,\hat\beta_2|\hat\alpha)|
}{h\Deb^2([n\tau]-[n\tau_*^\beta])}
\ge\delta
\right) 
\nonumber
\\
&\le
P\left(
\sup_{\tau\in D_{n,M}^\beta}
\frac{
|{\mathcal M}_n^\beta(\tau:\hat\beta_1,\hat\beta_2|\hat\alpha)|
}{h\Deb^2([n\tau]-[n\tau_*^\beta])}
\ge\delta,\ 
|\hat\beta_1-\hat\beta_2|\le2\Deb,\ 
\hat\alpha\in\mathcal O_{\alpha^*},\ 
\hat\beta_k\in\mathcal O_{\beta_0}
\right)
\nonumber
\\
&\qquad+P(|\hat\beta_1-\hat\beta_2|>2\Deb)
+P(\hat\alpha\notin\mathcal O_{\alpha^*})
+P(\hat\beta_1\notin\mathcal O_{\beta_0})
+P(\hat\beta_2\notin\mathcal O_{\beta_0})
\nonumber
\\
&\le
P\left(
\sup_{\tau\in D_{n,M}^\beta}
\frac{
\sup_{(\alpha,\beta)\in\Theta}
|{\mathbb M}_n^\beta(\tau:\beta|\alpha)|
}{h\Deb^2([n\tau]-[n\tau_*^\beta])}
|\hat\beta_1-\hat\beta_2|
\ge\delta,\ 
|\hat\beta_1-\hat\beta_2|\le2\Deb
\right)
\nonumber
\\
&\qquad+P(|\hat\beta_1-\hat\beta_2|>2\Deb)
+P(\hat\alpha\notin\mathcal O_{\alpha^*})
+P(\hat\beta_1\notin\mathcal O_{\beta_0})
+P(\hat\beta_2\notin\mathcal O_{\beta_0})
\nonumber
\\
&\le
P\left(
\sup_{\tau\in D_{n,M}^\beta}
\frac{
2 \sup_{(\alpha,\beta)\in\Theta}
|{\mathbb M}_n^\beta(\tau:\beta|\alpha)|
}{h\Deb([n\tau]-[n\tau_*^\beta])}
\ge\delta
\right)
\nonumber
\\
&\qquad+P(|\hat\beta_1-\hat\beta_2|>2\Deb)
+P(\hat\alpha\notin\mathcal O_{\alpha^*})
+P(\hat\beta_1\notin\mathcal O_{\beta_0})
+P(\hat\beta_2\notin\mathcal O_{\beta_0}).
\label{eq7.199}
\end{align}
By the uniform version on the H\'ajek-Renyi inequality in Lemma 2 
of Iacus and Yoshida (2012), 
we obtain
\begin{align}
&P\left(
\sup_{\tau\in D_{n,M}^\beta}
\frac{
2 \sup_{(\alpha,\beta)\in\Theta}
|{\mathbb M}_n^\beta(\tau:\beta|\alpha)|
}{h\Deb([n\tau]-[n\tau_*^\beta])}
\ge\delta
\right)
\nonumber
\\
&=
P\left(
\sup_{\tau\in D_{n,M}^\beta}
\frac{
2 \sup_{(\alpha,\beta)\in\Theta}
|{\mathbb M}_n^\beta(\tau:\beta|\alpha)|
}{[n\tau]-[n\tau_*^\beta]}
\ge\delta h\Deb
\right)
\nonumber
\\
&\le
P\left(
\max_{j>M/h\Deb^2-1}
\frac{2}{j}
\sup_{(\alpha,\beta)\in\Theta}
\left|
\sum_{i=[n\tau_*^\beta]+1}^{[n\tau_*^\beta]+j}
\left(
\partial_\beta G_i(\beta|\alpha)-
\EE[\partial_\beta G_i(\beta|\alpha)|\GG]
\right)
\right|
\ge\delta h\Dea
\right)
\nonumber
\\
&\le
\sum_{j>M/h\Deb^2-1}\frac{Ch}{(\delta h\Deb j)^2}
\nonumber
\\
&\le
\frac{C'h}{(\delta h\Deb)^2}\frac{h\Deb^2}{M}
=\frac{C'}{\delta^2M}
=:\gamma_\beta(M).
\label{eq7.198}
\end{align}
Noting that if
$|\hat\beta_k-\beta_k^*|\le\Deb/2$ holds, 
then 
$|\hat\beta_1-\hat\beta_2|
\le
|\hat\beta_1-\beta_1^*|
+|\beta_1^*-\beta_2^*|
+|\hat\beta_2-\beta_2^*|
\le 2\Deb$, 
i.e., 
\begin{align*}
\{|\hat\beta_1-\hat\beta_2|>2\Deb\}
\subset
\left\{|\hat\beta_1-\beta_1^*|>\frac{\Deb}{2}\right\}
\cup\left\{|\hat\beta_2-\beta_2^*|>\frac{\Deb}{2}\right\},
\end{align*}
and for sufficiently large $n$ so that 
$P(|\hat\beta_k-\beta_k^*|>\Deb/2)<\epsilon/2$
because of $\Deb^{-1}(\hat\beta_k-\beta_k^*)=o_p(1)$, 
we see
\begin{align}
P(|\hat\beta_1-\hat\beta_2|>2\Deb)
&\le
P\left(|\hat\beta_1-\beta_1^*|>\frac{\Deb}{2}
\right)
+P\left(
|\hat\beta_2-\beta_2^*|>\frac{\Deb}{2}
\right)
<\epsilon.
\label{eq7.197}
\end{align}
From \textbf{[C6-II]} and \textbf{[A2-II]}, 
we have 
$\hat\alpha\pto\alpha^*$ and 
$\hat\beta_k\pto\beta_0$ as $n\to\infty$,
that is, 
$P(\hat\alpha\notin\mathcal O_{\alpha^*})<\epsilon$ and
$P(\hat\beta_k\notin\mathcal O_{\beta_0})<\epsilon$ for large $n$.
Therefore, from \eqref{eq7.199}-\eqref{eq7.197} and these, we have
$P_{1,n}^\beta\le\gamma_{\beta}(M)+4\epsilon$ for large $n$.


[ii] Evaluation of $P_{2,n}^\beta$.
If $\hat\alpha\in\mathcal O_{\alpha^*}$ and $\hat\beta_k\in\mathcal O_{\beta_0}$,
then we have
\begin{align*}
&\tr\left(
A_{i-1}^{-1}(\hat\alpha)
(b_{i-1}(\hat\beta_1)-b_{i-1}(\hat\beta_2))^{\otimes2} 
\right)\\
&=
\tr\left(
A_{i-1}^{-1}(\alpha^*)
(b_{i-1}(\hat\beta_1)-b_{i-1}(\hat\beta_2))^{\otimes2} 
\right)\\
&\qquad+
\int_0^1
\left.\partial_\alpha
\tr\left(
A_{i-1}^{-1}(\alpha)
(b_{i-1}(\hat\beta_1)-b_{i-1}(\hat\beta_2))^{\otimes2}\right)
\right|_{\alpha=\alpha^*+u(\hat\alpha-\alpha^*)}
\dd u
(\hat\alpha-\alpha^*)\\
&=
\left.\partial_\beta
\tr\left(
A_{i-1}^{-1}(\alpha^*)
(b_{i-1}(\beta)-b_{i-1}(\hat\beta_2))^{\otimes2} 
\right)\right|_{\beta=\hat\beta_2}
(\hat\beta_1-\hat\beta_2)
\\
&\qquad+
\frac12\left.\partial_\beta^2
\tr\left(
A_{i-1}^{-1}(\alpha^*)
(b_{i-1}(\beta)-b_{i-1}(\hat\beta_2))^{\otimes2} 
\right)\right|_{\beta=\hat\beta_2}
\otimes(\hat\beta_1-\hat\beta_2)^{\otimes2}
\\
&\qquad+
\frac{1}{3!}\left.\partial_\beta^3
\tr\left(
A_{i-1}^{-1}(\alpha^*)
(b_{i-1}(\beta)-b_{i-1}(\hat\beta_2))^{\otimes2} 
\right)\right|_{\beta=\hat\beta_2}
\otimes(\hat\beta_1-\hat\beta_2)^{\otimes3}
\\
&\qquad+
\int_0^1
\frac{(1-u)^3}{3!}\left.\partial_\beta^4
\tr\left(
A_{i-1}^{-1}(\alpha^*)
(b_{i-1}(\beta)-b_{i-1}(\hat\beta_2))^{\otimes2} 
\right)\right|_{\beta=\hat\beta_2+u(\hat\beta_1-\hat\beta_2)}
\dd u
\otimes(\hat\beta_1-\hat\beta_2)^{\otimes4}
\\
&\qquad+
\int_0^1\int_0^1
\left.\partial_\beta\partial_\alpha
\tr\left(
A_{i-1}^{-1}(\alpha)
(b_{i-1}(\beta)-b_{i-1}(\hat\beta_2))^{\otimes2}\right)
\right|_{\substack{\alpha=\alpha^*+u(\hat\alpha-\alpha^*)\\
\beta=\hat\beta_2+v(\hat\beta_1-\hat\beta_2)
}}
\dd u\dd v
\otimes(\hat\alpha-\alpha^*)
\otimes(\hat\beta_1-\hat\beta_2)\\
&=
\Xi_{i-1}^\beta(\alpha^*,\hat\beta_2)
\otimes(\hat\beta_1-\hat\beta_2)^{\otimes2}\\
&\qquad+
\frac{1}{3!}\left.\partial_\beta^3
\tr\left(
A_{i-1}^{-1}(\alpha^*)
(b_{i-1}(\beta)-b_{i-1}(\hat\beta_2))^{\otimes2} 
\right)\right|_{\beta=\hat\beta_2}
\otimes(\hat\beta_1-\hat\beta_2)^{\otimes3}
\\
&\qquad+
\int_0^1
\frac{(1-u)^3}{3!}\left.\partial_\beta^4
\tr\left(
A_{i-1}^{-1}(\alpha^*)
(b_{i-1}(\beta)-b_{i-1}(\hat\beta_2))^{\otimes2} 
\right)\right|_{\beta=\hat\beta_2+u(\hat\beta_1-\hat\beta_2)}
\dd u
\otimes(\hat\beta_1-\hat\beta_2)^{\otimes4}
\\
&\qquad+
\int_0^1\int_0^1
\left.\partial_\beta\partial_\alpha
\tr\left(
A_{i-1}^{-1}(\alpha)
(b_{i-1}(\beta)-b_{i-1}(\hat\beta_2))^{\otimes2}\right)
\right|_{\substack{\alpha=\alpha^*+u(\hat\alpha-\alpha^*)\\
\beta=\hat\beta_2+v(\hat\beta_1-\hat\beta_2)
}}
\dd u\dd v
\otimes(\hat\alpha-\alpha^*)
\otimes(\hat\beta_1-\hat\beta_2)\\
&=
\Xi_{i-1}^\beta(\alpha^*,\beta_0)
\otimes(\hat\beta_1-\hat\beta_2)^{\otimes2}
+
\partial_\beta\Xi_{i-1}^\beta(\alpha^*,\beta_0)
\otimes(\hat\beta_1-\hat\beta_2)^{\otimes2}
\otimes(\hat\beta_2-\beta_0)\\
&\qquad+
\int_0^1
(1-u)
\partial_\beta^2\Xi_{i-1}^\beta(\alpha^*,\beta_0+u(\hat\beta_2-\beta_0))
\dd u
\otimes(\hat\beta_1-\hat\beta_2)^{\otimes2}
\otimes(\hat\beta_2-\beta_0)^{\otimes2}\\
&\qquad+
\frac{1}{3!}\left.\partial_\beta^3
\tr\left(
A_{i-1}^{-1}(\alpha^*)
(b_{i-1}(\beta)-b_{i-1}(\hat\beta_2))^{\otimes2} 
\right)\right|_{\beta=\hat\beta_2}
\otimes(\hat\beta_1-\hat\beta_2)^{\otimes3}
\\
&\qquad+
\int_0^1
\frac{(1-u)^3}{3!}\left.\partial_\beta^4
\tr\left(
A_{i-1}^{-1}(\alpha^*)
(b_{i-1}(\beta)-b_{i-1}(\hat\beta_2))^{\otimes2} 
\right)\right|_{\beta=\hat\beta_2+u(\hat\beta_1-\hat\beta_2)}
\dd u
\otimes(\hat\beta_1-\hat\beta_2)^{\otimes4}
\\
&\qquad+
\int_0^1\int_0^1
\left.\partial_\beta\partial_\alpha
\tr\left(
A_{i-1}^{-1}(\alpha)
(b_{i-1}(\beta)-b_{i-1}(\hat\beta_2))^{\otimes2}\right)
\right|_{\substack{\alpha=\alpha^*+u(\hat\alpha-\alpha^*)\\
\beta=\hat\beta_2+v(\hat\beta_1-\hat\beta_2)
}}
\dd u\dd v
\otimes(\hat\alpha-\alpha^*)
\otimes(\hat\beta_1-\hat\beta_2)\\
&\ge
\left(
\lambda_1[\Xi_{i-1}^\beta(\alpha^*,\beta_0)]+r_{i-1}
\right)
|\hat\beta_1-\hat\beta_2|^2,
\end{align*}
where $r_{i-1}$ satisfies 
\begin{align*}
\sup_{\tau\in D_{n,M}^\beta}
\left|
\frac{1}{[n\tau]-[n\tau_*^\beta]}
\sum_{i=[n\tau_*^\beta]+1}^{[n\tau]}r_{i-1}
\right|
=o_p(1)
\end{align*}
from \textbf{[A1-II]} and \textbf{[A2-II]}. Therefore we have
\begin{align}
P_{2,n}^\beta
&=
P\left(
\inf_{\tau\in D_{n,M}^\beta}
\frac{
{\mathcal A}_n^\beta(\tau:\hat\beta_1,\hat\beta_2|\hat\alpha)
}{h\Deb^2([n\tau]-[n\tau_*^\beta])}
\le2\delta
\right)
\nonumber
\\
&\le
P\left(
\inf_{\tau\in D_{n,M}^\beta}
\frac{
{\mathcal A}_n^\beta(\tau:\hat\beta_1,\hat\beta_2|\hat\alpha)
}{h\Deb^2([n\tau]-[n\tau_*^\beta])}
\le2\delta,\ 
|\hat\beta_1-\hat\beta_2|\ge\frac{\Deb}{2},\ 
\hat\alpha\in\mathcal O_{\alpha^*},\ 
\hat\beta_k\in\mathcal O_{\beta_0}
\right)
\nonumber
\\
&\qquad+P\left(|\hat\beta_1-\hat\beta_2|<\frac{\Deb}{2}\right)
+P(\hat\alpha\notin\mathcal O_{\alpha^*})
+P(\hat\beta_1\notin\mathcal O_{\beta_0})
+P(\hat\beta_2\notin\mathcal O_{\beta_0})
\nonumber
\\
&\le
P\left(
\inf_{\tau\in D_{n,M}^\beta}
\frac{1}{\Deb^2([n\tau]-[n\tau_*^\beta])}
\sum_{i=[n\tau_*^\beta]+1}^{[n\tau]}
\left(
\lambda_1[\Xi_{i-1}^\beta(\alpha^*,\beta_0)]+r_{i-1}
\right)
|\hat\beta_1-\hat\beta_2|^2
\le2\delta,\ 
|\hat\beta_1-\hat\beta_2|\ge\frac{\Deb}{2}
\right)
\nonumber
\\
&\qquad+P\left(|\hat\beta_1-\hat\beta_2|<\frac{\Deb}{2}\right)
+P(\hat\alpha\notin\mathcal O_{\alpha^*})
+P(\hat\beta_1\notin\mathcal O_{\beta_0})
+P(\hat\beta_2\notin\mathcal O_{\beta_0})
\nonumber
\\
&\le
P\left(
\inf_{\tau\in D_{n,M}^\beta}
\frac{1}{[n\tau]-[n\tau_*^\beta]}
\sum_{i=[n\tau_*^\beta]+1}^{[n\tau]}
\left(
\lambda_1[\Xi_{i-1}^\beta(\beta_0)]+r_{i-1}
\right)
\le8\delta 
\right)
\nonumber
\\
&\qquad+P\left(|\hat\beta_1-\hat\beta_2|<\frac{\Deb}{2}\right)
+P(\hat\alpha\notin\mathcal O_{\alpha^*})
+P(\hat\beta_1\notin\mathcal O_{\beta_0})
+P(\hat\beta_2\notin\mathcal O_{\beta_0}).
\label{eq7.299}
\end{align}
According to \textbf{[A2-II]},
if we set
\begin{align*}
\delta=\frac{1}{10}
\int_{\mathbb R^d}\lambda_1[\Xi^\beta(x,\alpha^*,\beta_0)]
\dd\mu_{(\alpha^*,\beta_0)}(x)>0,
\end{align*}
then for large $n$,
\begin{align}
&P\left(
\inf_{\tau\in D_{n,M}^\beta}
\frac{1}{[n\tau]-[n\tau_*^\beta]}
\sum_{i=[n\tau_*^\beta]+1}^{[n\tau]}
\left(
\lambda_1[\Xi_{i-1}^\beta(\alpha^*,\beta_0)]+r_{i-1}
\right)
\le8\delta
\right)
\nonumber
\\
&\le
P\left(
\inf_{\tau\in D_{n,M}^\beta}
\frac{1}{[n\tau]-[n\tau_*^\beta]}
\sum_{i=[n\tau_*^\beta]+1}^{[n\tau]}
\lambda_1[\Xi_{i-1}^\beta(\alpha^*,\beta_0)]
\le9\delta
\right)
+
P\left(
\inf_{\tau\in D_{n,M}^\beta}
\frac{1}{[n\tau]-[n\tau_*^\beta]}
\sum_{i=[n\tau_*^\beta]+1}^{[n\tau]}
r_{i-1}
\le-\delta
\right)
\nonumber
\\
&\le
P\left(
\inf_{\tau\in D_{n,M}^\beta}
\frac{1}{[n\tau]-[n\tau_*^\beta]}
\sum_{i=[n\tau_*^\beta]+1}^{[n\tau]}
\lambda_1[\Xi_{i-1}^\beta(\alpha^*,\beta_0)]
\le9\delta
\right)
+
P\left(
\sup_{\tau\in D_{n,M}^\beta}
\left|
\frac{1}{[n\tau]-[n\tau_*^\beta]}
\sum_{i=[n\tau_*^\beta]+1}^{[n\tau]}
r_{i-1}
\right|
\ge\delta
\right)
\nonumber
\\
&\le
P\left(
\inf_{\tau\in D_{n,M}^\beta}
\left(
\frac{1}{[n\tau]-[n\tau_*^\beta]}
\sum_{i=[n\tau_*^\beta]+1}^{[n\tau]}
\lambda_1[\Xi_{i-1}^\beta(\alpha^*,\beta_0)]
-10\delta
\right)
\le-\delta
\right)
+\epsilon
\nonumber
\\
&\le
P\left(
\sup_{\tau\in D_{n,M}^\beta}
\left|
\frac{1}{[n\tau]-[n\tau_*^\beta]}
\sum_{i=[n\tau_*^\beta]+1}^{[n\tau]}
\lambda_1[\Xi_{i-1}^\beta(\alpha^*,\beta_0)]
-10\delta
\right|
\ge\delta
\right)
+\epsilon
\nonumber
\\
&\le
P\left(
\sup_{k> M/h\Deb^2-1}
\left|
\frac{1}{k}
\sum_{i=[n\tau_*^\beta]+1}^{[n\tau_*^\beta]+k}
\lambda_1[\Xi_{i-1}^\beta(\alpha^*,\beta_0)]
-10\delta
\right|
\ge\delta
\right)
+\epsilon
\nonumber
\\
&\le
P\left(
\max_{[n^{1/r}]\le k\le n}
\left|
\frac{1}{k}
\sum_{i=[n\tau_*^\beta]+1}^{[n\tau_*^\beta]+k}
\lambda_1[\Xi_{i-1}^\beta(\alpha^*,\beta_0)]
-10\delta
\right|
\ge\delta
\right)
+\epsilon
\nonumber
\\
&\le 2\epsilon.
\label{eq7.298}
\end{align}

Noting that if
$|\hat\beta_k-\beta_k^*|\le\Deb/4$ holds, 
then 
$\Deb=|\beta_1^*-\beta_2^*|
\le
|\beta_1^*-\hat\beta_1|
+|\hat\beta_1-\hat\beta_2|
+|\hat\beta_2-\beta_2^*|
\le |\hat\beta_1-\hat\beta_2|+\Deb/2$, 
i.e., 
\begin{align*}
\left\{|\hat\beta_1-\hat\beta_2|<\frac\Deb2\right\}
\subset
\left\{|\hat\beta_1-\beta_1^*|>\frac{\Deb}{4}\right\}
\cup\left\{|\hat\beta_2-\beta_2^*|>\frac{\Deb}{4}\right\}
\end{align*}
and for sufficiently large $n$ so that 
$P(|\hat\beta_k-\beta_k^*|>\Deb/4)<\epsilon/2$
because of $\Deb^{-1}(\hat\beta_k-\beta_k^*)=o_p(1)$,
we see
\begin{align}
P\left(
|\hat\beta_1-\hat\beta_2|<\frac\Deb2
\right)
&\le
P\left(|\hat\beta_1-\beta_1^*|>\frac{\Deb}{4}
\right)
+P\left(
|\hat\beta_2-\beta_2^*|>\frac{\Deb}{4}
\right)
<\epsilon.
\label{eq7.297}
\end{align}
Therefore, from \eqref{eq7.299}-\eqref{eq7.297}, we obtain
$P_{2,n}^\beta\le6\epsilon$ for large $n$.

[iii] Evaluation of $P_{3,n}^\beta$.
We have, for large $n$,
\begin{align*}
&\tr\Bigl(
A_{i-1}^{-1}(\hat\alpha)
(b_{i-1}(\hat\beta_2)-h\EE[\DeX|\GG])
(b_{i-1}(\hat\beta_1)-b_{i-1}(\hat\beta_2))^\TT
\Bigr)\\
&\le
h\tr\Bigl(
A_{i-1}^{-1}(\hat\alpha)
(b_{i-1}(\hat\beta_2)-b_{i-1}(\beta_2^*))
(b_{i-1}(\hat\beta_1)-b_{i-1}(\hat\beta_2))^\TT
\Bigr)
+
\Rd|b_{i-1}(\hat\beta_1)-b_{i-1}(\hat\beta_2)|\\
&=
h\left[\tr\Bigl(
A_{i-1}^{-1}(\hat\alpha)
\partial_{\beta^{\ell_1}} b_{i-1}(\beta_2^*)
(b_{i-1}(\hat\beta_1)-b_{i-1}(\hat\beta_2))^\TT
\Bigr)
\right]_{\ell_1}
(\hat\beta_2-\beta_2^*)\\
&\qquad+
h\int_0^1
(1-u)
\left.
\left[\tr\Bigl(
A_{i-1}^{-1}(\hat\alpha)
\partial_{\beta^{\ell_2}}\partial_{\beta^{\ell_1}} b_{i-1}(\beta)
(b_{i-1}(\hat\beta_1)-b_{i-1}(\hat\beta_2))^\TT
\Bigr)
\right]_{\ell_1,\ell_2}
\right|_{\beta=\beta_2^*+u(\hat\beta_2-\beta_2^*)}
\dd u\\
&\qquad\qquad
\otimes(\hat\beta_2-\beta_2^*)^{\otimes2}
+
\Rd|b_{i-1}(\hat\beta_1)-b_{i-1}(\hat\beta_2)|\\
&=
h\left[\tr\Bigl(
A_{i-1}^{-1}(\hat\alpha)
\partial_{\beta^{\ell_1}} b_{i-1}(\beta_2^*)
\partial_{\beta^{\ell_2}}b_{i-1}(\hat\beta_2)^\TT
\Bigr)
\right]_{\ell_1,\ell_2}
\otimes(\hat\beta_2-\beta_2^*)
\otimes(\hat\beta_1-\hat\beta_2)\\
&\qquad+
h\int_0^1
(1-u)
\left.
\left[\tr\Bigl(
A_{i-1}^{-1}(\hat\alpha)
\partial_{\beta^{\ell_1}} b_{i-1}(\beta_2^*)
\partial_{\beta^{\ell_3}}\partial_{\beta^{\ell_2}}b_{i-1}(\beta)^\TT
\Bigr)
\right]_{\ell_1,\ell_2,\ell_3}
\right|_{\beta=\hat\beta_2+u(\hat\beta_1-\hat\beta_2)}
\dd u\\
&\qquad\qquad
\otimes(\hat\beta_2-\beta_2^*)
\otimes(\hat\beta_1-\hat\beta_2)^{\otimes2}\\
&\qquad+
h\int_0^1
(1-u)
\left.
\left[\tr\Bigl(
A_{i-1}^{-1}(\hat\alpha)
\partial_{\beta^{\ell_2}}\partial_{\beta^{\ell_1}} b_{i-1}(\beta)
(b_{i-1}(\hat\beta_1)-b_{i-1}(\hat\beta_2))^\TT
\Bigr)
\right]_{\ell_1,\ell_2}
\right|_{\beta=\beta_2^*+u(\hat\beta_2-\beta_2^*)}
\dd u\\
&\qquad\qquad
\otimes(\hat\beta_2-\beta_2^*)^{\otimes2}
+
\Rd|b_{i-1}(\hat\beta_1)-b_{i-1}(\hat\beta_2)|\\
&=
h\Xi_{i-1}^\beta(\hat\alpha,\beta_2^*)
\otimes(\hat\beta_2-\beta_2^*)
\otimes(\hat\beta_1-\hat\beta_2)\\
&\qquad+
h\int_0^1
\left[\tr\Bigl(
A_{i-1}^{-1}(\hat\alpha)
\partial_{\beta^{\ell_1}} b_{i-1}(\beta_2^*)
\partial_{\beta^{\ell_3}}\partial_{\beta^{\ell_2}}
b_{i-1}(\beta_2^*+u(\hat\beta_2-\beta_2^*))^\TT
\Bigr)
\right]_{\ell_1,\ell_2,\ell_3}
\dd u
\\
&\qquad\qquad
\otimes(\hat\beta_2-\beta_2^*)^{\otimes2}
\otimes(\hat\beta_1-\hat\beta_2)\\
&\qquad+
h\int_0^1
(1-u)
\left[\tr\Bigl(
A_{i-1}^{-1}(\hat\alpha)
\partial_{\beta^{\ell_1}} b_{i-1}(\beta_2^*)
\partial_{\beta^{\ell_3}}\partial_{\beta^{\ell_2}}
b_{i-1}(\hat\beta_2+u(\hat\beta_1-\hat\beta_2))^\TT
\Bigr)
\right]_{\ell_1,\ell_2,\ell_3}
\dd u
\\
&\qquad\qquad
\otimes(\hat\beta_2-\beta_2^*)
\otimes(\hat\beta_1-\hat\beta_2)^{\otimes2}\\
&\qquad+
h\int_0^1
(1-u)
\left[\tr\Bigl(
A_{i-1}^{-1}(\hat\alpha)
\partial_{\beta^{\ell_2}}\partial_{\beta^{\ell_1}} 
b_{i-1}(\beta_2^*+u(\hat\beta_2-\beta_2^*))
(b_{i-1}(\hat\beta_1)-b_{i-1}(\hat\beta_2))^\TT
\Bigr)
\right]_{\ell_1,\ell_2}
\dd u
\\
&\qquad\qquad
\otimes(\hat\beta_2-\beta_2^*)^{\otimes2}
+
\Rd|b_{i-1}(\hat\beta_1)-b_{i-1}(\hat\beta_2)|\\
&=
h\Xi_{i-1}^\beta(\alpha^*,\beta_0)
\otimes(\hat\beta_2-\beta_2^*)
\otimes(\hat\beta_1-\hat\beta_2)\\
&
\qquad+
h\int_0^1
\left.
\partial_{(\alpha,\beta)}\Xi_{i-1}^\beta(\alpha,\beta)
\right|_{
\substack
{\alpha=\alpha^*+u(\hat\alpha-\alpha^*)\\
\beta=\beta_0+u(\beta_2^*-\beta_0)}
}
\dd u
\otimes(\hat\beta_2-\beta_2^*)
\otimes(\hat\beta_1-\hat\beta_2)
\otimes
\left(
\begin{array}{cc}
\hat\alpha-\alpha^*\\
\beta_2^*-\beta_0
\end{array}
\right)
\\
&\qquad+
h\int_0^1
\left[\tr\Bigl(
A_{i-1}^{-1}(\hat\alpha)
\partial_{\beta^{\ell_1}} b_{i-1}(\beta_2^*)
\partial_{\beta^{\ell_3}}\partial_{\beta^{\ell_2}}
b_{i-1}(\beta_2^*+u(\hat\beta_2-\beta_2^*))^\TT
\Bigr)
\right]_{\ell_1,\ell_2,\ell_3}
\dd u
\\
&\qquad\qquad
\otimes(\hat\beta_2-\beta_2^*)^{\otimes2}
\otimes(\hat\beta_1-\hat\beta_2)\\
&\qquad+
h\int_0^1
(1-u)
\left[\tr\Bigl(
A_{i-1}^{-1}(\hat\alpha)
\partial_{\beta^{\ell_1}} b_{i-1}(\beta_2^*)
\partial_{\beta^{\ell_3}}\partial_{\beta^{\ell_2}}
b_{i-1}(\hat\beta_2+u(\hat\beta_1-\hat\beta_2))^\TT
\Bigr)
\right]_{\ell_1,\ell_2,\ell_3}
\dd u
\\
&\qquad\qquad
\otimes(\hat\beta_2-\beta_2^*)
\otimes(\hat\beta_1-\hat\beta_2)^{\otimes2}\\
&\qquad+
h\int_0^1
(1-u)
\left[\tr\Bigl(
A_{i-1}^{-1}(\hat\alpha)
\partial_{\beta^{\ell_2}}\partial_{\beta^{\ell_1}} 
b_{i-1}(\beta_2^*+u(\hat\beta_2-\beta_2^*))
(b_{i-1}(\hat\beta_1)-b_{i-1}(\hat\beta_2))^\TT
\Bigr)
\right]_{\ell_1,\ell_2}
\dd u
\\
&\qquad\qquad
\otimes(\hat\beta_2-\beta_2^*)^{\otimes2}
+
\Rd|b_{i-1}(\hat\beta_1)-b_{i-1}(\hat\beta_2)|\\
&\le
h\Xi_{i-1}^\beta(\alpha^*,\beta_0)
\otimes(\hat\beta_2-\beta_2^*)
\otimes(\hat\beta_1-\hat\beta_2)\\
&\qquad+
\frac{h\Deb}{\sqrt T}\left|\int_0^1
\left.
\partial_{(\alpha,\beta)}\Xi_{i-1}^\beta(\alpha,\beta)
\right|_{
\substack
{\alpha=\alpha^*+u(\hat\alpha-\alpha^*)\\
\beta=\beta_0+u(\beta_2^*-\beta_0)}
}
\dd u
\right|
\sqrt{T}|\hat\beta_2-\beta_2^*|
\Deb^{-1}|\hat\beta_1-\hat\beta_2|
(|\hat\alpha-\alpha^*|+|\beta_2^*-\beta_0|)\\
&\qquad+
\frac{h\Deb}{T}\left|
\int_0^1
\left[\tr\Bigl(
A_{i-1}^{-1}(\hat\alpha)
\partial_{\beta^{\ell_1}} b_{i-1}(\beta_2^*)
\partial_{\beta^{\ell_3}}\partial_{\beta^{\ell_2}}
b_{i-1}(\beta_2^*+u(\hat\beta_2-\beta_2^*))^\TT
\Bigr)
\right]_{\ell_1,\ell_2,\ell_3}
\dd u
\right|
\\
&\qquad\qquad
\times
\left(\sqrt{T}|\hat\beta_2-\beta_2^*|\right)^2
\Deb^{-1}|\hat\beta_1-\hat\beta_2|\\
&\qquad+
\frac{h\Deb^2}{\sqrt{T}}
\left|
\int_0^1
(1-u)
\left[\tr\Bigl(
A_{i-1}^{-1}(\hat\alpha)
\partial_{\beta^{\ell_1}} b_{i-1}(\beta_2^*)
\partial_{\beta^{\ell_3}}\partial_{\beta^{\ell_2}}
b_{i-1}(\hat\beta_2+u(\hat\beta_1-\hat\beta_2))^\TT
\Bigr)
\right]_{\ell_1,\ell_2,\ell_3}
\dd u
\right|
\\
&\qquad\qquad
\times
\sqrt{T}|\hat\beta_2-\beta_2^*|
\left(\Deb^{-1}|\hat\beta_1-\hat\beta_2|\right)^2\\
&\qquad+
\frac{h\Deb}{T}
\biggl|
\int_0^1\int_0^1
(1-u)
\Bigl[\tr\Bigl(
A_{i-1}^{-1}(\hat\alpha)
\partial_{\beta^{\ell_2}}\partial_{\beta^{\ell_1}} 
b_{i-1}(\beta_2^*+u(\hat\beta_2-\beta_2^*))
\\
&\qquad\qquad\qquad
\times
\partial_{\beta^{\ell_3}}b_{i-1}(\hat\beta_2+v(\hat\beta_1-\hat\beta_2))^\TT
\Bigr)
\Bigr]_{\ell_1,\ell_2,\ell_3}
\dd u\dd v
\biggr|
\left(\sqrt{T}|\hat\beta_2-\beta_2^*|\right)^2
\Deb^{-1}|\hat\beta_1-\hat\beta_2|
\\
&\qquad\qquad+
\Rd|b_{i-1}(\hat\beta_1)-b_{i-1}(\hat\beta_2)|\\
&\le
\frac{h\Deb}{\sqrt T}\Xi_{i-1}^\beta(\alpha^*,\beta_0)
\otimes\sqrt{T}(\hat\beta_2-\beta_2^*)
\otimes\Deb^{-1}(\hat\beta_1-\hat\beta_2)
+
\left(
\frac{h\Deb}{\sqrt{nT}}
+\frac{h\Deb^2}{\sqrt T}
+\frac{h\Deb}{T}
+h^2
\right)\Ro
\end{align*}
and
\begin{align*}
\sup_{\tau\in D_{n,M}^\beta}
\frac{|\varrho_n^\beta(\tau:\hat\beta_1,\hat\beta_2|\hat\alpha)|}
{h\Deb^2([n\tau]-[n\tau_*^\beta])}
&\le
\frac{1}{\sqrt{T}\Deb}
\sup_{\tau\in D_{n,M}^\beta}
\left|
\frac{1}{[n\tau]-[n\tau_*^\beta]}
\sum_{i=[n\tau_*^\beta]+1}^{[n\tau]}
\Xi_{i-1}^\beta(\alpha^*,\beta_0)
\right|
\sqrt{T}|\hat\beta_2-\beta_2^*|
\Deb^{-1}|\hat\beta_1-\hat\beta_2|\\
&
\qquad+
\left(
\frac{h\Deb}{\sqrt{nT}}
+\frac{h\Deb^2}{\sqrt T}
+\frac{h\Deb}{T}
+h^2
\right)\sup_{\tau\in D_{n,M}^\beta}
\frac{1}{h\Deb^2([n\tau]-[n\tau_*^\beta])}
\sum_{i=[n\tau_*^\beta]+1}^{[n\tau]}
\Ro
\\
&\le
O_p\left(
\frac{1}{\sqrt{T}\Deb}
\right)
+O_p(\sqrt h\Deb)
+O_p(\sqrt{T}\Deb^2)
+O_p(\Deb)
+O_p(nh^2)
=o_p(1).
\end{align*}
Hence, we see
\begin{align*}
P_{3,n}^\beta
&\le
P\left(
\sup_{\tau\in D_{n,M}^\beta}
\frac{|\varrho_n^\beta(\tau:\hat\beta_1,\hat\beta_2|\hat\alpha)|}
{h\Deb^2([n\tau]-[n\tau_*^\beta])}
\ge\delta,\ 
\hat\alpha\in\mathcal O_{\alpha^*},\ 
\hat\beta_1,\hat\beta_2\in\mathcal O_{\beta_0}
\right)\\
&\qquad+P(\hat\alpha\notin\mathcal O_{\alpha^*})
+P(\hat\beta_1\notin\mathcal O_{\beta_0})
+P(\hat\beta_2\notin\mathcal O_{\beta_0})
\\
&\le4\epsilon
\end{align*}
for large $n$.


[iv] From the evaluations in Steps [i]-[iii], we have
\begin{align*}
\varlimsup_{n\to\infty}
P(T\Deb^2(\hat\tau_n^\beta-\tau_*^\beta)>M)
\le
\gamma_\beta(M)+14\epsilon
\end{align*}
for any $M\ge1$ and $\epsilon>0$. 
Therefore
\begin{align*}
\varlimsup_{M\to\infty}
\varlimsup_{n\to\infty}
P(T\Deb^2(\hat\tau_n^\beta-\tau_*^\beta)>M)
\le
14\epsilon.
\end{align*}

\end{proof}

\begin{proof}[\bf{Proof of Theorem \ref{th4}}]
It is sufficient to show 
\begin{align*}
\varlimsup_{M\to\infty}
\varlimsup_{n\to\infty}
P(T(\hat\tau_n^\beta-\tau_*^\beta)>M)
=0.
\end{align*} 

Let $M\ge 1$, 
$D_{n,M}^\beta=\{\tau\in[0,1]|T(\tau-\tau_*^\beta)>M\}$.
For all $\delta>0$, we have
\begin{align*}
P\left(T(\hat\tau_n^\beta-\tau_*^\beta)>M\right)
&\le
P\left(
\sup_{\tau\in D_{n,M}^\beta}
\frac{
|{\mathcal M}_n^\beta(\tau:\hat\beta_1,\hat\beta_2|\hat\alpha)|
}{h([n\tau]-[n\tau_*^\beta])}
\ge\delta
\right) 
+P\left(
\inf_{\tau\in D_{n,M}^\beta}
\frac{
{\mathcal A}_n^\beta(\tau:\hat\beta_1,\hat\beta_2|\hat\alpha)
}{h([n\tau]-[n\tau_*^\beta])}
\le2\delta
\right)\\
&\quad+P\left(
\sup_{\tau\in D_{n,M}^\beta}
\frac{
|{\varrho}_n^\beta(\tau:\hat\beta_1,\hat\beta_2|\hat\alpha)|
}{h([n\tau]-[n\tau_*^\beta])}
\ge\delta
\right)\\
&=:
P_{1,n}^\beta+P_{2,n}^\beta+P_{3,n}^\beta.
\end{align*}

[i] Evaluation of $P_{1,n}^\beta$.
For large $n$, we have
\begin{align*}
P_{1,n}^\beta
&\le
P\left(
\sup_{\tau\in D_{n,M}^\beta}
\frac{|{\mathcal M}_n^\beta(\tau:\hat\beta_1,\hat\beta_2|\hat\alpha)|}
{h([n\tau]-[n\tau_*^\beta])}
\ge\delta,\ 
\hat\alpha\in\mathcal O_{\alpha^*}.\ 
\hat\beta_1\in\mathcal O_{\beta_1^*},\ 
\hat\beta_2\in\mathcal O_{\beta_2^*} 
\right)\\
&\qquad
+P(\hat\beta_1\in\mathcal O_{\beta_1^*})+P(\hat\beta_2\in\mathcal O_{\beta_2^*})
+P(\hat\alpha\in\mathcal O_{\alpha^*})\\
&\le
P\left(
\sup_{\tau\in D_{n,M}^\beta}
\frac{
\sup_{\alpha\in\mathcal O_{\alpha^*},
\beta_k\in\mathcal O_{\beta_k^*}}
|{\mathcal M}_n^\beta(\tau:\beta_1,\beta_2|\alpha)|
}
{[n\tau]-[n\tau_*^\beta]}
\ge\delta h
\right)\\
&\qquad
+P(\hat\beta_1\in\mathcal O_{\beta_1^*})+P(\hat\beta_2\in\mathcal O_{\beta_2^*})
+P(\hat\alpha\in\mathcal O_{\alpha^*})
\end{align*}
By the uniform version on the H\'ajek-Renyi inequality in Lemma 2 
of Iacus and Yoshida (2012), 
we obtain
\begin{align*}
&P\left(
\sup_{\tau\in D_{n,M}^\beta}
\frac{1}
{[n\tau]-[n\tau_*^\beta]}
\sup_{\substack{\alpha\in\mathcal O_{\alpha^*},\\
\beta_k\in\mathcal O_{\beta_k^*}
}}
|{\mathcal M}_n^\beta(\tau:\beta_1,\beta_2|\alpha)|
\ge\delta h
\right)\\
&\le
P\left(
\max_{j>M/h-1}
\frac{1}{j}
\sup_{\substack{\alpha\in\mathcal O_{\alpha^*},\\
\beta_k\in\mathcal O_{\beta_k^*}
}}
\left|
\sum_{i=[n\tau_*^\beta]+1}^{[n\tau_*^\beta]+j}
\left(
G_i(\beta_1|\alpha)-G_i(\beta_2|\alpha)
-\EE_{\beta_2^*}[G_i(\beta_1|\alpha)-G_i(\beta_2|\alpha)|\GG]
\right)
\right|
\ge\delta h
\right)\\
&\le
\sum_{j>M/h-1}
\frac{Ch}{(\delta h j)^2}
\\
&\le
\frac{C'h}{(\delta h)^2}\frac{h}{M}
=\frac{C'}{\delta^2M}
=:\gamma_\beta(M).
\end{align*}
From \textbf{[C6-II]}, 
we have  
$P(\hat\alpha\notin\mathcal O_{\alpha^*})<\epsilon$ and
$P(\hat\beta_k\notin\mathcal O_{\beta_k^*})<\epsilon$ for large $n$.
Therefore 
$P_{1,n}^\beta\le \gamma_\beta(M)+3\epsilon$
for large $n$.

 
[ii] Evaluation of $P_{2,n}^\beta$. 
If $\hat\alpha\in\mathcal O_{\alpha^*}$ and $\hat\beta_k\in\mathcal O_{\beta_k^*}$,
then there exists a positive constant $c$ independent of $i$ such that
\begin{align*}
\Gamma^\beta_{i-1}(\hat\alpha,\hat\beta_1,\hat\beta_2)
&=
\Gamma^\beta_{i-1}(\alpha^*,\beta_1^*,\beta_2^*)
+
\int_0^1
\left.\partial_{(\alpha,\beta_1,\beta_2)}
\Gamma^\beta_{i-1}(\alpha,\beta_1,\beta_2)
\right|_{
\substack{
\alpha=\alpha^*+u(\hat\alpha-\alpha^*)\\
\beta_k=\beta_k^*+u(\hat\beta_k-\beta_k^*)
}
}
\dd u
\left(
\begin{array}{c}
\hat\alpha-\alpha^*\\
\hat\beta_1-\beta_1^*\\
\hat\beta_2-\beta_2^*
\end{array}
\right)\\
&\ge
\Gamma^\beta_{i-1}(\alpha^*,\beta_1^*,\beta_2^*)
-c(|\hat\alpha-\alpha^*|+|\hat\beta_1-\beta_1^*|+|\hat\beta_2-\beta_2^*|).
\end{align*}
According to \textbf{[B1-II]},
if we set
\begin{align*}
\delta=\frac{1}{4}\inf_{x}\Gamma^\beta(x,\alpha^*,\beta_1^*,\beta_2^*)>0,
\end{align*}
then for large $n$, 
\begin{align*}
P_{2,n}^\beta
&\le
P\left(
\inf_{\tau\in D_{n,M}^\beta}
\frac{
{\mathcal A}_n^\beta(\tau:\hat\beta_1,\hat\beta_2|\hat\alpha)
}{h([n\tau]-[n\tau_*^\beta])}
\le2\delta,\ 
\hat\alpha\in\mathcal O_{\alpha^*},\ 
\hat\beta_1\in\mathcal O_{\beta_1^*}
\hat\beta_2\in\mathcal O_{\beta_2^*}
\right)\\
&\qquad
+P(\hat\alpha\notin\mathcal O_{\alpha^*})
+P(\hat\beta_1\notin\mathcal O_{\beta_1^*})
+P(\hat\beta_2\notin\mathcal O_{\beta_2^*})\\
&\le
P\left(
\inf_{\tau\in D_{n,M}^\beta}
\frac{1}{[n\tau]-[n\tau_*^\beta]}
\sum_{i=[n\tau_*^\beta]+1}^{[n\tau]}
\left(
\Gamma^\beta_{i-1}(\alpha^*,\beta_1^*,\beta_2^*)
-c\bigl(|\hat\alpha-\alpha^*|+|\hat\beta_1-\beta_1^*|+|\hat\beta_2-\beta_2^*|
\bigr)
\right)
\le2\delta
\right)\\
&\qquad
+P(\hat\alpha\notin\mathcal O_{\alpha^*})
+P(\hat\beta_1\notin\mathcal O_{\beta_1^*})
+P(\hat\beta_2\notin\mathcal O_{\beta_2^*})\\
&\le
P\left(
\inf_{\tau\in D_{n,M}^\beta}
\frac{1}{[n\tau]-[n\tau_*^\beta]}
\sum_{i=[n\tau_*^\beta]+1}^{[n\tau]}
\Gamma^\beta_{i-1}(\alpha^*,\beta_1^*,\beta_2^*)
\le3\delta
\right)\\
&\qquad+
P\left(
-c\bigl(|\hat\alpha-\alpha^*|+|\hat\beta_1-\beta_1^*|+|\hat\beta_2-\beta_2^*|
\bigr)
\le-\delta
\right)
+P(\hat\alpha\notin\mathcal O_{\alpha^*})
+P(\hat\beta_1\notin\mathcal O_{\beta_1^*})
+P(\hat\beta_2\notin\mathcal O_{\beta_2^*})\\
&\le
P\left(
\inf_{x}
\Gamma^\beta(x,\alpha^*,\beta_1^*,\beta_2^*)
\le3\delta
\right)
+P\left(
|\hat\alpha-\alpha^*|+|\hat\beta_1-\beta_1^*|+|\hat\beta_2-\beta_2^*|
\ge\frac{\delta}{c}
\right)\\
&\qquad
+P(\hat\alpha\notin\mathcal O_{\alpha^*})
+P(\hat\beta_1\notin\mathcal O_{\beta_1^*})
+P(\hat\beta_2\notin\mathcal O_{\beta_2^*})\\
&\le 6\epsilon
\end{align*}
thanks to
\begin{align*}
&P\left(
|\hat\alpha-\alpha^*|+|\hat\beta_1-\beta_1^*|+|\hat\beta_2-\beta_2^*|
\ge\frac{\delta}{c}
\right)
\\
&\le
P\left(
|\hat\alpha-\alpha^*|
\ge\frac{\delta}{3c}
\right)
+P\left(
|\hat\beta_1-\beta_1^*|
\ge\frac{\delta}{3c}
\right)
+P\left(
|\hat\beta_2-\beta_2^*|
\ge\frac{\delta}{3c}
\right)\\
&\le3\epsilon
\end{align*}
from \textbf{[C6-II]}.

[iii] Evaluation of $P_{3,n}^\beta$. 
We have, for large $n$,
\begin{align*}
&\tr\Bigl(
A_{i-1}^{-1}(\hat\alpha)
(b_{i-1}(\hat\beta_2)-h\EE_{\beta_2^*}[\DeX|\GG])
(b_{i-1}(\hat\beta_1)-b_{i-1}(\hat\beta_2))^\TT
\Bigr)\\
&=
h\tr\Bigl(
A_{i-1}^{-1}(\hat\alpha)
(b_{i-1}(\hat\beta_2)-b_{i-1}(\beta_2^*))
(b_{i-1}(\hat\beta_1)-b_{i-1}(\hat\beta_2))^\TT
\Bigr)
+\Rd\\
&=
h\int_0^1
\left.
\partial_{\beta}\tr\Bigl(
A_{i-1}^{-1}(\hat\alpha)
(b_{i-1}(\beta)-b_{i-1}(\beta_2^*))
(b_{i-1}(\hat\beta_1)-b_{i-1}(\hat\beta_2))^\TT
\Bigr)
\right|_{\beta=\beta_2^*+u(\hat\beta_2-\beta_2^*)}
\dd u
(\hat\beta_2-\beta_2^*)
+\Rd
\end{align*}
and
\begin{align*}
&\sup_{\tau\in D_{n,M}^\beta}
\frac{|\varrho_n^\beta(\tau:\hat\beta_1,\hat\beta_2)|}
{h([n\tau]-[n\tau_*^\beta])}\\
&\le
\frac{1}{\sqrt{T}}
\sup_{\tau\in D_{n,M}^\beta}
\Biggl|
\frac{1}{[n\tau]-[n\tau_*^\beta]}
\sum_{i=[n\tau_*^\beta]+1}^{[n\tau]}
\\
&\qquad\qquad
\int_0^1
\left.
\partial_{\beta}\tr\Bigl(
A_{i-1}^{-1}(\hat\alpha)
(b_{i-1}(\beta)-b_{i-1}(\beta_2^*))
(b_{i-1}(\hat\beta_1)-b_{i-1}(\hat\beta_2))^\TT
\Bigr)
\right|_{\beta=\beta_2^*+u(\hat\beta_2-\beta_2^*)}
\dd u
\Biggr|
\sqrt{T}|\hat\beta_2-\beta_2^*|\\
&\qquad
+\sup_{\tau\in D_{n,M}^\beta}
\frac{h^2}{h([n\tau]-[n\tau_*^\beta])}
\sum_{i=[n\tau_*^\beta]+1}^{[n\tau]}
\Ro
\\
&\le
\frac{1}{\sqrt{T}}
\sup_{x,\alpha,\beta_k}
\left|
\left[
\tr\Bigl(
A^{-1}(x,\alpha)
\partial_{\beta^\ell} b(x,\beta_1)
\bigl(b(x,\beta_2)-b(x,\beta_3)\bigr)^\TT
\Bigr)
\right]_\ell
\right|
\sqrt{T}|\hat\beta_2-\beta_2^*|\\
&\qquad
+\frac{h^2}{M}
\sum_{i=[n\tau_*^\beta]+1}^{n}
\Ro
\\
&=o_p(1).
\end{align*}
Hence, we see 
\begin{align*}
P_{3,n}^\beta
&\le
P\left(
\sup_{\tau\in D_{n,M}^\beta}
\frac{|\varrho_n^\beta(\tau:\hat\beta_1,\hat\beta_2)|}
{h([n\tau]-[n\tau_*^\beta])}
\ge\delta,\ 
\hat\alpha\in\mathcal O_{\alpha^*},\ 
\hat\beta_1\in\mathcal O_{\beta_1^*},\ 
\hat\beta_2\in\mathcal O_{\beta_2^*}
\right)
\\
&\qquad+
P(\hat\alpha\notin\mathcal O_{\alpha^*})
+P(\hat\beta_1\notin\mathcal O_{\beta_1^*})
+P(\hat\beta_2\notin\mathcal O_{\beta_2^*})
\\
&\le4\epsilon
\end{align*}
for large $n$.


[iv] From the evaluations in Steps [i]-[iii], we have
\begin{align*}
\varlimsup_{n\to\infty}
P(T(\hat\tau_n^\beta-\tau_*^\beta)>M)
\le
\gamma_\beta(M)+13\epsilon
\end{align*}
for any $M\ge1$ and $\epsilon>0$. 
Therefore
\begin{align*}
\varlimsup_{M\to\infty}
\varlimsup_{n\to\infty}
P(T(\hat\tau_n^\beta-\tau_*^\beta)>M)
\le
13\epsilon.
\end{align*}
\end{proof}

\begin{proof}[\bf{Proof of Proposition \ref{prop3}}]
We see
\begin{align}
{\mathcal T}_n^\alpha
=\frac{1}{\sqrt{2dn}}\max_{1\le k\le n}
\left|
\sum_{i=1}^k\hat\eta_i-\frac{k}{n}\sum_{i=1}^n\hat\eta_i
\right|
&\ge
\frac{1}{\sqrt{2dn}}
\left|
\sum_{i=1}^{[n\tau^\alpha_*]}\hat\eta_i
-\frac{[n\tau^\alpha_*]}{n}\sum_{i=1}^n\hat\eta_i
\right|
\nonumber
\\
&=
\sqrt{\frac{n\vartheta_\alpha^2}{2d}}
\left|
\frac{1}{n\vartheta_\alpha}\sum_{i=1}^{[n\tau^\alpha_*]}\hat\eta_i
-\frac{[n\tau^\alpha_*]}{n}\frac{1}{n\vartheta_\alpha}\sum_{i=1}^n\hat\eta_i
\right|.
\label{eq7.690}
\end{align}
Now we can express
\begin{align*}
\hat\eta_i
&=\tr\left(
A_{i-1}^{-1}(\hat\alpha)\frac{(\DeX)^{\otimes2}}{h}
\right)\\
&=\tr\left(
A_{i-1}^{-1}(\alpha_0)\frac{(\DeX)^{\otimes2}}{h}
\right)
+\left.\partial_\alpha
\tr\left(
A_{i-1}^{-1}(\alpha)\frac{(\DeX)^{\otimes2}}{h}
\right)\right|_{\alpha=\alpha_0}
(\hat\alpha-\alpha_0)\\
&\quad+
\int_0^1\left.(1-u)
\partial_\alpha^2
\tr\left(
A_{i-1}^{-1}(\alpha)\frac{(\DeX)^{\otimes2}}{h}
\right)\right|_{\alpha=\alpha_0+u(\hat\alpha-\alpha_0)}\dd u
\otimes(\hat\alpha-\alpha_0)^{\otimes2}\\
&=
\tr\left(
A_{i-1}^{-1}(\alpha_0)\frac{(\DeX)^{\otimes2}}{h}
\right)
+\left(\tr\left(
A^{-1}\partial_{\alpha^\ell}AA^{-1}_{i-1}(\alpha_0)\frac{(\DeX)^{\otimes2}}{h}
\right)\right)_\ell
(\hat\alpha-\alpha_0)\\
&\quad+
\int_0^1(1-u)
\left.\partial_\alpha^2
\tr\left(
A_{i-1}^{-1}(\alpha)\frac{(\DeX)^{\otimes2}}{h}
\right)\right|_{\alpha=\alpha_0+u(\hat\alpha-\alpha_0)}\dd u
\otimes(\hat\alpha-\alpha_0)^{\otimes2}\\
&=:\eta_{1,i}
+\eta_{2,i}(\hat\alpha-\alpha_0)
%
+
\int_0^1(1-u)
\left.\partial_\alpha^2
\tr\left(
A_{i-1}^{-1}(\alpha)\frac{(\DeX)^{\otimes2}}{h}
\right)\right|_{\alpha=\alpha_0+u(\hat\alpha-\alpha_0)}\dd u
\otimes(\hat\alpha-\alpha_0)^{\otimes2}.
\end{align*}
Therefore we have, from \textbf{[E1]}, 
\begin{align}
&\frac{1}{n\vartheta_\alpha}\sum_{i=1}^{[n\tau^\alpha_*]}\hat\eta_i
-\frac{[n\tau^\alpha_*]}{n}\frac{1}{n\vartheta_\alpha}\sum_{i=1}^n\hat\eta_i
\nonumber
\\
&=
\frac{1}{n\vartheta_\alpha}\sum_{i=1}^{[n\tau^\alpha_*]}\eta_{1,i}
-\frac{[n\tau^\alpha_*]}{n}\frac{1}{n\vartheta_\alpha}\sum_{i=1}^n\eta_{1,i}
+
\left(
\frac{1}{n}\sum_{i=1}^{[n\tau^\alpha_*]}\eta_{2,i}
-\frac{[n\tau^\alpha_*]}{n}\frac{1}{n}\sum_{i=1}^n\eta_{2,i}
\right)
\vartheta_\alpha^{-1}(\hat\alpha-\alpha_0)
+o_p(1)
\nonumber
\\
&=:
\Hi+\Hd\vartheta_\alpha^{-1}(\hat\alpha-\alpha_0)
+o_p(1).
\label{eq7.699}
\end{align}
Here $\Hi$ and $\Hd$ can be transformed as follows.
\begin{align*}
\Hi
&=
\left(1-\frac{[n\tau^\alpha_*]}{n}\right)\frac{1}{n\vartheta_\alpha}
\sum_{i=1}^{[n\tau^\alpha_*]}\eta_{1,i}
-\frac{[n\tau^\alpha_*]}{n}\frac{1}{n\vartheta_\alpha}
\sum_{i=[n\tau^\alpha_*]+1}^n\eta_{1,i},
\\
\Hd
&=
\left(1-\frac{[n\tau^\alpha_*]}{n}\right)\frac{1}{n}\sum_{i=1}^{[n\tau^\alpha_*]}\eta_{2,i}
-\frac{[n\tau^\alpha_*]}{n}\frac{1}{n}\sum_{i=[n\tau^\alpha_*]+1}^n\eta_{2,i}.
\end{align*}
We have 
\begin{align*}
&\left(1-\frac{[n\tau^\alpha_*]}{n}\right)\frac{1}{n\vartheta_\alpha}
\sum_{i=1}^{[n\tau^\alpha_*]}
\EE_{\alpha_1^*}[\eta_{1,i}|\GG]
-\frac{[n\tau^\alpha_*]}{n}\frac{1}{n\vartheta_\alpha}
\sum_{i=[n\tau^\alpha_*]+1}^n
\EE_{\alpha_2^*}[\eta_{1,i}|\GG]\\
&=
\left(1-\frac{[n\tau^\alpha_*]}{n}\right)\frac{1}{n\vartheta_\alpha}
\sum_{i=1}^{[n\tau^\alpha_*]}
\tr(A_{i-1}^{-1}(\alpha_0)A_{i-1}(\alpha_1^*))\\
&\qquad
-\frac{[n\tau^\alpha_*]}{n}\frac{1}{n\vartheta_\alpha}\sum_{i=[n\tau^\alpha_*]+1}^n
\tr(A_{i-1}^{-1}(\alpha_0)A_{i-1}(\alpha_2^*))
+O_p(h/\vartheta_\alpha)\\
&=
\left(1-\frac{[n\tau^\alpha_*]}{n}\right)\frac{1}{n\vartheta_\alpha}
\sum_{i=1}^{[n\tau^\alpha_*]}
\biggl(
d
+\left.
\partial_\alpha\tr(A_{i-1}^{-1}(\alpha_0)A_{i-1}(\alpha))
\right|_{\alpha=\alpha_0}
(\alpha_1^*-\alpha_0)
\\
&\qquad\qquad+\int_0^1(1-u)
\left.\partial_\alpha^2\tr(A_{i-1}^{-1}(\alpha_0)A_{i-1}(\alpha))
\right|_{\alpha=\alpha_0+u(\alpha_1^*-\alpha_0)}
\otimes(\alpha_1^*-\alpha_0)^{\otimes2}
\biggr)\\
&\quad-
\frac{[n\tau^\alpha_*]}{n}\frac{1}{n\vartheta_\alpha}\sum_{i=[n\tau^\alpha_*]+1}^n
\biggl(
d
+\left.
\partial_\alpha\tr(A_{i-1}^{-1}(\alpha_0)A_{i-1}(\alpha))
\right|_{\alpha=\alpha_0}
(\alpha_2^*-\alpha_0)
\\
&\qquad\qquad+\int_0^1(1-u)
\left.\partial_\alpha^2\tr(A_{i-1}^{-1}(\alpha_0)A_{i-1}(\alpha))
\right|_{\alpha=\alpha_0+u(\alpha_2^*-\alpha_0)}
\otimes(\alpha_2^*-\alpha_0)^{\otimes2}
\biggr)
+o_p(1)\\
&=
\left(1-\frac{[n\tau^\alpha_*]}{n}\right)\frac{1}{n}
\sum_{i=1}^{[n\tau^\alpha_*]}
\Bigl(
\tr\bigl(
A_{i-1}^{-1}(\alpha_0)\partial_{\alpha^\ell} A_{i-1}(\alpha_0)
\bigr)
\Bigr)_\ell
\vartheta_\alpha^{-1}(\alpha_1^*-\alpha_0)
\\
&\quad-
\frac{[n\tau^\alpha_*]}{n}\frac{1}{n}\sum_{i=[n\tau^\alpha_*]+1}^n
\Bigl(
\tr\bigl(
A_{i-1}^{-1}(\alpha_0)\partial_{\alpha^\ell} A_{i-1}(\alpha_0)
\bigr)
\Bigr)_\ell
\vartheta_\alpha^{-1}(\alpha_2^*-\alpha_0)
+o_p(1)\\
&\pto
\tau^\alpha_*(1-\tau^\alpha_*)
\int_{\mathbb R^d}
\Bigl(
\tr\bigl(
A^{-1}\partial_{\alpha^\ell} A(x,\alpha_0)
\bigr)
\Bigr)_\ell
\dd\mu_{\alpha_0}(x)(c_1-c_2)
\end{align*}
and
\begin{align*}
&\left(1-\frac{[n\tau^\alpha_*]}{n}\right)^2\frac{1}{(n\vartheta_\alpha)^2}
\sum_{i=1}^{[n\tau^\alpha_*]}
\EE_{\alpha_1^*}[\eta_{1,i}^2|\GG]
-\left(\frac{[n\tau^\alpha_*]}{n}\right)^2\frac{1}{(n\vartheta_\alpha)^2}
\sum_{i=[n\tau^\alpha_*]+1}^n
\EE_{\alpha_2^*}[\eta_{1,i}^2|\GG]\\
&=
\frac{1}{n\vartheta_\alpha^2}
\left(
\left(1-\frac{[n\tau^\alpha_*]}{n}\right)^2\frac{1}{n}
\sum_{i=1}^{[n\tau^\alpha_*]}
\EE_{\alpha_1^*}[\eta_{1,i}^2|\GG]
-\left(\frac{[n\tau^\alpha_*]}{n}\right)^2\frac{1}{n}
\sum_{i=[n\tau^\alpha_*]+1}^n
\EE_{\alpha_2^*}[\eta_{1,i}^2|\GG]
\right)\\
&\pto 0.
\end{align*}
Therefore, from Lemma 9 of Genon-Catalot and Jacod (1993), we obtain 
\begin{align}
\Hi
\pto
\tau^\alpha_*(1-\tau^\alpha_*)
\int_{\mathbb R^d}
\Bigl(
\tr\bigl(
A^{-1}\partial_{\alpha^\ell} A(x,\alpha_0)
\bigr)
\Bigr)_\ell
\dd\mu_{\alpha_0}(x)(c_1-c_2).
\label{eq7.698}
\end{align}
Further, we have
\begin{align*}
&\left(1-\frac{[n\tau^\alpha_*]}{n}\right)\frac{1}{n}
\sum_{i=1}^{[n\tau^\alpha_*]}
\EE_{\alpha_1^*}[\eta_{2,i}|\GG]
-\frac{[n\tau^\alpha_*]}{n}\frac{1}{n}\sum_{i=[n\tau^\alpha_*]+1}^n
\EE_{\alpha_2^*}[\eta_{2,i}|\GG]\\
&=
\left(1-\frac{[n\tau^\alpha_*]}{n}\right)\frac{1}{n}
\sum_{i=1}^{[n\tau^\alpha_*]}
\left[
\tr\Bigl(
A^{-1}\partial_{\alpha^\ell}A A_{i-1}^{-1}(\alpha_0)A_{i-1}(\alpha_1^*)
\Bigr)
\right]_\ell\\
&\quad-
\frac{[n\tau^\alpha_*]}{n}\frac{1}{n}\sum_{i=[n\tau^\alpha_*]+1}^n
\left[
\tr\Bigl(
A^{-1}\partial_{\alpha^\ell}A A_{i-1}^{-1}(\alpha_0)A_{i-1}(\alpha_2^*)
\Bigr)
\right]_\ell
+O_p(h)\\
&=
\left(1-\frac{[n\tau^\alpha_*]}{n}\right)\frac{1}{n}
\sum_{i=1}^{[n\tau^\alpha_*]}
\biggl(
\left[
\tr\Bigl(
A^{-1}\partial_{\alpha^\ell} A_{i-1}(\alpha_0)
\Bigr)
\right]_\ell
\\
&\qquad\qquad+\int_0^1
\left.
\partial_{\alpha}
\left[
\tr\Bigl(
A^{-1}\partial_{\alpha^\ell}A A_{i-1}^{-1}(\alpha_0)A_{i-1}(\alpha)
\Bigr)
\right]_\ell
\right|_{\alpha=\alpha_0+u(\alpha_1^*-\alpha_0)}
(\alpha_1^*-\alpha_0)
\biggr)\\
&\quad-
\frac{[n\tau^\alpha_*]}{n}\frac{1}{n}\sum_{i=[n\tau^\alpha_*]+1}^n
\biggl(
\left[
\tr\Bigl(
A^{-1}\partial_{\alpha^\ell} A_{i-1}(\alpha_0)
\Bigr)
\right]_\ell
\\
&\qquad\qquad+\int_0^1
\left.
\partial_{\alpha}
\left[
\tr\Bigl(
A^{-1}\partial_{\alpha^\ell}A A_{i-1}^{-1}(\alpha_0)A_{i-1}(\alpha)
\Bigr)
\right]_\ell
\right|_{\alpha=\alpha_0+u(\alpha_2^*-\alpha_0)}
(\alpha_2^*-\alpha_0)
\biggr)
+o_p(1)\\
&=
\left(1-\frac{[n\tau^\alpha_*]}{n}\right)\frac{1}{n}
\sum_{i=1}^{[n\tau^\alpha_*]}
\left[
\tr\Bigl(
A^{-1}\partial_{\alpha^\ell} A_{i-1}(\alpha_0)
\Bigr)
\right]_\ell
-\frac{[n\tau^\alpha_*]}{n}\frac{1}{n}\sum_{i=[n\tau^\alpha_*]+1}^n
\left[
\tr\Bigl(
A^{-1}\partial_{\alpha^\ell} A_{i-1}(\alpha_0)
\Bigr)
\right]_\ell
+o_p(1)\\
&\pto0
\end{align*}
and
\begin{align*}
&\left(1-\frac{[n\tau^\alpha_*]}{n}\right)^2\frac{1}{n^2}
\sum_{i=1}^{[n\tau^\alpha_*]}
\EE_{\alpha_1^*}[\eta_{2,i}^2|\GG]
-\left(\frac{[n\tau^\alpha_*]}{n}\right)^2\frac{1}{n^2}
\sum_{i=[n\tau^\alpha_*]+1}^n
\EE_{\alpha_2^*}[\eta_{2,i}^2|\GG]\\
&=
\frac{1}{n}
\left(
\left(1-\frac{[n\tau^\alpha_*]}{n}\right)^2\frac{1}{n}
\sum_{i=1}^{[n\tau^\alpha_*]}
\EE_{\alpha_1^*}[\eta_{2,i}^2|\GG]
-\left(\frac{[n\tau^\alpha_*]}{n}\right)^2\frac{1}{n}
\sum_{i=[n\tau^\alpha_*]+1}^n
\EE_{\alpha_2^*}[\eta_{2,i}^2|\GG]
\right)\\
&\pto 0.
\end{align*}
Therefore, from Lemma 9 of Genon-Catalot and Jacod (1993), we see
$\Hd\pto0$.
Hence, from \eqref{eq7.699}, \eqref{eq7.698} and this, we obtain
\begin{align*}
\frac{1}{n\vartheta_\alpha}\sum_{i=1}^{[n\tau^\alpha_*]}\hat\eta_i
-\frac{[n\tau^\alpha_*]}{n}\frac{1}{n\vartheta_\alpha}\sum_{i=1}^n\hat\eta_i
\pto
\tau^\alpha_*(1-\tau^\alpha_*)
\int_{\mathbb R^d}
\left[
\tr\Bigl(
A^{-1}\partial_{\alpha^\ell} A(x,\alpha_0)
\Bigr)
\right]_\ell
\dd\mu_{\alpha_0}(x)(c_1-c_2)
\neq0,
\end{align*}
which implies $\mathcal T_n^\alpha\lto\infty$ 
from \eqref{eq7.690} and $n\Dea^2\lto\infty$,
that is,  
$P(\mathcal T_n^\alpha>w_1(\epsilon))\lto1$.
\end{proof}


\begin{proof}[\bf{Proof of Proposition \ref{prop4}}]
We see
\begin{align}
{\mathcal T}_{1,n}^\beta
=\frac{1}{\sqrt{dT}}\max_{1\le k\le n}
\left|
\sum_{i=1}^k\hat\xi_i-\frac{k}{n}\sum_{i=1}^n\hat\xi_i
\right|
&\ge
\frac{1}{\sqrt{dT}}
\left|
\sum_{i=1}^{[n\tau^\beta_*]}\hat\xi_i-\frac{[n\tau^\beta_*]}{n}\sum_{i=1}^n\hat\xi_i
\right|
\nonumber
\\
&=
\sqrt{\frac{T\vartheta_\beta^2}{d}}
\left|
\frac{1}{T\vartheta_\beta}\sum_{i=1}^{[n\tau^\beta_*]}\hat\xi_i
-\frac{[n\tau^\beta_*]}{n}\frac{1}{T\vartheta_\beta}\sum_{i=1}^n\hat\xi_i
\right|.
\label{eq7.790}
\end{align}
Now we can express
\begin{align*}
\hat\xi_i
&=1_d^\TT a_{i-1}^{-1}(\hat\alpha)(\DeX-hb_{i-1}(\hat\beta))\\
&=1_d^\TT a_{i-1}^{-1}(\alpha^*)(\DeX-hb_{i-1}(\hat\beta))
+\int_0^1\left.\partial_\alpha
\left(
1_d^\TT a_{i-1}^{-1}(\alpha)(\DeX-hb_{i-1}(\hat\beta))
\right)
\right|_{\alpha=\alpha^*+u(\hat\alpha-\alpha^*)}
\dd u
(\hat\alpha-\alpha^*)\\
&=1_d^\TT a_{i-1}^{-1}(\alpha^*)(\DeX-hb_{i-1}(\bar\beta^*))
-h1_d^\TT a_{i-1}^{-1}(\alpha^*)(b_{i-1}(\hat\beta)-b_{i-1}(\bar\beta^*))\\
&\quad+\int_0^1\left.\partial_\alpha
\left(
1_d^\TT a_{i-1}^{-1}(\alpha)(\DeX-hb_{i-1}(\hat\beta))
\right)
\right|_{\alpha=\alpha^*+u(\hat\alpha-\alpha^*)}
\dd u
(\hat\alpha-\alpha^*)\\
&=:\xi_i
-h\int_0^11_d^\TT a_{i-1}^{-1}(\alpha^*)
\partial_{\beta} b_{i-1}(\bar\beta^*+u(\hat\beta-\bar\beta^*))\dd u
(\hat\beta-\bar\beta^*)\\
&\quad+\int_0^1\left.\partial_\alpha
\left(
1_d^\TT a_{i-1}^{-1}(\alpha)(\DeX-hb_{i-1}(\hat\beta))
\right)
\right|_{\alpha=\alpha^*+u(\hat\alpha-\alpha^*)}
\dd u
(\hat\alpha-\alpha^*)
\end{align*}
Therefore, we have
\begin{align}
\frac{1}{T\Deb}\sum_{i=1}^{[n\tau^\beta_*]}\hat\xi_i
-\frac{[n\tau^\beta_*]}{n}\frac{1}{T\Deb}\sum_{i=1}^n\hat\xi_i
&=
\frac{1}{T\Deb}\sum_{i=1}^{[n\tau^\beta_*]}\xi_i
-\frac{[n\tau^\beta_*]}{n}\frac{1}{T\Deb}\sum_{i=1}^n\xi_i
+o_p(1)
\nonumber
\\
&=
\left(1-\frac{[n\tau^\beta_*]}{n}\right)\frac{1}{T\vartheta_\beta}
\sum_{i=1}^{[n\tau^\beta_*]}\xi_i
-\frac{[n\tau^\beta_*]}{n}\frac{1}{T\vartheta_\beta}\sum_{i=[n\tau^\beta_*]+1}^n\xi_i
+o_p(1),\label{eq7.799}
\end{align}
\begin{align*}
&\left(1-\frac{[n\tau^\beta_*]}{n}\right)\frac{1}{T\vartheta_\beta}
\sum_{i=1}^{[n\tau^\beta_*]}
\EE_{\beta_1^*}[\xi_i|\GG]
-\frac{[n\tau^\beta_*]}{n}\frac{1}{T\vartheta_\beta}
\sum_{i=[n\tau^\beta_*]+1}^n
\EE_{\beta_2^*}[\xi_i|\GG]\\
&=
\left(1-\frac{[n\tau^\beta_*]}{n}\right)\frac{1}{T\vartheta_\beta}
\sum_{i=1}^{[n\tau^\beta_*]}
h1_d^\TT a_{i-1}^{-1}(\alpha^*)(b_{i-1}(\beta_1^*)-b_{i-1}(\bar\beta^*))\\
&\qquad-\frac{[n\tau^\beta_*]}{n}\frac{1}{T\vartheta_\beta}
\sum_{i=[n\tau^\beta_*]+1}^n
h1_d^\TT a_{i-1}^{-1}(\alpha^*)(b_{i-1}(\beta_2^*)-b_{i-1}(\bar\beta^*))
+O_p(h/\vartheta_\beta)\\
&=
\left(1-\frac{[n\tau^\beta_*]}{n}\right)\frac{1}{n\vartheta_\beta}
\sum_{i=1}^{[n\tau^\beta_*]}
1_d^\TT a_{i-1}^{-1}(\alpha^*)(b_{i-1}(\beta_0)-b_{i-1}(\bar\beta^*))\\
&\qquad-\frac{[n\tau^\beta_*]}{n}\frac{1}{n\vartheta_\beta}
\sum_{i=[n\tau^\beta_*]+1}^n
1_d^\TT a_{i-1}^{-1}(\alpha^*)(b_{i-1}(\beta_0)-b_{i-1}(\bar\beta^*))\\
&\qquad+
\left(1-\frac{[n\tau^\beta_*]}{n}\right)\frac{1}{n\vartheta_\beta}
\sum_{i=1}^{[n\tau^\beta_*]}
1_d^\TT a_{i-1}^{-1}(\alpha^*)\partial_\beta b_{i-1}(\beta_0)(\beta_1^*-\beta_0)
\\
&\qquad
-\frac{[n\tau^\beta_*]}{n}\frac{1}{n\vartheta_\beta}
\sum_{i=[n\tau^\beta_*]+1}^n
1_d^\TT a_{i-1}^{-1}(\alpha^*)\partial_\beta b_{i-1}(\beta_0)(\beta_2^*-\beta_0)\\
&\qquad+
\left(1-\frac{[n\tau^\beta_*]}{n}\right)\frac{1}{n\vartheta_\beta}
\sum_{i=1}^{[n\tau^\beta_*]}
1_d^\TT a_{i-1}^{-1}(\alpha^*)
\int_0^1(1-u)\partial_\beta^2 b_{i-1}(\beta_0+u(\beta_1^*-\beta_0))\dd u
\otimes(\beta_1^*-\beta_0)^{\otimes2}\\
&\qquad-
\frac{[n\tau^\beta_*]}{n}\frac{1}{n\vartheta_\beta}
\sum_{i=[n\tau^\beta_*]+1}^n
1_d^\TT a_{i-1}^{-1}(\alpha^*)
\int_0^1(1-u)\partial_\beta^2 b_{i-1}(\beta_0+u(\beta_2^*-\beta_0))\dd u
\otimes(\beta_2^*-\beta_0)^{\otimes2}
+o_p(1)\\
&=
-\left(1-\frac{[n\tau^\beta_*]}{n}\right)\frac{1}{n\vartheta_\beta}
\sum_{i=1}^{[n\tau^\beta_*]}
1_d^\TT a_{i-1}^{-1}(\alpha^*)
\biggl(
\partial_\beta b_{i-1}(\beta_0)(\bar\beta^*-\beta_0)
\\
&\qquad\qquad
+\int_0^1(1-u)\partial_\beta^2 b_{i-1}(\beta_0+u(\bar\beta^*-\beta_0))\dd u
\otimes(\bar\beta^*-\beta_0)^{\otimes2}
\biggr)
\\
&\qquad+\frac{[n\tau^\beta_*]}{n}\frac{1}{n\vartheta_\beta}
\sum_{i=[n\tau^\beta_*]+1}^n
1_d^\TT a_{i-1}^{-1}(\alpha^*)
\biggl(
\partial_\beta b_{i-1}(\beta_0)(\bar\beta^*-\beta_0)
\\
&\qquad\qquad
+\int_0^1(1-u)\partial_\beta^2 b_{i-1}(\beta_0+u(\bar\beta^*-\beta_0))\dd u
\otimes(\bar\beta^*-\beta_0)^{\otimes2}
\biggr)
\\
&\qquad+
\left(1-\frac{[n\tau^\beta_*]}{n}\right)\frac{1}{n}
\sum_{i=1}^{[n\tau^\beta_*]}
1_d^\TT a_{i-1}^{-1}(\alpha^*)\partial_\beta b_{i-1}(\beta_0)
\vartheta_\beta^{-1}(\beta_1^*-\beta_0)\\
&\qquad-\frac{[n\tau^\beta_*]}{n}\frac{1}{n}
\sum_{i=[n\tau^\beta_*]+1}^n
1_d^\TT a_{i-1}^{-1}(\alpha^*)\partial_\beta b_{i-1}(\beta_0)
\vartheta_\beta^{-1}(\beta_2^*-\beta_0)
+o_p(1)\\
&=
-\left(1-\frac{[n\tau^\beta_*]}{n}\right)\frac{1}{n}
\sum_{i=1}^{[n\tau^\beta_*]}
1_d^\TT a_{i-1}^{-1}(\alpha^*)
\partial_\beta b_{i-1}(\beta_0)\vartheta_\beta^{-1}(\bar\beta^*-\beta_0)
\\
&\qquad+\frac{[n\tau^\beta_*]}{n}\frac{1}{n}
\sum_{i=[n\tau^\beta_*]+1}^n
1_d^\TT a_{i-1}^{-1}(\alpha^*)
\partial_\beta b_{i-1}(\beta_0)\vartheta_\beta^{-1}(\bar\beta^*-\beta_0)
\\
&\qquad+
\left(1-\frac{[n\tau^\beta_*]}{n}\right)\frac{1}{n}
\sum_{i=1}^{[n\tau^\beta_*]}
1_d^\TT a_{i-1}^{-1}(\alpha^*)\partial_\beta b_{i-1}(\beta_0)
\vartheta_\beta^{-1}(\beta_1^*-\beta_0)\\
&\qquad-\frac{[n\tau^\beta_*]}{n}\frac{1}{n}
\sum_{i=[n\tau^\beta_*]+1}^n
1_d^\TT a_{i-1}^{-1}(\alpha^*)\partial_\beta b_{i-1}(\beta_0)
\vartheta_\beta^{-1}(\beta_2^*-\beta_0)
+o_p(1)\\
&\pto
\tau^\beta_*(1-\tau^\beta_*)
\int_{\mathbb R^d}
1_d^\TT a^{-1}(x,\alpha^*)\partial_\beta b(x,\beta_0)
\dd\mu_{(\alpha^*,\beta_0)}(x)(d_1-d_2),
\end{align*}
and
\begin{align*}
&\left(1-\frac{[n\tau^\beta_*]}{n}\right)^2\frac{1}{(T\vartheta_\beta)^2}
\sum_{i=1}^{[n\tau^\beta_*]}
\EE_{\beta_1^*}[\xi_i^2|\GG]
-\left(\frac{[n\tau^\beta_*]}{n}\right)^2\frac{1}{(T\vartheta_\beta)^2}
\sum_{i=[n\tau^*]+1}^n
\EE_{\beta_2^*}[\xi_i^2|\GG]\\
&=
\frac{1}{T\vartheta_\beta^2}
\left(
\left(1-\frac{[n\tau^\beta_*]}{n}\right)^2\frac{1}{T}
\sum_{i=1}^{[n\tau^\beta_*]}
\EE_{\beta_1^*}[\xi_i^2|\GG]
-\left(\frac{[n\tau^\beta_*]}{n}\right)^2\frac{1}{T}\sum_{i=[n\tau^*]+1}^n
\EE_{\beta_2^*}[\xi_i^2|\GG]
\right)\\
&\pto 0.
\end{align*}
Therefore, from Lemma 9 of Genon-Catalot and Jacod (1993), we see
\begin{align*}
\frac{1}{T\vartheta_\beta}\sum_{i=1}^{[n\tau^\beta_*]}\xi_i
-\frac{[n\tau^\beta_*]}{n}\frac{1}{T\vartheta_\beta}\sum_{i=1}^n\xi_i
\pto
\tau^\beta_*(1-\tau^\beta_*)
\int_{\mathbb R^d}
1_d^\TT a^{-1}(x,\alpha^*)\partial_\beta b(x,\beta_0)
\dd\mu_{\beta_0}(x)(d_1-d_2).
\end{align*}
Hence, from \eqref{eq7.799} and this, we obtain
\begin{align*}
\frac{1}{T\vartheta_\beta}\sum_{i=1}^{[n\tau^\beta_*]}\hat\xi_i
-\frac{[n\tau^\beta_*]}{n}\frac{1}{T\vartheta_\beta}\sum_{i=1}^n\hat\xi_i
\pto
\tau^\beta_*(1-\tau^\beta_*)
\int_{\mathbb R^d}
1_d^\TT a^{-1}(x,\alpha^*)\partial_\beta b(x,\beta_0)
\dd\mu_{\beta_0}(x)(d_1-d_2)
\neq0,
\end{align*}
which implies $\mathcal T_{1,n}^\beta\lto\infty$ 
from \eqref{eq7.790} and $T\Deb^2\lto\infty$,
that is,  
$P(\mathcal T_{1,n}^\beta>w_1(\epsilon))\lto1$.
\end{proof}


\begin{proof}[\bf{Proof of Proposition \ref{prop5}}]
We see
\begin{align}
{\mathcal T}_{2,n}^\beta
&=\frac{1}{\sqrt{T}}\max_{1\le k\le n}
\left\|
\mathcal I_n^{-1/2}
\left(
\sum_{i=1}^k\hat\zeta_i-\frac{k}{n}\sum_{i=1}^n\hat\zeta_i
\right)
\right\|
\nonumber
\\
&\ge
\frac{1}{\sqrt{T}}
\left\|
\mathcal I_n^{-1/2}
\left(
\sum_{i=1}^{[n\tau^\beta_*]}\hat\zeta_i
-\frac{[n\tau^\beta_*]}{n}\sum_{i=1}^n\hat\zeta_i
\right)
\right\|
\nonumber
\\
&=
\sqrt{T\vartheta_\beta^2}
\left\|
\mathcal I_n^{-1/2}
\left(
\frac{1}{T\vartheta_\beta}\sum_{i=1}^{[n\tau^\beta_*]}\hat\zeta_i
-\frac{[n\tau^\beta_*]}{n}\frac{1}{T\vartheta_\beta}\sum_{i=1}^n\hat\zeta_i
\right)
\right\|.
\label{eq7.890}
\end{align}
Now we can express
\begin{align*}
\hat\zeta_i
&=\partial_{\beta}b_{i-1}(\hat\beta)^\TT A_{i-1}^{-1}(\hat\alpha)
(\DeX-hb_{i-1}(\hat\beta))\\
&=\partial_{\beta}b_{i-1}(\hat\beta)^\TT A_{i-1}^{-1}(\alpha^*)
(\DeX-hb_{i-1}(\hat\beta))\\
&\qquad+\int_0^1\left.\partial_\alpha
\left(
\partial_{\beta}b_{i-1}(\hat\beta)^\TT A_{i-1}^{-1}(\alpha)
(\DeX-hb_{i-1}(\hat\beta))\right)
\right|_{\alpha=\alpha^*+u(\hat\alpha-\alpha^*)}
\dd u
(\hat\alpha-\alpha^*)\\
&=\partial_{\beta}b_{i-1}(\hat\beta)^\TT A_{i-1}^{-1}(\alpha^*)
(\DeX-hb_{i-1}(\bar\beta^*))
-h\partial_{\beta}b_{i-1}(\hat\beta)^\TT A_{i-1}^{-1}(\alpha^*)
(b_{i-1}(\hat\beta)-b_{i-1}(\bar\beta^*))\\
&\qquad+\int_0^1\left.\partial_\alpha
\left(
\partial_{\beta}b_{i-1}(\hat\beta)^\TT A_{i-1}^{-1}(\alpha)(\DeX-hb_{i-1}(\hat\beta))
\right)
\right|_{\alpha=\alpha^*+u(\hat\alpha-\alpha^*)}
\dd u
(\hat\alpha-\alpha^*)\\
&=
\partial_{\beta}b_{i-1}(\beta^*)^\TT A_{i-1}^{-1}(\alpha^*)
(\DeX-hb_{i-1}(\bar\beta^*))
\\
&\qquad+\sum_{j=1}^{M-1}\partial_{\beta}^j
\left.
\Bigl(
\partial_{\beta}b_{i-1}(\beta)^\TT A_{i-1}^{-1}(\alpha^*)
(\DeX-hb_{i-1}(\bar\beta^*))
\Bigr)
\right|_{\beta=\bar\beta^*}
\otimes(\hat\beta-\bar\beta^*)^{\otimes j}
\\
&\qquad+
\int_0^1
\partial_{\beta}^M
\left.
\Bigl(
\partial_{\beta}b_{i-1}(\beta)^\TT A_{i-1}^{-1}(\alpha^*)
(\DeX-hb_{i-1}(\bar\beta^*))
\Bigr)
\right|_{\beta=\bar\beta^*+u(\hat\beta-\bar\beta^*)}
\otimes(\hat\beta-\bar\beta^*)^{\otimes M}
\\
&\qquad-
h\int_0^1\partial_{\beta}b_{i-1}(\hat\beta)^\TT A_{i-1}^{-1}(\alpha^*)
\partial_{\beta} b_{i-1}(\bar\beta^*+u(\hat\beta-\bar\beta^*))\dd u
(\hat\beta-\bar\beta^*)\\
&\qquad+\int_0^1\left.\partial_\alpha
\left(
\partial_{\beta}b_{i-1}(\hat\beta)^\TT A_{i-1}^{-1}(\alpha)(\DeX-hb_{i-1}(\hat\beta))
\right)
\right|_{\alpha=\alpha^*+u(\hat\alpha-\alpha^*)}
\dd u
(\hat\alpha-\alpha^*)\\
&=:\zeta_i
+\sum_{j=1}^{M-1}\partial_{\beta}^j
\left.
\Bigl(
\partial_{\beta}b_{i-1}(\beta)^\TT A_{i-1}^{-1}(\alpha^*)
(\DeX-hb_{i-1}(\bar\beta^*))
\Bigr)
\right|_{\beta=\bar\beta^*}
\otimes(\hat\beta-\bar\beta^*)^{\otimes j}
\\
&\qquad+
\int_0^1
\partial_{\beta}^M
\left.
\Bigl(
\partial_{\beta}b_{i-1}(\beta)^\TT A_{i-1}^{-1}(\alpha^*)
(\DeX-hb_{i-1}(\bar\beta^*))
\Bigr)
\right|_{\beta=\bar\beta^*+u(\hat\beta-\bar\beta^*)}
\otimes(\hat\beta-\bar\beta^*)^{\otimes M}
\\
&\qquad-
h\int_0^1\partial_{\beta}b_{i-1}(\hat\beta)^\TT A_{i-1}^{-1}(\alpha^*)
\partial_{\beta} b_{i-1}(\bar\beta^*+u(\hat\beta-\bar\beta^*))\dd u
(\hat\beta-\bar\beta^*)\\
&\qquad+\int_0^1\left.\partial_\alpha
\left(
\partial_{\beta}b_{i-1}(\hat\beta)^\TT A_{i-1}^{-1}(\alpha)(\DeX-hb_{i-1}(\hat\beta))
\right)
\right|_{\alpha=\alpha^*+u(\hat\alpha-\alpha^*)}
\dd u
(\hat\alpha-\alpha^*).
\end{align*}
Thus, we have
\begin{align}
\frac{1}{T\Deb}\sum_{i=1}^{[n\tau^\beta_*]}\hat\zeta_i
-\frac{[n\tau^\beta_*]}{n}\frac{1}{T\Deb}\sum_{i=1}^n\hat\zeta_i
&=
\frac{1}{T\Deb}\sum_{i=1}^{[n\tau^\beta_*]}\zeta_i
-\frac{[n\tau^\beta_*]}{n}\frac{1}{T\Deb}\sum_{i=1}^n\zeta_i
+o_p(1)
\nonumber
\\
&=
\left(1-\frac{[n\tau^\beta_*]}{n}\right)\frac{1}{T\Deb}
\sum_{i=1}^{[n\tau^\beta_*]}\zeta_i
-\frac{[n\tau^\beta_*]}{n}\frac{1}{T\Deb}\sum_{i=[n\tau^\beta_*]+1}^n\zeta_i
+o_p(1),
\label{eq7.899}
\end{align}
\begin{align*}
&\left(1-\frac{[n\tau^\beta_*]}{n}\right)\frac{1}{T\Deb}
\sum_{i=1}^{[n\tau^\beta_*]}
\EE_{\beta_1^*}[\zeta_i|\GG]
-\frac{[n\tau^\beta_*]}{n}\frac{1}{T\Deb}
\sum_{i=[n\tau^\beta_*]+1}^n
\EE_{\beta_2^*}[\zeta_i|\GG]\\
&=
\left(1-\frac{[n\tau^\beta_*]}{n}\right)\frac{1}{T\Deb}
\sum_{i=1}^{[n\tau^\beta_*]}
h\partial_{\beta}b_{i-1}(\bar\beta^*)^\TT A_{i-1}^{-1}(\alpha^*)
(b_{i-1}(\beta_1^*)-b_{i-1}(\bar\beta^*))\\
&\qquad
-\frac{[n\tau^\beta_*]}{n}\frac{1}{T\vartheta_\beta}
\sum_{i=[n\tau^\beta_*]+1}^n
h\partial_{\beta}b_{i-1}(\bar\beta^*)^\TT A_{i-1}^{-1}(\alpha^*)
(b_{i-1}(\beta_2^*)-b_{i-1}(\bar\beta^*))
+O_p(h/\vartheta_\beta)\\
&=
\left(1-\frac{[n\tau^\beta_*]}{n}\right)\frac{1}{n\vartheta_\beta}
\sum_{i=1}^{[n\tau^\beta_*]}
\partial_{\beta}b_{i-1}(\bar\beta^*)^\TT A_{i-1}^{-1}(\alpha^*)
(b_{i-1}(\beta_0)-b_{i-1}(\bar\beta^*))\\
&\qquad-
\frac{[n\tau^\beta_*]}{n}\frac{1}{n\vartheta_\beta}
\sum_{i=[n\tau^\beta_*]+1}^n
\partial_{\beta}b_{i-1}(\bar\beta^*)^\TT A_{i-1}^{-1}(\alpha^*)
(b_{i-1}(\beta_0)-b_{i-1}(\bar\beta^*))\\
&\qquad+
\left(1-\frac{[n\tau^\beta_*]}{n}\right)\frac{1}{n\vartheta_\beta}
\sum_{i=1}^{[n\tau^\beta_*]}
\partial_{\beta}b_{i-1}(\bar\beta^*)^\TT A_{i-1}^{-1}(\alpha^*)
\partial_\beta b_{i-1}(\beta_0)(\beta_1^*-\beta_0)\\
&\qquad-
\frac{[n\tau^\beta_*]}{n}\frac{1}{n\vartheta_\beta}
\sum_{i=[n\tau^\beta_*]+1}^n
\partial_{\beta}b_{i-1}(\bar\beta^*)^\TT A_{i-1}^{-1}(\alpha^*)
\partial_\beta b_{i-1}(\beta_0)
(\beta_2^*-\beta_0)\\
&\qquad+
\left(1-\frac{[n\tau^\beta_*]}{n}\right)\frac{1}{n\vartheta_\beta}
\sum_{i=1}^{[n\tau^\beta_*]}
\partial_{\beta}b_{i-1}(\bar\beta^*)^\TT A_{i-1}^{-1}(\alpha^*)
\int_0^1(1-u)\partial_\beta^2 b_{i-1}(\beta_0+u(\beta_1^*-\beta_0))\dd u
\otimes(\beta_1^*-\beta_0)^{\otimes2}\\
&\qquad-
\frac{[n\tau^\beta_*]}{n}\frac{1}{n\vartheta_\beta}
\sum_{i=[n\tau^\beta_*]+1}^n
\partial_{\beta}b_{i-1}(\bar\beta^*)^\TT A_{i-1}^{-1}(\alpha^*)
\int_0^1(1-u)\partial_\beta^2 b_{i-1}(\beta_0+u(\beta_2^*-\beta_0))\dd u
\otimes(\beta_2^*-\beta_0)^{\otimes2}
+o_p(1)\\
&=
-\left(1-\frac{[n\tau^\beta_*]}{n}\right)\frac{1}{n\vartheta_\beta}
\sum_{i=1}^{[n\tau^\beta_*]}
\partial_{\beta}b_{i-1}(\bar\beta^*)^\TT A_{i-1}^{-1}(\alpha^*)
\biggl(
\partial_\beta b_{i-1}(\beta_0)(\bar\beta^*-\beta_0)
\\
&\qquad\qquad+
\int_0^1(1-u)\partial_\beta^2 b_{i-1}(\beta_0+u(\bar\beta^*-\beta_0))\dd u
\otimes(\bar\beta^*-\beta_0)^{\otimes2}
\biggr)
\\
&\quad+\frac{[n\tau^\beta_*]}{n}\frac{1}{n\vartheta_\beta}
\sum_{i=[n\tau^\beta_*]+1}^n
\partial_{\beta}b_{i-1}(\bar\beta^*)^\TT A_{i-1}^{-1}(\alpha^*)
\biggl(
\partial_\beta b_{i-1}(\beta_0)(\bar\beta^*-\beta_0)
\\
&\qquad\qquad+
\int_0^1(1-u)\partial_\beta^2 b_{i-1}(\beta_0+u(\bar\beta^*-\beta_0))\dd u
\otimes(\bar\beta^*-\beta_0)^{\otimes2}
\biggr)
\\
&\quad+
\left(1-\frac{[n\tau^\beta_*]}{n}\right)\frac{1}{n}
\sum_{i=1}^{[n\tau^\beta_*]}
\partial_{\beta}b_{i-1}(\bar\beta^*)^\TT A_{i-1}^{-1}(\alpha^*)
\partial_\beta b_{i-1}(\beta_0)
\vartheta_\beta^{-1}(\beta_1^*-\beta_0)\\
&\quad-\frac{[n\tau^\beta_*]}{n}\frac{1}{n}
\sum_{i=[n\tau^\beta_*]+1}^n
\partial_{\beta}b_{i-1}(\bar\beta^*)^\TT A_{i-1}^{-1}(\alpha^*)
\partial_\beta b_{i-1}(\beta_0)
\vartheta_\beta^{-1}(\beta_2^*-\beta_0)
+o_p(1)\\
&=
-\left(1-\frac{[n\tau^\beta_*]}{n}\right)\frac{1}{n}
\sum_{i=1}^{[n\tau^\beta_*]}
\partial_{\beta}b_{i-1}(\bar\beta^*)^\TT A_{i-1}^{-1}(\alpha^*)
\partial_\beta b_{i-1}(\beta_0)\vartheta_\beta^{-1}(\bar\beta^*-\beta_0)
\\
&\quad+\frac{[n\tau^\beta_*]}{n}\frac{1}{n}
\sum_{i=[n\tau^\beta_*]+1}^n
\partial_{\beta}b_{i-1}(\bar\beta^*)^\TT A_{i-1}^{-1}(\alpha^*)
\partial_\beta b_{i-1}(\beta_0)\vartheta_\beta^{-1}(\bar\beta^*-\beta_0)
\\
&\quad+
\left(1-\frac{[n\tau^\beta_*]}{n}\right)\frac{1}{n}
\sum_{i=1}^{[n\tau^\beta_*]}
\partial_{\beta}b_{i-1}(\bar\beta^*)^\TT A_{i-1}^{-1}(\alpha^*)
\partial_\beta b_{i-1}(\beta_0)
\vartheta_\beta^{-1}(\beta_1^*-\beta_0)\\
&\quad-\frac{[n\tau^\beta_*]}{n}\frac{1}{n}
\sum_{i=[n\tau^\beta_*]+1}^n
\partial_{\beta}b_{i-1}(\bar\beta^*)^\TT A_{i-1}^{-1}(\alpha^*)
\partial_\beta b_{i-1}(\beta_0)
\vartheta_\beta^{-1}(\beta_2^*-\beta_0)
+o_p(1)\\
&\pto
\tau^\beta_*(1-\tau^\beta_*)
\int_{\mathbb R^d}
\partial_{\beta}b(x,\beta_0)^\TT A^{-1}(x,\alpha^*)
\partial_\beta b(x,\beta_0)
\dd\mu_{\beta_0}(x)(d_1-d_2),
\end{align*}
and
\begin{align*}
&\left(1-\frac{[n\tau^\beta_*]}{n}\right)^2\frac{1}{(T\vartheta_\beta)^2}
\sum_{i=1}^{[n\tau^\beta_*]}
\EE_{\beta_1^*}[\zeta_i^2|\GG]
-\left(\frac{[n\tau^\beta_*]}{n}\right)^2\frac{1}{(T\vartheta_\beta)^2}
\sum_{i=[n\tau^*]+1}^n
\EE_{\beta_2^*}[\zeta_i^2|\GG]\\
&=
\frac{1}{T\vartheta_\beta^2}
\left(
\left(1-\frac{[n\tau^\beta_*]}{n}\right)^2\frac{1}{T}
\sum_{i=1}^{[n\tau^\beta_*]}
\EE_{\beta_1^*}[\zeta_i^2|\GG]
-\left(\frac{[n\tau^\beta_*]}{n}\right)^2\frac{1}{T}\sum_{i=[n\tau^*]+1}^n
\EE_{\beta_2^*}[\zeta_i^2|\GG]
\right)\\
&\pto 0.
\end{align*}
Therefore, from Lemma 9 of Genon-Catalot and Jacod (1993), we see
\begin{align*}
\frac{1}{T\vartheta_\beta}\sum_{i=1}^{[n\tau^\beta_*]}\zeta_i
-\frac{[n\tau^\beta_*]}{n}\frac{1}{T\vartheta_\beta}\sum_{i=1}^n\zeta_i
\pto
\tau^\beta_*(1-\tau^\beta_*)
\int_{\mathbb R^d}
\partial_{\beta}b(x,\beta_0)^\TT A^{-1}(x,\alpha^*)
\partial_\beta b(x,\beta_0)
\dd\mu_{\beta_0}(x)(d_1-d_2)\end{align*}
Hence, from \eqref{eq7.899} and this, we obtain
\begin{align*}
\frac{1}{T\vartheta_\beta}\sum_{i=1}^{[n\tau^\beta_*]}\hat\zeta_i
-\frac{[n\tau^\beta_*]}{n}\frac{1}{T\vartheta_\beta}\sum_{i=1}^n\hat\zeta_i
\pto
\tau^\beta_*(1-\tau^\beta_*)
\int_{\mathbb R^d}
\partial_{\beta}b(x,\beta_0)^\TT A^{-1}(x,\alpha^*)
\partial_\beta b(x,\beta_0)
\dd\mu_{\beta_0}(x)(d_1-d_2)\neq0,
\end{align*}
which implies $\mathcal T_{2,n}^\beta\lto\infty$ 
from \eqref{eq7.890} and $T\Deb^2\lto\infty$,
that is,  
$P(\mathcal T_{2,n}^\beta>w_q(\epsilon))\lto1$.
\end{proof}


\section*{acknowledgements}
This work was partially supported by JST CREST Grant Number JPMJCR14D7 
and JSPS KAKENHI Grant Number JP17H01100.



\end{document}